\documentclass{article}

\usepackage{hyperref}

\usepackage{amsfonts, amsmath, amsthm}

\usepackage[margin=3cm]{geometry}

\usepackage{tikz}
\usetikzlibrary{calc, arrows, decorations.markings, decorations.pathmorphing, positioning, decorations.pathreplacing}
\makeatletter
\tikzset{nomorepostaction/.code={\let\tikz@postactions\pgfutil@empty}}
\makeatother

 \usepackage{fancyhdr}
\newtheorem{thm}{Theorem}[section]
\newtheorem{cor}[thm]{Corollary}
\newtheorem{lem}[thm]{Lemma}
\newtheorem{prop}[thm]{Proposition}

\newtheorem{defn}[thm]{Definition}

\newcommand{\To}{\longrightarrow}
\newcommand{\C}{\mathbb{C}}

\newcommand{\Z}{\mathbb{Z}}

\newcommand{\V}{\mathcal{V}}
\newcommand{\bV}{{\bf V}}
\newcommand{\D}{\mathcal{D}}
\newcommand{\Disc}{\mathbb{D}}
\newcommand{\R}{\mathbb{R}}
\newcommand{\s}{\mathfrak{s}}

\newcommand{\e}{\mathbf{e}}

\newcommand{\0}{{\bf 0}}
\newcommand{\1}{{\bf 1}}
\newcommand{\T}{\mathbb{T}}
\newcommand{\Tree}{\mathcal{T}}
\newcommand{\M}{\mathcal{M}}

\newcommand{\Q}{\mathbb{Q}}
\newcommand{\I}{\mathcal{I}}

\renewcommand{\P}{\mathcal{P}}

\DeclareMathOperator{\im}{Im}
\renewcommand{\Re}{\text{Re}\;}

\DeclareMathOperator{\Sym}{Sym}

\DeclareMathOperator{\ev}{ev}

\DeclareMathOperator{\Spin}{Spin}

\DeclareMathOperator{\DOS}{DOS}

\DeclareMathOperator{\Int}{Int}

\begin{document}

\title{Twisty itsy bitsy topological field theory}

\author{Daniel V. Mathews \\ School of Mathematical Sciences, Monash University}


\maketitle

\begin{abstract}
We extend the topological field theory (``itsy bitsy topological field theory"') of our previous work from mod-2 to twisted coefficients. This topological field theory is derived from sutured Floer homology but described purely in terms of surfaces with signed points on their boundary (occupied surfaces) and curves on those surfaces respecting signs (sutures). It has information-theoretic (``itsy'') and quantum-field-theoretic (``bitsy'') aspects. In the process we extend some results of sutured Floer homology, consider associated ribbon graph structures, and construct explicit admissible Heegaard decompositions.
\end{abstract}

\tableofcontents

\section{Introduction}

\subsection{Overview}

This paper is part of a series dealing with a certain type of 2-dimensional topological quantum field theory which we call \emph{sutured quadrangulated field theory} (SQFT). This paper extends that theory to a more general system of coefficients --- \emph{twisted coefficients} --- to define \emph{twisted SQFT}.

We defined SQFT in \cite{Me09Paper}, following previous work of ourselves and Honda--Kazez--Mati\'{c}. We facetiously called this TQFT ``itsy bitsy topological field theory'', as certain aspects resemble the ``its'' of information and the ``bits'' of particles, and is a toy model. It deals with certain 2-dimensional surfaces (with a little extra structure), which we called \emph{occupied surfaces}, and certain curves (\emph{sutures}) on those surfaces. It does not quite satisfy the standard axioms for a TQFT (e.g. \cite{Witten88}), but we describe it loosely as such.

The SQFT of \cite{Me09Paper} uses $\Z_2$ coefficients. Roughly, it associates graded $\Z_2$-vector spaces to occupied surfaces, and graded linear maps to certain continuous maps (\emph{decorated morphisms}) between occupied surfaces. ``Filling in'' a surface with sutures singles out a \emph{suture element} of the vector space. In these respects SQFT is something like a traditional TQFT; but SQFT has much further structure. Decomposing an occupied surface into \emph{squares} gives a tensor decomposition of the corresponding vector space; and on each square, two ``basic'' sets of sutures are possible, which can be considered as representing \emph{bits} of information. The simplest morphisms between surfaces are those which ``create'' and ``annihilate'' squares, and give vector space maps which look like \emph{creation and annihilation} operators of quantum field theory.

The definition of twisted SQFT generalises the coefficient ring from $\Z_2$ to the group ring of the first homology of the surface. Roughly, we now associate a graded $\Z[H_1(\Sigma)]$-module to an occupied surface $(\Sigma,V)$, and a graded module homomorphism to a decorated morphism between occupied surfaces. Thus, the coefficient ring now depends on the surface --- and is no longer a field. All of the structure of SQFT should generalise to twisted coefficients.

The motivation, and prototypical example, for SQFT derives from \emph{sutured Floer homology} (\emph{SFH}), an invariant of certain sutured 3-dimensional manifolds based on symplectic and contact geometry and holomorphic curves. This TQFT was first considered (in a related form) by Honda--Kazez--Mati\'{c} in \cite{HKM08}, and considered further (in no particular order) by Massot \cite{Massot09}, Ghiggini-Honda \cite{Ghiggini_Honda08}, Fink \cite{Fink12}, Tian \cite{Tian12_Uq, Tian13_UT} and the author \cite{MyThesis, Me09Paper, Me10_Sutured_TQFT, Me12_itsy_bitsy, Me11_torsion_tori, Mathews_Schoenfeld12_String}. We proved in \cite{Me09Paper} that an SQFT exists by showing that SFH presents an example.  Sutured Floer homology can be defined over various systems of coefficients; the previous work of \cite{Me09Paper} was based on SFH with $\Z_2$ coefficients, and this paper is based on SFH with twisted coefficients. Careful treatment has also recently been given to naturality properties, which are closely related to TQFT, in Floer homology theories, including Heegaard Floer \cite{Juhasz_Thurston12} and monopole and instanton Floer homology \cite{Baldwin_Sivek14}.

The main object of this paper is to define a twisted SQFT, present an argument why this definition is the natural one, and prove the existence of a twisted SQFT, as expressed in the following theorem. 
\begin{thm}
The sutured Floer homology of sutured 3-manifolds $(\Sigma \times S^1, F \times S^1)$, with twisted coefficients, contact elements, and maps on SFH induced by inclusions of surfaces, forms a twisted SQFT.
\end{thm}

Extending SQFT to twisted coefficients gives much more structure, but also presents various difficulties and diversions. For one thing, the TQFT itself is more complicated, and ambiguities up to units arise. For another, the computations in sutured Floer homology are more complicated. More practically, some of the existing literature in sutured Floer homology does not consider the twisted case, so that we must to extend some of the general theory of sutured Floer homology ourselves. And some considerations, like admissibility of Heegaard structures, lead us to consider auxiliary constructions such as \emph{spines} of quadrangulated surfaces, which are a type of ribbon graph we call \emph{tape graphs}. Tape graphs, so far as we know, are a new concept, and may be of interest in their own right. 

We detail the winding path of this paper in section \ref{sec:what_this_paper_does} below. In any case, because of the lengthy diversions required, in this paper we limit ourselves to consideration only of some very basic properties of twisted SQFT --- mainly, that it exists! Further properties will be considered in subsequent work.

In our view the difficulties in generalising to twisted coefficients is worth the effort. We expect that with twisted coefficients, connections between this topological quantum field theory and conjecturally underlying quantum group representations will be clarified. Just as Khovanov homology presents a categorification of the Jones polynomial, which is a $U_q \mathfrak{sl}(2)$ quantum invariant, Heegaard Floer homology presents a categorification of the Alexander polynomial, which is a $U_q \mathfrak{sl}(1|1)$ quantum invariant. Categorification of $U_q \mathfrak{sl}(1|1)$ has recently been explored by Tian \cite{Tian12_Uq, Tian13_UT}.

\subsection{What this paper does}
\label{sec:what_this_paper_does}

In generalising SQFT to twisted coefficients, several issues arise. As it turns out, this means that over half this paper deals with background, technicalities of sutured Floer homology, manipulations of Heegaard decompositions, and related considerations. We therefore take the time now to briefly explain these issues, and how we deal with them.

The first difficulty is that the most naive generalisation to twisted coefficients fails. An assignment of $\Z[H_1(\Sigma)]$-modules,  to occupied surfaces $(\Sigma,V)$, and module homomorphisms and suture elements strictly as in untwisted SQFT, does not exist. A version of this impossibility was known to Honda--Kazez--Mati\'{c} in \cite{HKM08} and explored in \cite{Me10_Sutured_TQFT}. We define two natural putative versions of twisted SQFT in section \ref{sec:ambiguity}, which we call \emph{optimistic} (definition \ref{defn:optimistic_twisted_SQFT}) and (after that fails) \emph{sanguine}, and obtain the following.
\begin{prop}[Necessity of ambiguity]
\label{prop:necessity_of_ambiguity}
No optimistic or sanguine twisted SQFT exists.
\end{prop}
Thus, we must resign ourselves to ambiguity in suture elements and module homomorphisms: they are naturally well-defined \emph{up to units} in the coefficient ring --- a situation which is familiar from quantum theory.

Having defined twisted SQFT, we can prove a few first properties. It turns out there is a natural notion of ``mod 2 reduction'', so that some properties are inherited from untwisted SQFT immediately. However, some statements which are obvious over $\Z_2$ become much more slippery, due to ambiguities (for instance, ``standard gluings of surfaces give identity morphisms of modules''). We may know, for example, that each basis element of a free module is mapped to itself up to units under a homomorphism --- but this does not imply that the homomorphism is the identity up to units! Nonetheless, we dealt with difficulties of a similar nature in previous work \cite{Me10_Sutured_TQFT}, and similar methods should be able to resolve these issues. But invoking those previous methods, which requires significant additional background \cite{Me09Paper, Me10_Sutured_TQFT}, is best left to the sequel; we instead focus on proving the existence of twisted SQFT via twisted SFH.

Proving that SFH presents a twisted SQFT leads to its own array of issues. Ozsv\'{a}th--Szab\'{o}'s definition of twisted coefficients in \cite{OS04Prop} came before sutured Floer homology; and when Juh\'{a}sz introduced SFH in \cite{Ju06} he did not cover the twisted case. As such, there are gaps in the literature. Although others have subsequently considered SFH with twisted coefficients (e.g. \cite{Massot09, Ghiggini_Honda08}), the theory has not been fully treated. We recall and clarify the underlying theory as needed for our purposes, following Ozsv\'{a}th--Szab\'{o} and Juh\'{a}sz. 

More particularly, we need to extend a theorem of Juh\'{a}sz on the effect of surface decompositions on sutured Floer homology to the case of twisted coefficients. Roughly the generalised version of Juh\'{a}sz' theorem is as follows; a precise statement is theorem \ref{thm:Juhasz_gluing}.
\begin{thm}
\label{thm:gluing_rough}
If $(M, \Gamma) \stackrel{U}{\rightsquigarrow} (M', \Gamma')$ be a sutured manifold decomposition along a sufficiently nice surface $U$, then $SFH(M', \Gamma'; \Z[H_2(M)])$ is a direct summand of $SFH(M, \Gamma; \Z[H_2(M)])$. Furthermore, explicit Heegaard decompositions, chain complexes, and isomorphisms between summands can be given. 
\end{thm}

When we then turn to the specific sutured 3-manifolds $(\Sigma \times S^1, V \times S^1)$ used to defined twisted SQFT, there are further specific issues. Computations for SFH with twisted coefficients for these manifolds have been made \cite{Massot09}, with explicit Heegaard decompositions and correct answers. However, at least so far as we understand, these Heegaard decompositions are sometimes inadmissible. We are therefore led to revisit this construction, and treat admissibility.

It turns out that these Heegaard decompositions fit nicely with quadrangulations of occupied surfaces. We are able to build Heegaard decompositions out of pieces we call \emph{Heegaard blocks}, one for each square of the quadrangulation. The following result is made precise in propositions \ref{prop:Heegaard_block_decomposition} and \ref{prop:Heegaard_block_decomposition_admissible}.
\begin{prop}
\label{prop:Heegaard_block_decomposition_rough}
Given a quadrangulation $Q$ of an occupied surface $(\Sigma,V)$, there is a naturally associated Heegaard decomposition $(S, \alpha, \beta)$ of $(\Sigma \times S^1, V \times S^1)$, where $S$ is built out of Heegaard blocks, one block for each square of $Q$. Each block contains one $\alpha$ and one $\beta$ curve. There is a well-defined isotopy of the $\beta$ curves which renders this Heegaard decomposition admissible.
\end{prop}

The periodic domains of these Heegaard decompositions are described conveniently in terms of a \emph{ribbon graph} associated to a quadrangulated occupied surface. It was mentioned in \cite[sections 4.5, 8.5]{Me12_itsy_bitsy} that quadrangulations can endow an occupied surface with a ribbon graph structure; we pursue these ideas in a slightly different direction to define the \emph{spine} of an occupied surface $(\Sigma,V)$, which is a ribbon graph, onto which $\Sigma$ deformation retracts. In fact, we do not quite have ribbon graphs in the traditional sense: ribbon graphs have a \emph{cyclic ordering} on edges at each vertex, but our spines have a \emph{total ordering} on edges at each vertex. We call such graphs \emph{tape graphs}.

Since we have not seen the concept of tape graphs before, we pause to consider some of the properties of tape graphs in general. These properties may be of interest in their own right. In any case they assist our analysis of periodic domains and admissibility. 

Once an admissible Heegaard decomposition is obtained, we compute SFH by successively applying the twisted version of Juh\'{a}sz' decomposition theorem, successively simplifying the Heegaard decomposition, and keeping track of admissibility via associated quadrangulations and tape/ribbon graphs. No nontrivial holomorphic curve ever contributes to the calculation of SFH. Previous calculations therefore come to the correct result, and the correct conclusion that the generators of SFH can be identified with generators of the chain complex, i.e. complete intersections of $\alpha$ and $\beta$ curves.

In sum, there are various hurdles of various types to be cleared, and various auxiliary considerations, in generalising SQFT to twisted coefficients. We must clarify why ambiguity enters the picture; we must extend a theorem about SFH, in the process revisiting some technicalities of SFH; we must clarify admissibility of previous constructions, and introduce some ribbon/tape graphs along the way. These various hurdles and diversions lead us in different directions, but all are necessary for a proper treatment of the main theorem.

As a result, computations in twisted SQFT, and proofs of some of the most basic properties of twisted SQFT, are deferred to the sequel. Nonetheless we establish some properties, including a bypass relation, and those inherited from the mod 2 case, in sections \ref{sec:first_properties}-- \ref{sec:inheritance}.

\subsection{Structure of this paper}

In section \ref{sec:background} we recall the definition of SQFT (the ``itsy bitsy topological field theory'' of the prequel) and related concepts. We then proceed to the twisted case. We first (section \ref{sec:algebraic_preliminaries}) discuss the types of maps which arise, between graded modules over different rings. Then we explain (section \ref{sec:ambiguity}) why ambiguity of suture elements and module homomorphisms cannot be avoided, proving the non-existence of ``optimistic'' and even ``sanguine'' versions of SQFT with twisted coefficients. This motivates the definition of twisted SQFT (section \ref{sec:definition}). We can quickly perform some computations in the simplest cases (section \ref{sec:first_properties}), including a version of the bypass relation. We show that twisted SQFT canonically reduces mod-2 to untwisted SQFT (section \ref{sec:reduction_mod_2}), so that results of \cite{Me12_itsy_bitsy} apply (section \ref{sec:inheritance}) and we obtain a few basic properties.

The rest of the paper is devoted to proving the existence of a twisted SQFT by performing computations in SFH. 

Section \ref{sec:twisted_SFH} deals with the technicalities and generalities of SFH with twisted coefficients. We begin by recalling the untwisted case (section \ref{sec:SFH_background}) and how it leads to untwisted SQFT (section \ref{sec:SQFT_from_SFH}), before turning to the twisted case (section \ref{sec:introducing_twisted_coeffs}). We revisit some of the technicalities on homology of the symmetric product (section \ref{sec:homology_symmetric_product}) and homotopy classes of Whitney discs (section \ref{sec:homotopy_classes_Whitney_discs}), adapting them to the sutured case --- these are essential for extending Juh\'{a}sz' theorem in section \ref{sec:decompositions_twisted}. We also recall the gluing theorem of Ghiggini--Honda with twisted coefficients (section \ref{sec:contact_gluing}).

Section \ref{sec:TSQFT_Heegaard} then deals with Heegaard decompositions of our specific 3-manifolds, and related constructions. We begin by discussing quadrangulations and spines of occupied surfaces (section \ref{sec:spines_of_occupied_surfaces}). We consider some generalities about ribbon/tape graphs, which are useful for subsequent discussion (section \ref{sec:tape_graphs}). We define Heegaard blocks, the building blocks of our decompositions, in section \ref{sec:Heegaard_blocks}, and then use them to build Heegaard decompositions in section \ref{sec:building_Heegaard}. We discuss the relationship between these Heegaard decompositions and spines in section \ref{sec:Heegaard_and_spines}, which allows us to analyse periodic domains in section \ref{sec:periodic_domains} and admissibility in section \ref{sec:admissibility}. Having achieved an admissible Heegaard decomposition, we show how to compute SFH by successive decomposition in section \ref{sec:computing_SFH}, and finally demonstrate how it forms an example of twisted SQFT in section \ref{sec:twisted_SQFT_from_SFH}.

\section{Twisted SQFT}

\subsection{Previous work and background}
\label{sec:background}

We begin by recalling some of the ``itsy bitsy topological field theory'' i.e. sutured quadrangulated field theory (SQFT), of \cite{Me12_itsy_bitsy}. This theory is all over $\Z_2$, and when we refer to SQFT without adjectives, we mean this ``untwisted'' mod 2 case. This is a very brief summary; for full details see \cite{Me12_itsy_bitsy}.

We do not consider twisted coefficients in this section; for now we merely recall previous work.

\begin{defn}
An \emph{occupied surface} is a pair $(\Sigma,V)$, where $\Sigma$ is a compact oriented surface (possibly disconnected), and $V \subset \partial \Sigma$ is a finite set of $\pm$-signed \emph{vertices} ($V = V_+ \cup V_-$), such that
\begin{enumerate}
\item each component of $\Sigma$ has nonempty boundary
\item each component of $\partial \Sigma$ contains vertices, and
\item along each boundary component, vertices alternate in sign.
\end{enumerate}
\end{defn}

Throughout this paper, a pair $(\Sigma,V)$ will always denote an occupied surface. The empty set is a trivial occupied surface. A disc with two vertices is called a \emph{vacuum}. A disc with four vertices is called a \emph{square}. We denote $|V| = 2N$, so there are $N$ vertices of either sign, $N = |V_+| = |V_-|$. An arc along $\partial \Sigma$ from one vertex to the next is called a \emph{boundary edge} and is naturally oriented from $V_-$ to $V_+$.

Occupied surfaces decompose nicely into squares. A \emph{decomposing arc} is a properly embedded arc in $\Sigma$ from a vertex to a vertex of opposite sign. An occupied surface $(\Sigma,V)$ decomposes into $I(\Sigma,V) = N - \chi(\Sigma)$ squares; there are many ways to so decompose, but the number of squares is always the same, and we call $I(\Sigma,V)$ the \emph{index} of $(\Sigma,V)$. (Conversely, the $I(\Sigma,V)$ squares can be glued into $(\Sigma,V)$.) We call such a decomposition a \emph{quadrangulation}. A quadrangulation can be described via the set of decomposing arcs (\emph{internal edges}) cutting $\Sigma$ into squares; or, equivalently, as a set of square subsurfaces of $(\Sigma,V)$. Any two quadrangulations are related by local adjustments called \emph{diagonal slides}, shown in figure \ref{fig:diagonal_slides}.

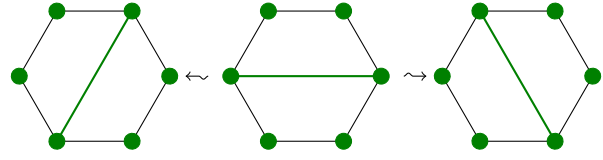
\begin{figure}
\begin{center}
\begin{tikzpicture}[
scale=1, 
boundary/.style={ultra thick}, 
decomposition/.style={thick, draw=green!50!black}, 
]

\foreach \x/\rot in {-80/60, 0/0, 80/-60}
{
\draw [xshift=\x, rotate=\rot] (0:1) -- (60:1) -- (120:1) -- (180:1) -- (240:1) -- (300:1) -- cycle;
\draw [xshift=\x, rotate=\rot, decomposition] (0:1) -- (180:1);

\foreach \angle in {0, 60, 120, 180, 240, 300}
\fill [green!50!black, draw=green!50!black, xshift=\x, rotate=\rot] (\angle:1) circle (3pt);
}

\draw [shorten >=1mm, -to, decorate, decoration={snake,amplitude=.4mm, segment length = 2mm, pre=moveto, pre length = 1mm, post length = 2mm}]
(1.2,0) -- (1.7,0);
\draw [shorten >=1mm, -to, decorate, decoration={snake,amplitude=.4mm, segment length = 2mm, pre=moveto, pre length = 1mm, post length = 2mm}] (-1.2,0) -- (-1.7,0);

\end{tikzpicture}
\caption{Diagonal slides.}
\label{fig:diagonal_slides}
\end{center}
\end{figure}

A \emph{morphism} $\phi: (\Sigma,V) \To (\Sigma',V')$ is, roughly (see \cite[sec. 3.1]{Me12_itsy_bitsy} for full details), a continuous map $\Sigma \To \Sigma'$ which embeds of the interior $\Int \Sigma$ of $\Sigma$ into $\Sigma'$, and which treats boundary edges as follows. Boundary edges of $\Sigma$ may be folded or glued under $\phi$; but if they intersect other at endpoints, they coincide. Signs of vertices must be respected. Any homeomorphism of $\Sigma$ preserving $V_+$ and $V_-$ is a morphism.
Occupied surfaces and morphisms form a category. The empty set is an initial object in this category.

In \cite{Me12_itsy_bitsy} we show any morphism is a composition of the following four types of morphisms, illustrated in figure \ref{fig:elementary_morphisms}.
\begin{enumerate}
\item 
A \emph{creation} takes an occupied surface $(\Sigma,V)$ and places a disjoint square $(\Sigma^\square, V^\square)$ (the \emph{created square}) next to it to obtain $(\Sigma',V') = (\Sigma,V) \sqcup (\Sigma^\square,V^\square)$.
\item
A \emph{standard gluing} glues two non-consecutive boundary edges of $(\Sigma,V)$ respecting orientations, to obtain $(\Sigma',V')$ with $\chi$ decreased by $1$ and $N$ decreased by $1$, so $I(\Sigma',V') = I(\Sigma,V)$.
\item
A \emph{fold} glues two consecutive boundary edges of $(\Sigma,V)$, respecting orientations, which do not comprise an entire boundary component of $\Sigma$. The resulting $(\Sigma',V')$ has $\Sigma'$ homeomorphic to $\Sigma$ but $|V'| = |V| - 2$.
\item
A \emph{zip} glues together two consecutive boundary edges of $(\Sigma,V)$ which comprise an entire boundary component. This decreases $N$ by $1$ and increases $\chi$ by $1$.
\end{enumerate}

\begin{figure}
\begin{center}

\begin{tabular}{c|c}
\begin{tabular}{c}
Creation:
\\
\def\svgwidth{180pt}
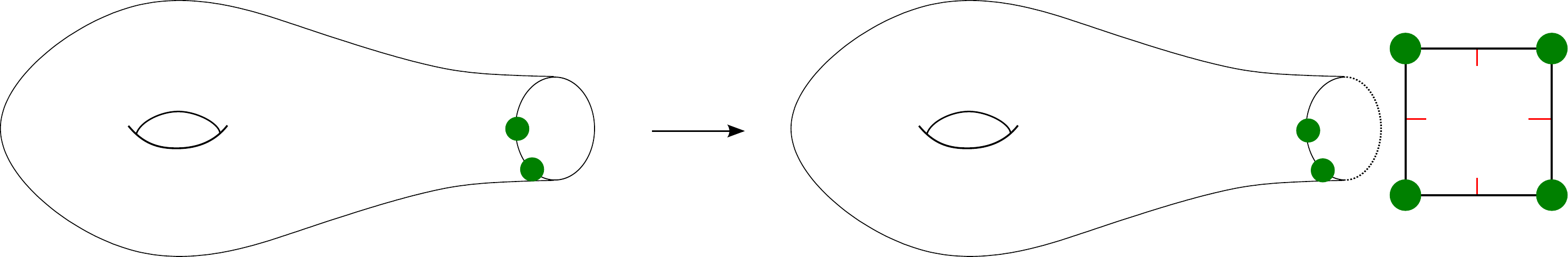
\end{tabular}
&
\begin{tabular}{c}
Standard gluing:
\\
\def\svgwidth{180pt}
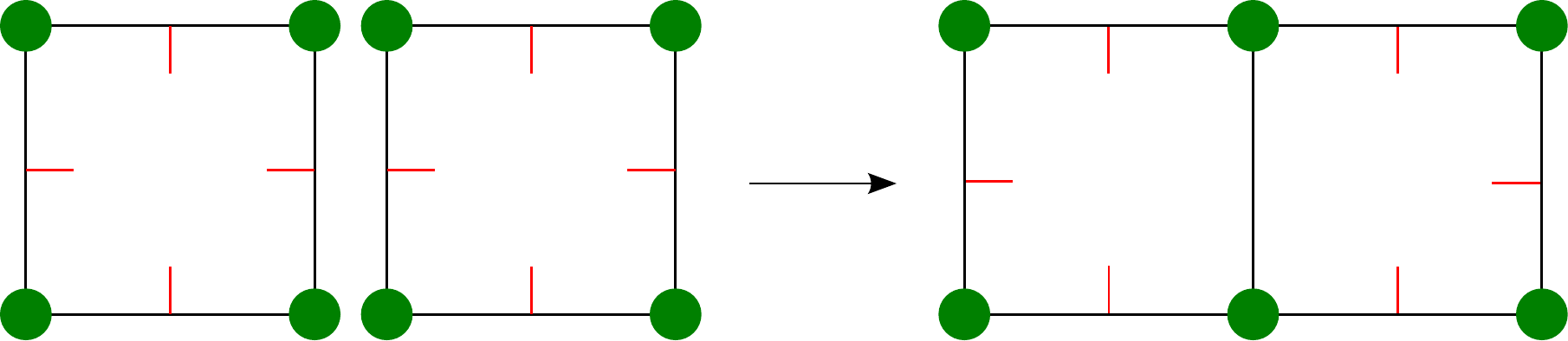
\end{tabular}
\\
\hline
\begin{tabular}{c}
Fold:
\\
\def\svgwidth{160pt}
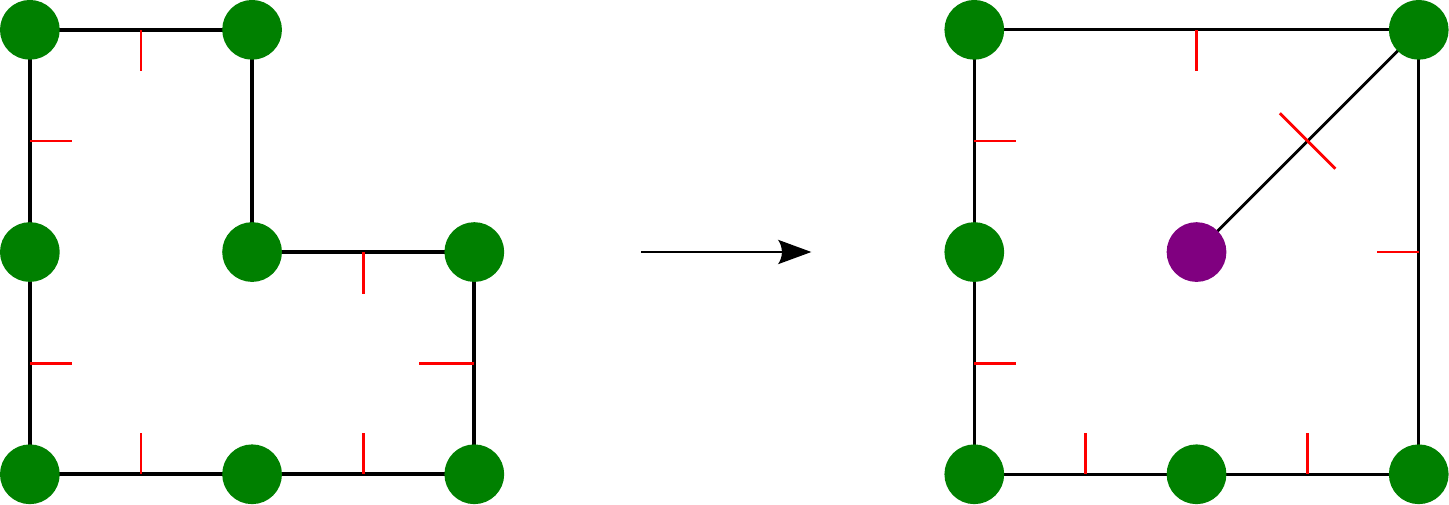
\end{tabular}
&
\begin{tabular}{c}
Zip:
\\
\def\svgwidth{160pt}
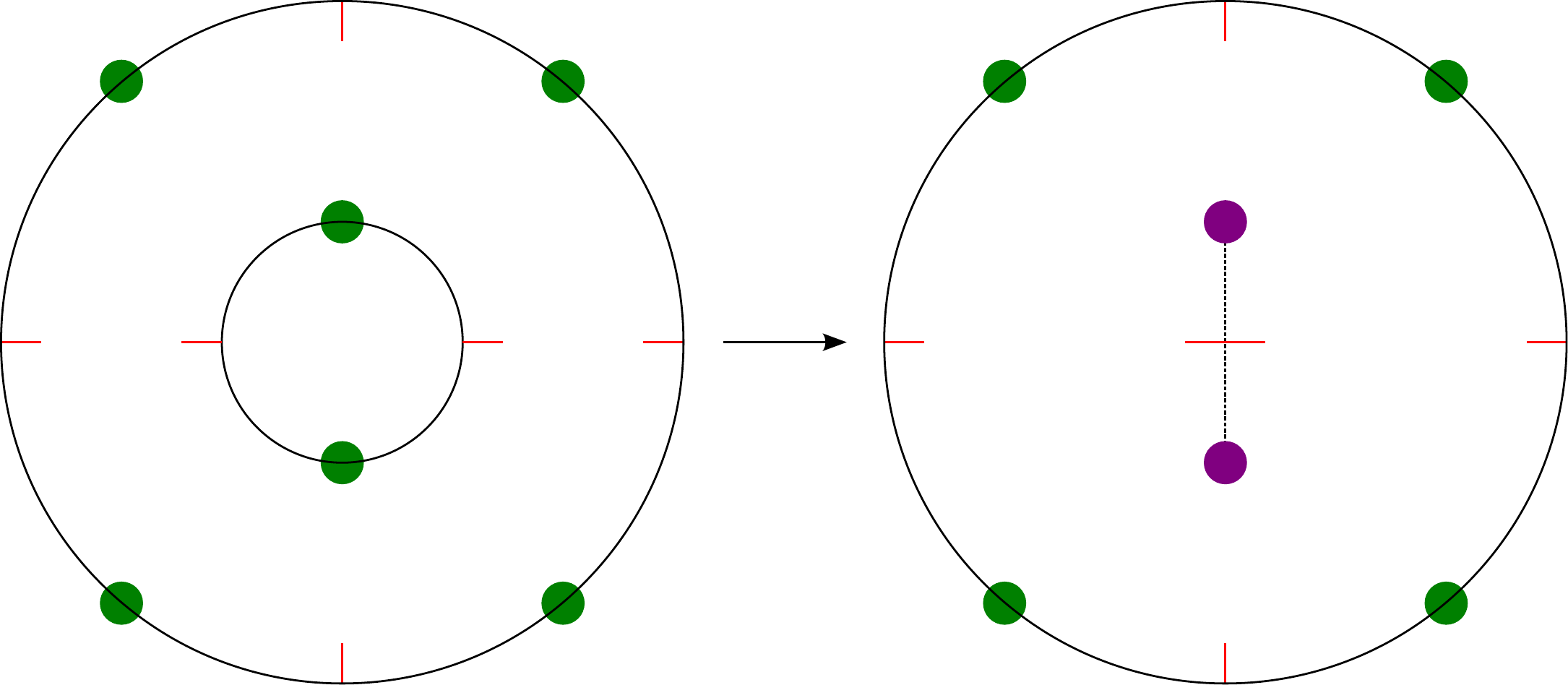
\end{tabular}
\end{tabular}

\caption{Elementary morphisms.}
\label{fig:elementary_morphisms}
\end{center}
\end{figure}

The image of a morphism is a well-defined occupied surface. However, the image occupied surface may have fewer vertices; a vertex of an occupied surface need not be sent to a vertex of the image occupied surface. For instance, this occurs with a fold or zip. In this case we say a vertex has been \emph{swallowed}. The \emph{complement} of the image of a morphism is also a well-defined occupied surface.

Given a quadrangulated occupied surface, after applying a standard gluing, or creation, we naturally obtain a quadrangulation on the image occupied surface. Even after applying a fold or zip, we obtain a \emph{slack quadrangulation} of the resulting surface, which is essentially a quadrangulation with vertices permitted in the interior of $\Sigma$. Isotoping these interior vertices to boundary vertices we perform \emph{slack square collapse} and obtain a bona fide quadrangulation.

Occupied surfaces form a natural setting for \emph{sutures}. For our purposes (slightly nonstandard), a \emph{sutured surface} $(\Sigma,\Gamma)$ is a compact oriented surface $\Sigma$, and a properly embedded oriented 1-submanifold $\Gamma \subset \Sigma$, such that:
\begin{enumerate}
\item $\Sigma \backslash \Gamma = R_+ \cup R_-$, where $R_\pm$ are surfaces oriented as $\pm \Sigma$,
\item $\overline{\partial R_\pm \backslash \partial \Sigma} = \Gamma$ as oriented 1-manifolds, and
\item for every component $C$ of $\partial \Sigma$, $C \cap \Gamma \neq \emptyset$.
\end{enumerate}

The whole of $\Gamma$ is called a \emph{set of sutures} and each component a \emph{suture}. We only ever consider sutures up to isotopy relative to endpoints. Roughly speaking, $\Gamma$ splits $\Sigma$ into positive regions $R_+$ and negative regions $R_-$; crossing $\Gamma$ switches the sign. The \emph{Euler class} of a set of sutures $\Gamma$ is $e(\Gamma) = \chi(R_+) - \chi(R_-)$. 

In \cite{Me12_itsy_bitsy} we define a \emph{sutured background surface} $(\Sigma,F)$ to be the boundary data of a sutured surface: it consists of the surface $\Sigma$ together with a collection of (signed) points $F$ on $\partial \Sigma$, which cut $\partial \Sigma$ into positive and negative arcs. Such a $(\Sigma,F)$ naturally corresponds to an occupied surface $(\Sigma,V)$ by taking one positive (resp. negative)  vertex in each positive (resp. negative) arc of $\partial \Sigma \backslash F$. The data of a sutured background surface and an occupied surface are equivalent; an occupied surface also provides natural boundary conditions for sutures. In our diagrams we draw vertices in green and sutures in red. We may speak of a set of sutures $\Gamma$ on an occupied surface $(\Sigma,V)$.

Sutures are \emph{trivial} if they contain a contractible closed curve. Sutures are \emph{confining} if there is a component of $\Sigma \backslash \Gamma$ which does not intersect $\partial \Sigma$ (a ``confined component''). Trivial sutures are confining. For nontrivial sutures $\Gamma$ on occupied $(\Sigma,V)$ we always have  $|e(\Gamma)| \leq I(\Sigma,V)$, and $e(\Gamma) \cong I(\Sigma,V)$ mod $2$.

There is a natural surgery on sutures called \emph{bypass surgery}. Bypass surgery is performed along an \emph{attaching arc} $c$, which is an embedded arc in $\Sigma$ intersecting $\Gamma$ at its endpoints and precisely one interior point. In a neighbourhood of $c$, sutures look as in figure \ref{fig:bypass_surgery}, and we may make the two surgeries shown. Sets of sutures related by bypass surgery naturally come in triples and we refer to such triples as \emph{bypass triples}. Bypass surgery preserves Euler class.

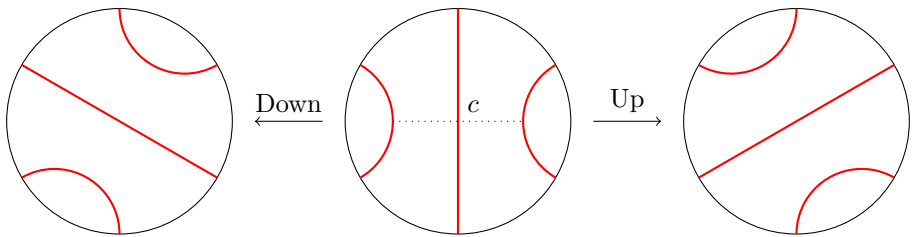
\begin{figure}
\begin{center}

\begin{tikzpicture}[
scale=1.5, 
suture/.style={thick, draw=red}]

\draw (3,0) circle (1 cm); 	
\draw (0,0) circle (1 cm);
\draw (-3,0) circle (1 cm);

\draw [suture] (30:1) arc (120:240:0.57735);
\draw [suture] (0,1) -- (0,-1);
\draw [suture] (210:1) arc (-60:60:0.57735);
\draw [dotted] (-0.57735,0) -- node [above right] {$c$} (0.57735,0);

\draw [suture] (3,0) ++ (150:1) arc (-120:0:0.57735);
\draw [suture] (3,0) ++ (30:1) -- ($ (3,0) + (210:1) $);
\draw [suture] (3,0) ++ (-30:1) arc (60:180:0.57735);

\draw [suture] (-3,1) arc  (180:300:0.57735);
\draw [suture] ($ (-3,0) + (-30:1) $) -- ($ (-3,0) +  (150:1) $);
\draw [suture] (-3,-1) arc (0:120:0.57735);

\draw[->] (-1.2,0) -- node [above] {Down} (-1.8,0);
\draw[->] (1.2,0) -- node [above] {Up} (1.8,0);

\end{tikzpicture}

\caption{Bypass surgery.}
\label{fig:bypass_surgery}
\end{center}
\end{figure}

Sutures are nicely behaved under various cutting and gluing operations. Gluing boundary edges of a sutured occupied surface $(\Sigma,\Gamma,V)$ (a surjective morphism) gives another sutured occupied surface. Performing a standard gluing on a sutured occupied surface results in sutures $\Gamma'$ with $e(\Gamma') = e(\Gamma)$. A decomposing arc $a$ on $(\Sigma,V)$ which is transverse to $\Gamma$ must intersect $\Gamma$ an odd number of times; if $|a \cap \Gamma| = 1$, then cutting along $a$ gives another sutured occupied surface with the same Euler class.

Sutures also play nicely with quadrangulations. Let $Q$ be a quadrangulation of $(\Sigma,V)$. A set of sutures $\Gamma$ is called \emph{basic} for $Q$ if $\Gamma$ is nontrivial and intersects each internal edge of $Q$ in a single point; equivalently, if on every square of $Q$ the sutures are of one of the two ``basic'' types $\Gamma_+, \Gamma_-$ shown in figure \ref{fig:sutured_squares}. As there are $I(\Sigma,V)$ squares in a quadrangulation of $(\Sigma,V)$, there are precisely $2^{I(\Sigma,V)}$ basic sets of sutures.

\begin{figure}
\begin{center}

\begin{tikzpicture}[
scale=2, 
suture/.style={thick, draw=red}, 
]

\coordinate [label = above left:{$-$}] (1tl) at (0,1);
\coordinate [label = above right:{$+$}] (1tr) at (1,1);
\coordinate [label = below left:{$+$}] (1bl) at (0,0);
\coordinate [label = below right:{$-$}] (1br) at (1,0);

\draw (1bl) -- (1br) -- (1tr) -- (1tl) -- cycle;
\draw [suture] (0.5,0) to [bend left=45] (1,0.5);
\draw [suture] (0,0.5) to [bend right=45] (0.5,1);
\draw (0.2,0.8) node {$-$};
\draw (0.5,0.5) node {$+$};
\draw (0.8,0.2) node {$-$};
\draw (0.5,-0.5) node {$c(\Gamma_-) = \0$};

\coordinate [label = above left:{$-$}] (2tl) at (3,1);
\coordinate [label = above right:{$+$}] (2tr) at (4,1);
\coordinate [label = below left:{$+$}] (2bl) at (3,0);
\coordinate [label = below right:{$-$}] (2br) at (4,0);

\draw (2bl) -- (2br) -- (2tr) -- (2tl) -- cycle;
\draw [suture] (3.5,1) to [bend right=45] (4,0.5);
\draw [suture] (3.5,0) to [bend right=45] (3,0.5);
\draw (3.8,0.8) node {$+$};
\draw (3.5,0.5) node {$-$};
\draw (3.2,0.2) node {$+$};
\draw (3.5,-0.5) node {$c(\Gamma_+) = \1$};

\foreach \point in {1bl, 1br, 1tl, 1tr, 2bl, 2br, 2tl, 2tr}
\fill [green!50!black] (\point) circle (2pt);

\end{tikzpicture}

\caption{Basic sets of sutures on the occupied square.}
\label{fig:sutured_squares}
\end{center}
\end{figure}
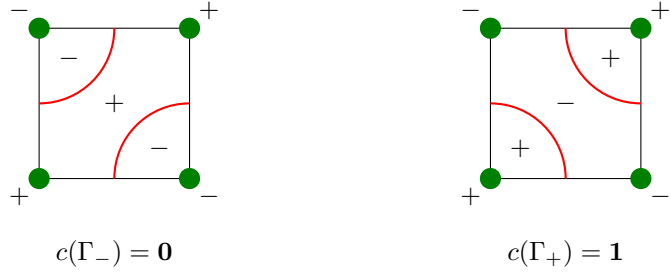

Basic sutures are always nonconfining. Conversely, given nonconfining sutures $\Gamma$ on occupied $(\Sigma,V)$, there exists a quadrangulation for which $\Gamma$ is basic. Any nontrivial set of sutures on a quadrangulated surface can be transformed to a basic set of sutures via bypass surgeries which reduce the number of intersections with decomposing arcs, and such that the set of sutures at each stage remains nontrivial.

The notion of morphism extends to \emph{decorated morphisms}, which send sutures to sutures.
\begin{defn}
A \emph{decorated morphism} is a pair $(\phi, \Gamma_c)$, where $\phi$ is a morphism, and $\Gamma_c$ is a set of sutures on the complement of the image of $\phi$.
\end{defn}
Occupied surfaces and decorated morphisms form a category, $\mathcal{DOS}$. 

Any surjective morphism is also a decorated morphism, as the complement of the image of $\phi$ is empty. Thus standard gluings, folds, and zips are also decorated morphisms. A creation morphism, with basic sutures on the created square, is a decorated morphism. See figure \ref{fig:decorated_creation}. Any decorated morphism is a composition of decorated creations, standard gluings, folds and zips.

\begin{figure}
\begin{center}
\includegraphics[scale=0.3]{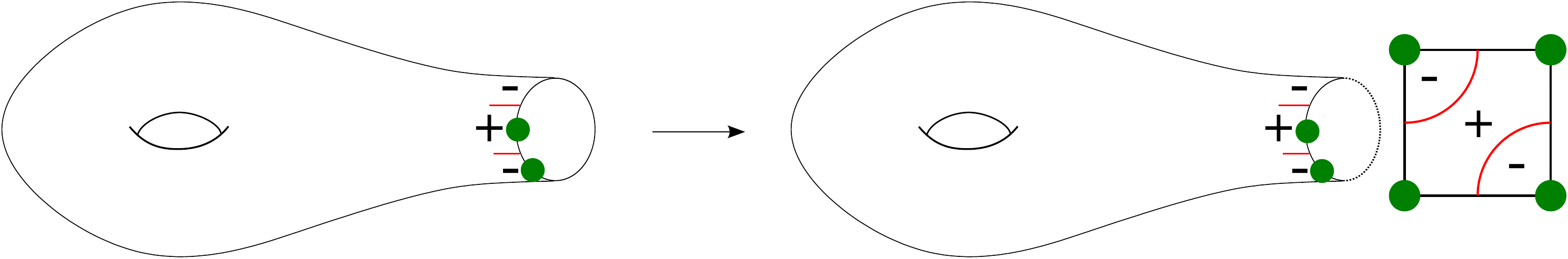}
\caption{A decorated creation.}
\label{fig:decorated_creation}
\end{center}
\end{figure}

An \emph{isotopy of decorated morphisms} is a family of decorated morphisms $(\phi_t, \Gamma_t)$, which varies smoothly with $t$ or admits singularities in which vacua are removed or created, as in figure \ref{fig:decorated_surface_isotopy_singularity}. Decorated morphisms related by such an isotopy are \emph{decorated-isotopic}. Any decorated morphism $(\Sigma,V) \To (\Sigma',V')$ is decorated-isotopic to one with image in the interior $\Int \Sigma'$ of $\Sigma'$.

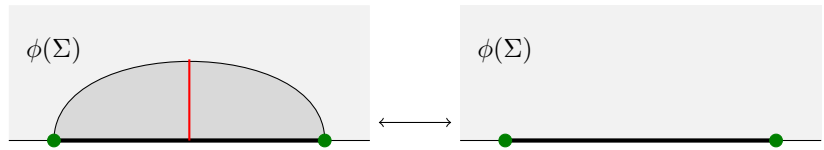
\begin{figure}
\begin{center}

\begin{tikzpicture}[
scale=1.2, 
boundary/.style={ultra thick}, 
vertex/.style={draw=green!50!black, fill=green!50!black},
suture/.style={thick, draw=red}
]

{
\fill [gray!10] (-0.5,0) -- (0,0) to [bend left=90] (3,0) -- (3.5,0) -- (3.5,1.5) -- (-0.5,1.5) -- cycle;
\fill [gray!30] (0,0) to [bend left=90] (3,0) -- cycle;

\draw (-0.5,0) -- (0,0) to [bend left=90] (3,0) -- (3.5,0);
\draw [boundary] (0,0) -- (3,0);
\draw [suture] (1.5,0.9) -- (1.5,0);
\draw (0,1) node {$\phi(\Sigma)$};

\fill [gray!10] (4.5,0) -- (8.5,0) -- (8.5,1.5) -- (4.5,1.5) -- cycle;

\draw (4.5,0) -- (5,0) (8,0) -- (8.5,0);
\draw [boundary] (5,0) -- (8,0);
\draw (5,1) node {$\phi(\Sigma)$};

\draw [<->] (3.6,0.2) -- (4.4,0.2);
}

\fill [vertex] (0,0) circle (2pt);
\fill [vertex] (3,0) circle (2pt);
\fill [vertex] (5,0) circle (2pt);
\fill [vertex] (8,0) circle (2pt);

\end{tikzpicture}

\caption{Removing or creating vacua in decorated isotopy.}
\label{fig:decorated_surface_isotopy_singularity}
\end{center}
\end{figure}

A decorated morphism $(\phi, \Gamma_c)$ adds a well-defined number to Euler class, as the following lemma shows; we call this the \emph{degree} of $(\phi, \Gamma_c)$. This point was not made in \cite{Me12_itsy_bitsy}, so we prove it now.
\begin{lem}
\label{lem:decorated_morphism_degree}
Let $(\phi, \Gamma_c): (\Sigma,V) \To (\Sigma',V')$ be a decorated morphism. There exists an integer $\deg(\phi, \Gamma_c)$ such that for any set of sutures $\Gamma$ on $(\Sigma,V)$,
\[
e(\Gamma \cup \Gamma_c) = e(\Gamma) + \deg(\phi, \Gamma_c).
\]
Moreover, if $\phi(\Sigma) \subset \Int \Sigma'$ then $\deg(\phi, \Gamma_c) = e(\Gamma_c)$. 

Further, degree is additive under composition: for decorated morphisms $(\Sigma,V) \stackrel{(\phi,\Gamma_c)}{\To} (\Sigma',V') \stackrel{(\phi',\Gamma'_c)}{\To} (\Sigma',V'')$ we have
\[
\deg(\phi' \circ \phi, \Gamma_c \cup \Gamma'_c) = \deg(\phi, \Gamma_c) + \deg(\phi', \Gamma'_c).
\]
\end{lem}

In writing $e(\Gamma_c)$ we consider $\Gamma_c$ as a set of sutures on the complementary occupied surface $(\Sigma^c, V^c)$.

\begin{proof}
If $\phi(\Sigma) \subset \Int (\Sigma')$, then $(\Sigma',V')$ splits into $(\Sigma,V)$ and $(\Sigma^c,V^c)$ by cutting along a finite collection of disjoint simple closed curves, each of which intersect $\Gamma \cup \Gamma_c$ in an even number of points. So by \cite[lem. 5.7]{Me12_itsy_bitsy}, $e(\Gamma \cup \Gamma_c) = e(\Gamma) + e(\Gamma_c)$; hence $\deg(\phi,\Gamma_c)$ exists and equals $e(\Gamma_c)$.

A general $(\phi, \Gamma_c)$ is decorated-isotopic to a $(\phi', \Gamma'_c)$ with image in $\Int \Sigma'$. For any sutures $\Gamma$ on $(\Sigma,V)$, $\Gamma \cup \Gamma_c$ and $\Gamma \cup \Gamma'_c$ are isotopic, so $e(\Gamma \cup \Gamma_c) = e(\Gamma \cup \Gamma'_c) = e(\Gamma) + e(\Gamma'_c)$, giving a well-defined degree $\deg(\phi,\Gamma_c) = e(\Gamma'_c)$.

For composable morphisms as above we have, for any sutures $\Gamma$ on $(\Sigma,V)$,
\begin{align*}
\deg(\phi' \circ \phi, \Gamma_c \cup \Gamma'_c) &= e(\Gamma \cup \Gamma_c \cup \Gamma'_c) - e(\Gamma) \\
&= e(\Gamma \cup \Gamma_c \cup \Gamma') - e(\Gamma \cup \Gamma_c) + e(\Gamma \cup \Gamma_c) - e(\Gamma) \\
&= \deg(\phi', \Gamma'_c) + \deg(\phi, \Gamma_c).
\end{align*}
\end{proof}

The SQFT of \cite{Me12_itsy_bitsy} is defined as follows.
\begin{defn}
\label{def:SQFT}
A \emph{sutured quadrangulated field theory} is a pair $(\D, c)$ where
\begin{enumerate}
\item
$\D$ is a functor $\DOS \To \Z_2 \mathcal{VS}$ from the decorated occupied surface category to the category of $\Z_2$ vector spaces, associating a $\Z_2$-vector space $\V(\Sigma,V)$ to an occupied surface $(\Sigma,V)$, and a linear map $\mathcal{D}_{(\phi,\Gamma_c)} \; : \; \V(\Sigma,V) \To \V(\Sigma',V')$ to a decorated morphism $(\phi, \Gamma_c)$,
\item
$c$ assigns to each (isotopy class of) sutures $\Gamma$ on $(\Sigma,V)$ an element $c(\Gamma) \in \V(\Sigma,V)$, (a \emph{suture element}),
\end{enumerate}
satisfying the following conditions.
\begin{enumerate}
\item
Quadrangulations give tensor decompositions. An (isotopy class of) quadrangulation of $(\Sigma,V)$ with squares $(\Sigma_i^\square, V_i^\square)$ gives
\[
\V(\Sigma,V) \cong \bigotimes_i \V(\Sigma_i^\square, V_i^\square).
\]
(This includes the ``null quadrangulation'' on any vacuum component $(\Sigma^\emptyset, V^\emptyset)$, and $\V(\Sigma^\emptyset, V^\emptyset) \cong \Z_2$.) If $\Gamma$ is a basic set of sutures restricting to $\Gamma_i$ on $(\Sigma_i^\square, V_i^\square)$, then under this isomorphism $c(\Gamma)$ corresponds to $\bigotimes_i c(\Gamma_i)$.
\item
Suture elements are respected: for any sutures $\Gamma$ on $(\Sigma,V)$,
\[
\mathcal{D}_{\phi,\Gamma_c}(c(\Gamma)) = c(\Gamma \cup \Gamma_c).
\]
\item
Basic sutures give basis elements: $c(\Gamma_+) = \1$ and $c(\Gamma_-) = \0$ form a basis for the vector space $\V(\Sigma^\square, V^\square)$ of a square $(\Sigma^\square, V^\square)$. In particular, $\V(\Sigma^\square, V^\square)$ is 2-dimensional.
\item
Euler class gives a grading. On $\V(\Sigma^\square, V^\square)$ we set the grading $e(\0) = -1$, $e(\1) = 1$. This induces a grading on any $\V(\Sigma,V)$ via a quadrangulation and tensor decomposition. For any sutures $\Gamma$, $c(\Gamma)$ has grading $e(\Gamma)$.
\end{enumerate}
\end{defn}

It is argued in \cite{Me12_itsy_bitsy} that the definition of SQFT is a natural extension of the standard axioms for a topological quantum field theory (e.g. \cite{Witten88}).

It follows from conditions (i) and (iii) that the 2-dimensional vector space of a square, with basis $\0$ and $\1$, is a fundamental algebraic building block for SQFT: we can write it as $\bV$, so
\[
\bV = \V(\Sigma^\square, V^\square) = \Z_2 \0 \oplus \Z_2 \1.
\]
Note that $\0 \neq 0$ and $\1 \neq 1$. We use this notation for convenience, and to emphasise connections with information theory. A quadrangulation on a $(\Sigma,V)$ gives an isomorphism $\V(\Sigma,V) \cong \bV^{\otimes n}$, where $n = I(\Sigma,V)$ and a basis is given by tensor products of $n$ $\0$'s and $\1$'s, i.e. $\e_1 \otimes \cdots \otimes \e_n$ where each $\e_n = \0$ or $\1$. To avoid confusion, we will adopt the notation in this paper
\[
| \e_1 \e_2 \cdots \e_n \rangle = \e_1 \otimes \cdots \otimes \e_n
\]
so that $\V(\Sigma,V) \cong \bV^{\otimes n}$ has basis $| \0 \cdots \0 \rangle$, $| \0 \cdots \1 \rangle$, $\ldots$, $|\1 \cdots \1 \rangle$.

In \cite{Me12_itsy_bitsy} we prove various properties of SQFT. For instance, for a disjoint union $(\Sigma_1, V_1) \sqcup (\Sigma_2, V_2) = (\Sigma_1 \sqcup \Sigma_2, V_1 \sqcup V_2)$ we have $\V(\Sigma_1 \sqcup \Sigma_2, V_1 \sqcup V_2) = \V(\Sigma_1, V_1) \otimes \V(\Sigma_2, V_2)$, and if $\Gamma_1, \Gamma_2$ are sutures on $(\Sigma_1, V_1)$, $(\Sigma_2, V_2)$ respectively, then $c(\Gamma_1 \sqcup \Gamma_2) = c(\Gamma_1) \otimes c(\Gamma_2)$.

The vacuum $(\Sigma^\emptyset, V^\emptyset)$ has $\V(\Sigma^\emptyset, V^\emptyset) \cong \Z_2$ and for the vacuum sutures $\Gamma^\emptyset$ we have $c(\Gamma^\emptyset) = 1$.

Confining sutures $\Gamma$ give $c(\Gamma) = 0$. (This includes trivial sutures.) 

For a bypass triple $\Gamma, \Gamma', \Gamma''$, their suture elements satisfy $c(\Gamma) + c(\Gamma') + c(\Gamma'') = 0$; we call this the \emph{bypass relation}. 

We can also describe the linear maps $\D_{\phi, \Gamma_c}$ associated to various elementary decorated morphisms $(\phi, \Gamma_c)$:
\begin{enumerate}
\item
An identity morphism gives the identity isomorphism of vector spaces.
\item
A standard gluing gives an isomorphism of vector spaces: a standard gluing preserves a quadrangulation, and the isomorphism is the identity on each corresponding tensor factor.
\item
A positive decorated creation $(\phi, \Gamma^+)$, with basic positive sutures on the created square, gives a \emph{positive digital creation operator}, which is a map $\bV^{\otimes n} \To \bV^{\otimes (n+1)}$ sending $x \mapsto x \otimes \1$. A negative decorated creation $(\phi, \Gamma^-)$, with basic negative sutures on the created square, gives a \emph{negative digital creation operator}, which is similar except it sends $x \mapsto x \otimes \0$.
\item
When, after a morphism, a quadrangulation becomes slack, each slack square collapse gives a \emph{general digital annihilation operator}.
\end{enumerate}

A \emph{digital annihilation operator} is a map $\bV^{\otimes (n+1)} \To \bV^{\otimes n}$. A $\1$-annihilation $a_\1$ deletes a $\1$ in a particular tensor factor, if possible; else it deletes the $\0$ there, and compensates by replacing a $\1$ with a $\0$ elsewhere, summing over the various possibilities.
\[
\begin{array}{ccccc}
a_\1 &:& \0 \otimes x_1 \otimes \cdots \otimes x_n & \mapsto & \sum_{x_i = \1} x_1 \otimes \cdots \otimes x_{i-1} \otimes \0 \otimes x_{i+1} \otimes \cdots \otimes x_n \\
&:& \1 \otimes x_1 \otimes \cdots \otimes x_n & \mapsto & x_1 \otimes \cdots \otimes x_n.
\end{array}
\]
A $\0$-annihilation operator behaves similar with the roles of $\0$ and $\1$ reversed. A \emph{general digital annihilation operator} is a map of the form $a_\0 \otimes 1^{\otimes m}$ or $a_\1 \otimes 1^{\otimes m}$, i.e. a tensor product with some tensor power of the identity, leading to a map $\bV^{\otimes (n+1+m)} \To \bV^{\otimes (n+m)}$. 

The main result of \cite{Me12_itsy_bitsy} is that any decorated morphism $(\phi, \Gamma_c)$ gives a graded module homomorphism $\D_{\phi, \Gamma^c}$ which is a composition of digital creation and general digital annihilation operators.

Also, an SQFT exists! An example is given by sutured Floer homology --- as discussed in section \ref{sec:SQFT_from_SFH}.

\subsection{Algebraic preliminaries}
\label{sec:algebraic_preliminaries}

We now turn to extending SQFT to to twisted coefficients. The first thing we do is establish some basic algebraic preliminaries: defining the types of algebraic objects, and maps between them.

The ``natural'' coefficients for a twisted version of SQFT, for an occupied surface $(\Sigma,V)$, lie in $\Z[H_1(\Sigma)]$, the group ring of first homology of $\Sigma$. (By default, all homology considered of manifolds is singular homology with integer coefficients.) This is since twisted coefficients for SFH of a balanced sutured 3-manifold $(M, \Gamma)$ lie in $\Z[H_2(M)]$, and $H_2(\Sigma \times S^1) = H_1(\Sigma)$. See section \ref{sec:twisted_coefficient_SFH} below for further details.

We use exponential notation for group rings. Thus if $A \in H_1(\Sigma)$ then we write $e^A$ for the corresponding element in $\Z[H_1(\Sigma)]$. Given two elements $A, B \in H_2(M)$ their sum is $A+B \in H_2(M)$ and we have $e^A \cdot e^B = e^{A+B} \in \Z[H_2(M)]$. We always consider rings to be commutative with unit.

For any occupied surface $(\Sigma,V)$, $H_1(\Sigma)$ is a free abelian group of some rank $n$. Taking a basis $A_1, \ldots, A_m$, a general element of $\Z[H_1(\Sigma)]$ can be written as
\[
c_1 e^{A_1} + c_2 e^{A_2} + \cdots + c_m e^{A_m}
\]
where $c_i \in \Z$ and $A_i \in H_1(\Sigma)$. Thus $\Z[H_1(\Sigma)]$ is the ring of Laurent polynomials in the $n$ variables $e^{A_1}, \ldots, e^{A_m}$, with we can also write as $q_1, \ldots, q_m$. The units are just the elements of the form $\pm e^A$, where $A \in H_1(\Sigma)$; equivalently, the units are (up to sign) the monomials.

(More generally for any free group $G$, and in fact for any indicable group, the group ring $\Z[G]$ has units consisting of precisely of the elements of the form $\pm e^g$ where $g \in G$. See e.g. \cite{Higman40}.)

In twisted SQFT, the object $\V(\Sigma,V)$ associated to an occupied surface $(\Sigma,V)$ has coefficients in $\Z[H_1(\Sigma)]$, i.e. is a $\Z[H_1(\Sigma)]$-module. And the map $\D_{\phi, \Gamma_c} : \V(\Sigma,V) \To \V(\Sigma',V')$ associated to a decorated morphism $(\phi, \Gamma_c) : (\Sigma,V) \To (\Sigma',V')$ is a map from a $\Z[H_1(\Sigma)]$-module to a $\Z[H_1(\Sigma')]$-module. As $\phi$ is (among other things) a continuous map $\Sigma \To \Sigma'$, it also induces a ring homomorphism $\phi_*: \Z[H_1(\Sigma)] \To \Z[H_1(\Sigma')]$.

The most natural type of map between modules over different rings is one which preserves addition and is equivariant with respect to a ring homomorphism, as in the following.
\begin{defn}
Let $M_i, M_j$ be modules over rings $R_i, R_j$ respectively. A \emph{module homomorphism} $f: M_i \to M_j$ over a ring homomorphism $\bar{f}: R_i \To R_j$ is an abelian group homomorphism which is equivariant with respect to $\bar{f}$. That is, for all $r \in R_i$ and $m \in M_i$,
\[
f(rm) = \bar{f}(r) \; f(m).
\]
\end{defn}

It is easy to check that the composition of two module homomorphisms $M_1 \stackrel{f}{\To} M_2 \stackrel{g}{\To} M_3$ over ring homomorphisms $R_1 \stackrel{\bar{f}}{\To} R_2 \stackrel{\bar{g}}{\To} R_3$ is a module homomorphism $M_1 \To M_3$ over $\bar{g} \circ \bar{f}$. For any module there is also an identity module homomorphism over the identity ring homomorphism. 

Since SQFT requires gradings on the $\V(\Sigma,V)$, we consider graded modules, in the following sense. (This definition could easily be generalised in several directions, but this is all we need.)
\begin{defn}
A graded  $R$-module is an $R$-module $M$ which splits as a direct sum of $R$-modules $M_e$, indexed by $e \in \Z$.
\[
M = \bigoplus_{e \in \Z} M_e,
\]
where only finitely many of the $M_e$ are nonzero.
\end{defn}
We say elements of $M_e$ have \emph{degree} $e$. Note that as each $M_e$ is an $R$-module, ``$R$ has zero grading". The natural notion of map between graded modules is the following.
\begin{defn}
Let $M_i = \bigoplus_e M_{i,e}$, $M_j = \bigoplus_e M_{j,e}$ be graded modules over $R_i$, $R_j$ respectively. A \emph{graded module homomorphism} $f: M_i \To M_j$ \emph{of degree} $n \in \Z$ over a ring homomorphism $\bar{f}: R_i \To R_j$ is an abelian group homomorphism $f: M_i \To M_j$, which decomposes as 
\[
f = \bigoplus_{e \in \Z} f_e,
\quad \text{where each} \quad
f_e: M_{i,e} \To M_{j,e+n}
\]
is a module homomorphism over $\bar{f}$.
\end{defn}

That is, each $f_e$ is equivariant with respect to $\bar{f}$; for each $r \in R_i$ and $m \in M_{i,e}$ we have $f_e (rm) = \bar{f}(r) f_e (m) \in M_{j,e+n}$. We write $\deg f$ for the degree of $f$; $f$ shifts gradings by $\deg f$. Note that only finitely many of the $f_e$ are nonzero. 

Given two graded module homomorphisms $M_1 \stackrel{f}{\To} M_2 \stackrel{g}{\To} M_3$ of degrees $m,n$ over ring homomorphisms $R_1 \stackrel{\bar{f}}{\To} R_2 \stackrel{\bar{g}}{\To} R_3$, the composition $g \circ f$ is a graded module homomorphism of degree $m+n$ over $\bar{g} \circ \bar{f}$. In particular, $\deg(g \circ f) = \deg (f) + \deg(g)$. The identity map $M \To M$ is a graded module homomorphism of degree $0$.  Hence we can consider categories built out of graded modules, as follows.
\begin{defn}
A \emph{graded module category} $\M$ is a category such that:
\begin{enumerate}
\item each object is a pair $(M,R)$, where $M$ is a graded $R$-module;
\item a morphism $(M,R) \To (M',R')$ is a pair $(f, \bar{f})$ where $f$ is a graded module homomorphism $M \To M'$ over the ring homomorphism $\bar{f}: R \To R'$.
\end{enumerate}
\end{defn}

Such a category must, of course, contain identity morphisms and be closed under composition. Note that the morphisms of $\M$ have integer gradings, and  $\deg(g \circ f) = \deg (f) + \deg(g)$, but the objects do not. 
In this sense, $\M$ is \emph{relatively graded}.

\subsection{Necessity of ambiguity}
\label{sec:ambiguity}

An SQFT is defined as $(\D, c)$, where $\D$ a functor from $\DOS$ to graded $\Z_2$-vector spaces, and $c$ picks out suture elements in those vector spaces. The most straightforward generalisation to twisted coefficients, then, would involve a functor from $\DOS$ to a graded module category, together with suture elements.

The most striking difference of twisted SQFT from the untwisted (mod 2) case, is that both notions fail. We do not obtain a functor from $\DOS$ to a graded module category, and we do not have suture elements. We do, however, have weakened versions of both. Both suture elements and graded module homomorphisms are \emph{ambiguous up to units} in the appropriate coefficient rings. 

We now explain why a naive generalisation to twisted coefficient fails; and also, why a slightly-less-naive generalisation also fails!

The naive optimist, based on the above, and taking care to define graded modules over the correct rings and homomorphisms between them, might define a twisted SQFT as follows.
\begin{defn}
\label{defn:optimistic_twisted_SQFT}
An \emph{optimistic twisted SQFT} is a pair $(\D,c)$ where
\begin{enumerate}
\item
$\D$ is a functor $\DOS \To \M$, where $\M$ is a graded module category, associating to an occupied surface $(\Sigma,V)$ a graded module $\V(\Sigma,V)$ over the ring $\Z[H_1(\Sigma)]$, and to a decorated-isotopy class of decorated morphisms $(\phi, \Gamma_c)$ a graded module homomorphism $\D_{\phi, \Gamma_c}: \V(\Sigma,V) \To \V(\Sigma',V')$ over the ring homomorphism $\phi_*: \Z[H_1(\Sigma)] \To \Z[H_1(\Sigma')]$;
\item
$c$ assigns to each isotopy class of sutures $\Gamma$ on $(\Sigma,V)$ an element $c(\Gamma) \in \V(\Sigma,V)$, called a \emph{suture element}.
\end{enumerate}
The pair $(\D,c)$ satisfy the following conditions:
\begin{enumerate}
\item
Quadrangulations give tensor decompositions. For an isotopy class of quadrangulation $Q$ of $(\Sigma,V)$ with squares $(\Sigma_i^\square, V_i^\square)$,
\[
\V(\Sigma,V) \cong \Z[H_1(\Sigma)] \otimes \bigotimes_i \V(\Sigma_i^\square, V_i^\square),
\]
where the tensor products are over $\Z$. (This includes the ``null quadrangulation'' on an occupied vacuum $(\Sigma^\emptyset, V^\emptyset)$, which has $\V(\Sigma^\emptyset, V^\emptyset) \cong \Z$.) If$\Gamma$ is a basic set of sutures on $(\Sigma,V)$ restricting to $\Gamma_i$ on $(\Sigma_i^\square, V_i^\square)$, then under this isomorphism $c(\Gamma) = 1 \otimes \bigotimes_i c(\Gamma_i)$.
\item
Suture elements are respected. For any sutures $\Gamma$ on $(\Sigma,V)$,
\[
\mathcal{D}_{\phi,\Gamma_c}(c(\Gamma)) = c(\Gamma \cup \Gamma_c).
\]
\item
Basic sutures give bases. The suture elements $c(\Gamma_+) = \1$ and $c(\Gamma_-) = \0$ of the standard positive and negative sutures on the occupied square $(\Sigma^\square, V^\square)$ form a free basis for $\V(\Sigma^\square, V^\square)$ as a $\Z$-module.
\item
Euler class gives grading. Set the grading $e(\0) = -1$, $e(\1) = 1$ on $\V(\Sigma^\square, V^\square)$ and extend to obtain a $\Z$-grading on any $\V(\Sigma,V)$ (any element of any coefficient ring has grading $0$). Then for any set of sutures $\Gamma$, every element of $c(\Gamma)$ has grading $e(\Gamma)$.
\end{enumerate}
\end{defn}

(Note in \cite{Me12_itsy_bitsy} we defined the maps $\D_{\phi,\Gamma_c}$ to depend only on the decorated morphism $(\phi, \Gamma_c)$. There, we proved that $\D_{\phi, \Gamma_c}$ depended only on the decorated isotopy class of $(\phi,\Gamma_c)$; here, as the situation is more involved, we simply assume it, for now. See section \ref{sec:definition} for further discussion.)

Perhaps the least obvious part of this definition, given the above, is the extra tensor product with $\Z[H_1(\Sigma)]$ in condition (i); this is simply to ensure the correct coefficient ring.

In any case, our optimism is shortly crushed.
\begin{prop}
\label{prop:no_optimistic_TSQFT}
No optimistic twisted SQFT exists.
\end{prop}

\begin{lem}
\label{lem:optimistic_trivial_sutures_zero}
In an optimistic twisted SQFT, $c(\Gamma) = 0$ for any trivial set of sutures $\Gamma$.
\end{lem}

\begin{proof}
On the vacuum, consider a set of sutures $\Gamma_\circ^\pm$ which contains a properly embedded arc together with a contractible closed loop enclosing a $\pm$ region. We have $e(\Gamma_\circ^\pm) = \pm 2$, so $c(\Gamma)$ lies in the $\pm 2$-graded summand of $\V(\Sigma^\emptyset, V^\emptyset)$. But $\V(\Sigma^\emptyset, V^\emptyset)$ just consists of the coefficient ring $\Z$ itself, which has grading $0$, hence $c(\Gamma_\circ^\pm) = 0$. 

Now for any trivial sutures $\Gamma$ on any occupied $(\Sigma,V)$, there is a decorated morphism $(\phi, \Gamma_c)$ including the vacuum into $(\Sigma,V)$ and sending $\Gamma^+_\circ$ or $\Gamma^-_\circ \mapsto \Gamma$. Then $c(\Gamma) = \D_{\phi, \Gamma_c} c(\Gamma^\pm_\circ) = \D_{\phi, \Gamma_c} 0 = 0$.
\end{proof}

\begin{proof}[Proof of proposition \ref{prop:no_optimistic_TSQFT}]
Consider $(\Sigma,V)$ the disc with $6$ points, with the quadrangulation shown in figure \ref{fig:optimistic_sutures}. We call the top square $(\Sigma_1^\square, V_1^\square)$ an the bottom square $(\Sigma_2^\square, V_2^\square)$ as shown, and consider the sutures $\Gamma_{-+}, \Gamma_\pm, \Gamma_{+-}$ shown. Now $\V(\Sigma,V)$ is $4$-dimensional over $\Z$, and the $0$-graded summand $\V(\Sigma,V;0)$ is $2$-dimensional with ordered basis $( | \0 \1 \rangle, | \1 \0 \rangle) = (c(\Gamma_{-+}), c(\Gamma_{+-}))$, as shown.

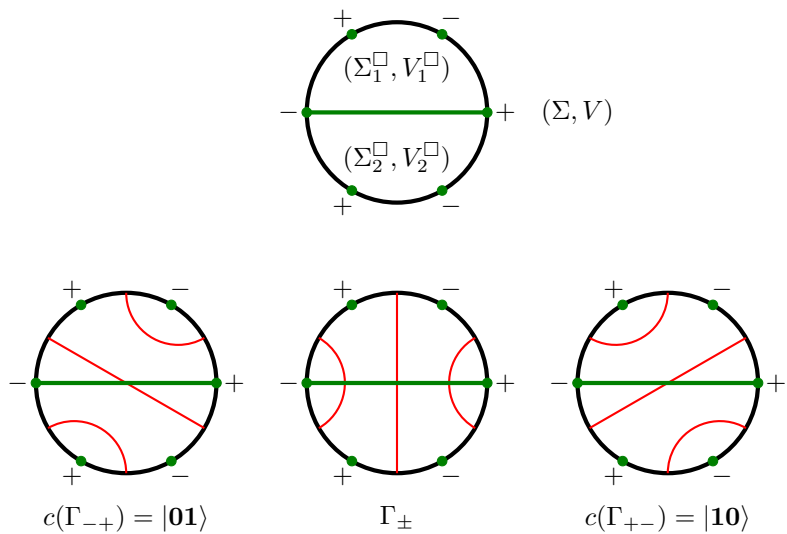
\begin{figure}
\begin{center}

\begin{tikzpicture}[
scale=1.2, 
suture/.style={thick, draw=red},
boundary/.style={ultra thick}, 
decomposition/.style={ultra thick, draw=green!50!black}]

\draw [boundary] (0,3) circle (1 cm);
\draw [decomposition] (-1,3) -- (1,3);
\foreach \angle in {0,120,240}
\draw (0,3) ++ (\angle:1.2) node {$+$};
\foreach \angle in {60,180,300}
\draw (0,3) ++ (\angle:1.2) node {$-$};
\foreach \angle in {0,60,120,180,240,300}
\fill [green!50!black, draw=green!50!black] (0,3) ++ (\angle:1) circle (1.5 pt);
\draw (0,3.5) node {$(\Sigma_1^\square, V_1^\square)$};
\draw (0,2.5) node {$(\Sigma_2^\square, V_2^\square)$};
\draw (2,3) node {$(\Sigma,V)$};

\draw [boundary] (3,0) circle (1 cm); 	
\draw [boundary] (0,0) circle (1 cm);
\draw [boundary] (-3,0) circle (1 cm);

\draw [suture] (30:1) arc (120:240:0.57735);
\draw [suture] (0,1) -- (0,-1);
\draw [suture] (210:1) arc (-60:60:0.57735);
\draw (0,-1.5) node {$\Gamma_\pm$};

\draw [suture] (3,0) ++ (150:1) arc (-120:0:0.57735);
\draw [suture] (3,0) ++ (30:1) -- ($ (3,0) + (210:1) $);
\draw [suture] (3,0) ++ (-30:1) arc (60:180:0.57735);
\draw (3,-1.5) node {$c(\Gamma_{+-}) = | \1 \0 \rangle$};

\draw [suture] (-3,1) arc  (180:300:0.57735);
\draw [suture] ($ (-3,0) + (-30:1) $) -- ($ (-3,0) +  (150:1) $);
\draw [suture] (-3,-1) arc (0:120:0.57735);
\draw (-3,-1.5) node {$c(\Gamma_{-+}) = | \0 \1 \rangle$};

\foreach \x in {-3,0,3}
{
\draw [decomposition] ($ (\x,0) + (-1,0) $) -- ($ (\x,0) + (1,0) $);
\foreach \angle in {0,60,120,180,240,300}
\fill [green!50!black, draw=green!50!black]  (\x,0) ++ (\angle:1) circle (1.5 pt);
\foreach \angle in {0,120,240}
\draw (\x,0) ++ (\angle:1.2) node {$+$};
\foreach \angle in {60,180,300}
\draw (\x,0) ++ (\angle:1.2) node {$-$};
}

\end{tikzpicture}

\caption{Sutures on the disc with $6$ points in optimistic twisted SQFT.}
\label{fig:optimistic_sutures}
\end{center}
\end{figure}

Now consider the decorated morphism $R : (\Sigma,V) \To (\Sigma,V)$ shown below, which rotates the disc $120^\circ$ anticlockwise. It has degree $0$, and we determine the $0$-graded component of $\D_R$, which is an abelian group homomorphism $\V(\Sigma,V;0) \To \V(\Sigma,V;0)$ described by a $2 \times 2$ integer matrix. We observe that $R$ sends the sutures $\Gamma_{-+} \mapsto \Gamma_{+-}$, and $R^3$ sends $\Gamma_{-+} \mapsto \Gamma_{-+}$ and $\Gamma_{+-} \mapsto \Gamma_{+-}$; so the integer matrix we seek has order $3$ and first column $(0,1)^T$. The only such matrix is $\begin{bmatrix} 0 & -1 \\ 1 & -1 \end{bmatrix}$. It follows that for the third set of sutures $\Gamma_\pm$ shown, $c(\Gamma_\pm ) = - |\0 \1 \rangle - | \1 \0 \rangle$. In particular, $c(\Gamma_{-+}) + c(\Gamma_{+-}) + c(\Gamma_|) = 0$.

Finally, consider the decorated morphism $(\Phi, \Gamma_c) : (\Sigma,V) \To (\Sigma^\square,V^\square)$ shown in figure \ref{fig:Phi}, from $(\Sigma,V)$ to the occupied square. It sends $\Gamma_{-+} \mapsto \Gamma_+$, $\Gamma_\pm \mapsto \Gamma_+$, and $\Gamma_{+-}$ to trivial sutures. 

\begin{figure}
\begin{center}

\begin{tikzpicture}[
scale=1.2, 
suture/.style={thick, draw=red},
boundary/.style={ultra thick}]

\draw [boundary] (0,0) circle (1 cm);
\draw [boundary] (0,0) circle (2 cm);

\draw [suture] (30:1) -- (45:2);
\draw [suture] (-30:1) -- (-45:2);
\draw [suture] (-90:1) -- (-135:2);
\draw [suture] (-150:1) to [bend left=45] (135:2);
\draw [suture] (90:1) arc (30:210:0.5);

\foreach \angle in {0,60,120,180,240,300}
\fill [green!50!black, draw=green!50!black]  (\angle:1) circle (1.5 pt);
\foreach \angle in {0,90,180,270}
\fill [green!50!black, draw=green!50!black]  (\angle:2) circle (1.5 pt);
\foreach \angle in {0,120,240}
\draw (\angle:0.8) node {$+$};
\foreach \angle in {60,180,300}
\draw (\angle:0.8) node {$-$};
\foreach \angle in {0,180}
\draw (\angle:2.2) node {$+$};
\foreach \angle in {90,270}
\draw (\angle:2.2) node {$-$};

\end{tikzpicture}

\caption{The morphism $(\Phi, \Gamma_c)$ giving a contradiction in optimistic twisted SQFT.}
\label{fig:Phi}
\end{center}
\end{figure}
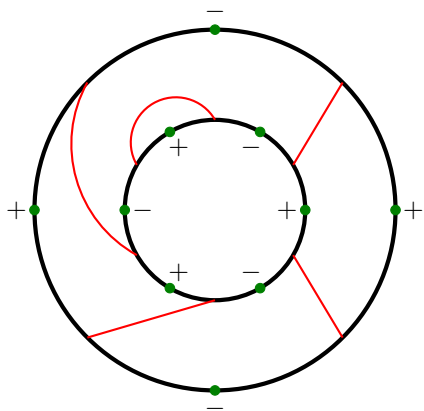

Thus $\D_{\Phi, \Gamma_c}$ must send $| \0 \1 \rangle \mapsto | \1 \rangle$, $| \1 \0 \rangle \mapsto 0$ and $- | \0 \1 \rangle - | \1 \0 \rangle \mapsto | \1 \rangle$, a contradiction.
\end{proof}

As the optimistic approach fails, it is natural, and consistent with quantum theory, and contact elements in Heegaard Floer homology \cite{HKM09}, to allow ambiguity up to units. We might hope, however, to relinquish as little certainty as possible. The suture elements $c(\Gamma)$ and the homomorphisms $\D_{\phi, \Gamma_c}$ are separate. But if we require suture elements $c(\Gamma)$ to remain unambiguous, then any requirement on homomorphisms to respect suture elements, i.e. $\D_{\phi,\Gamma_c}(c(\Gamma)) = c(\Gamma \cup \Gamma_c)$, makes the homomorphisms unambiguous also.

The possibility does remain, however, to retain certainty over the graded module homomorphisms $\D_{\phi, \Gamma_c}$ but allow the suture elements $c(\Gamma)$ to be ambiguous up to units. Thus we might define a \emph{sanguine twisted SQFT}, similarly to an optimistic twisted SQFT, except that:
\begin{enumerate}
\item
$c$ assigns to each isotopy class of sutures $\Gamma$ on $(\Sigma,V)$ an element $c(\Gamma)$ \emph{up to multiplication by units}; that is, $c(\Gamma) \subset \V(\Sigma,V)$ is the $\Z[H_1(\Sigma)]^\times$-orbit of a single element.
\item
A quadrangulation $(\Sigma_i^\square, V_i^\square)$ still gives a tensor decomposition $\V(\Sigma,V) \cong \Z[H_1(\Sigma)] \otimes \bigotimes_i \V(\Sigma_i^\square, V_i^\square)$, but now if we have basic sutures $\Gamma$ restricting to $\Gamma_i$ on $(\Sigma_i^\square, V_i^\square)$, then $c(\Gamma)$ is given by $\Z[H_1(\Sigma)]^\times \otimes \bigotimes_i c(\Gamma_i)$, i.e. every possible tensor product of elements of the $c(\Gamma_i)$, tensored also with units in $\Z[H_1(\Sigma)]$.
\item
Suture elements are still respected, but now we cannot guarantee equality $\D_{\phi, \Gamma_c} c(\Gamma) = c(\Gamma \cup \Gamma_c)$, as a graded module homomorphism may not send a $\Z[H_1(\Sigma)]^\times$-orbit to a whole $\Z[H_1(\Sigma')]^\times$-orbit. Rather we require that $\D_{\phi, \Gamma_c} c(\Gamma) \subseteq c(\Gamma \cup \Gamma_C)$.
\item
Basic sutures still give bases, but now we have, on the square $(\Sigma^\square, V^\square)$, that $c(\Gamma_+) = \{\pm \1\}$ and $c(\Gamma_-) = \{ \pm \0\}$, for some $\0, \1 \in \V(\Sigma^\square, V^\square)$. However $\0, \1$ still form a free basis for $\V(\Sigma^\square,V^\square)$ as $\Z$-module.
\end{enumerate}

Our sanguine attitude, however, is not justified.
\begin{prop}
\label{prop:no_sanguine_TSQFT}
No sanguine twisted SQFT exists.
\end{prop}

The proof of this proposition is a little more involved than the optimistic case. The idea is to use well-defined $\D_{\phi, \Gamma_c}$ to obtain enough well-defined single suture elements $s(\Gamma) \in c(\Gamma)$ to apply the idea of the proof in the optimistic case.

We repeatedly use the following fact: for any set of sutures $\Gamma$ on any occupied surface $(\Sigma,V)$, there is a decorated morphism $(\phi, \Gamma_c)$ from the vacuum $(\Sigma^\emptyset, V^\emptyset)$ to $(\Sigma,V)$ which sends the vacuum sutures $\Gamma^\emptyset$ to $\Gamma$; we call it a \emph{vacuum inclusion}.

First, consider trivial sutures. For trivial sutures $\Gamma_\circ^\pm$ on the vacuum, we have $c(\Gamma_\circ^\pm) = \{0\}$, having a bad grading, as in the optimistic case. A vacuum inclusion then gives the following.
\begin{lem}
\label{lem:sanguine_trivial_sutures_zero}
In a sanguine twisted SQFT, $c(\Gamma) = \{0\}$ for any trivial set of sutures $\Gamma$.
\qed
\end{lem}

Next, we consider vacuum sutures.
\begin{lem}
\label{lem:sanguine_vacuum_sutures_1}
$c(\Gamma^\emptyset) = \{\pm 1\} \in \Z \cong \V(\Sigma^\emptyset, V^\emptyset)$.
\end{lem}

\begin{proof}
As the vacuum has ``null quadrangulation'', we have $\V(\Sigma^\emptyset, V^\emptyset) \cong \Z$. Consider a vacuum inclusion into the square with standard positive sutures $\Gamma_+$. The corresponding module homomorphism $\Z \To \Z |\0 \rangle \oplus \Z | \1 \rangle$ sends $c(\Gamma^\emptyset)$ to $c(\Gamma_+) = \{\pm | \0 \rangle \}$. Hence $c(\Gamma^\emptyset)$ is a $\Z^\times$-orbit of $\Z$ consisting of primitive elements, so $c(\Gamma^\emptyset) = \{\pm 1\}$. 
\end{proof}

We can choose an isomorphism $\V(\Sigma^\emptyset, V^\emptyset) \cong \Z$ and define $s(\Gamma^\emptyset) = 1 \in \Z$.

We next consider the ``double vacuum'' $(\Sigma^\emptyset \sqcup \Sigma^\emptyset, V^\emptyset \sqcup V^\emptyset)$, which has double vacuum sutures $\Gamma^\emptyset \sqcup \Gamma^\emptyset$. Like the vacuum, it has a null quadrangulation, and we have
\[
\V(\Sigma^\emptyset \sqcup \Sigma^\emptyset, V^\emptyset \sqcup V^\emptyset) \cong \V(\Sigma^\emptyset, V^\emptyset) \otimes \V(\Sigma^\emptyset, V^\emptyset) \cong \Z \otimes_\Z \Z \cong \Z.
\]
The double sutures must thus have, via this ``null quadrangulation'', $c(\Gamma^\emptyset \sqcup \Gamma^\emptyset) = c(\Gamma^\emptyset) \otimes c(\Gamma^\emptyset) = \{\pm 1\} \otimes \{\pm 1\} \subseteq \Z \otimes \Z$, which under the isomorphism $\Z \otimes \Z \cong \Z$ corresponds to $c(\Gamma^\emptyset \sqcup \Gamma^\emptyset) = \{\pm 1\}$. 

There are decorated morphisms $\iota_1, \iota_2: (\Sigma^\emptyset, V^\emptyset) \To (\Sigma^\emptyset \sqcup \Sigma^\emptyset, V^\emptyset \sqcup V^\emptyset)$ given by the identity from vacuum onto the first and second copies of the vacuum respectively; with complementary vacuum sutures on the other copy of the vacuum. There are corresponding maps $\D_{\iota_1}, \D_{\iota_2} : \Z \To \Z \otimes \Z \cong \Z$.

\begin{lem}
\label{lem:sanguine_inclusion}
$\D_{\iota_1} = \D_{\iota_2}$.
\end{lem}

\begin{proof}
Consider both $\D_{\iota_1}, \D_{\iota_2}$ as maps $\Z \To \Z$. As both $\iota_1, \iota_2$ take $\Gamma^\emptyset \mapsto \Gamma^\emptyset \sqcup \Gamma^\emptyset$, they take $c(\Gamma^\emptyset) = \{\pm 1\}$ to $c(\Gamma^\emptyset \sqcup \Gamma^\emptyset) = \{\pm 1\}$, and hence $\D_{\iota_1} = \pm \D_{\iota_2}$. Suppose that $\D_{\iota_1} = - \D_{\iota_2}$.

We will now consider several decorated morphisms $\iota_{ij}$ from the ``double vacuum'' $(\Sigma^\emptyset \sqcup \Sigma^\emptyset, V^\emptyset \sqcup V^\emptyset)$ to the ``triple vacuum'' $(\Sigma^\emptyset \sqcup \Sigma^\emptyset \sqcup \Sigma^\emptyset, V^\emptyset \sqcup V^\emptyset \sqcup V^\emptyset)$. The map $\iota_{ij}$ takes the first copy of the vacuum to the $i$'th copy of the vacuum, and the second copy of the vacuum to the $j$'th copy, and has complementary vacuum sutures on the remaining vacuum. Thus for every pair $i \neq j$ of integers $i,j \in \{1,2,3\}$ we have such an $\iota_{ij}$. 

As with the double vacuum, the triple vacuum has a ``null quadrangulation'' and $\V(\Sigma^\emptyset \sqcup \Sigma^\emptyset \sqcup \Sigma^\emptyset, V^\emptyset \sqcup V^\emptyset \sqcup V^\emptyset) \cong \V(\Sigma^\emptyset, V^\emptyset)^{\otimes 3} \cong \Z^{\otimes 3} \cong \Z$. The ``triple vacuum sutures'' then must be $c(\Gamma^\emptyset \sqcup \Gamma^\emptyset \sqcup \Gamma^\emptyset) = c(\Gamma^\emptyset) \otimes c(\Gamma^\emptyset) \otimes c(\Gamma^\emptyset) = \{\pm 1\} \otimes \{\pm 1\} \otimes \{\pm 1\} \subset \Z^{\otimes 3}$, which under the isomorphism with $\Z$ gives $\{\pm 1\}$. Thus each $\D_{\iota_{ij}}$ is a homomorphism $\Z \To \Z$ which takes $c(\Gamma^\emptyset \sqcup \Gamma^\emptyset) = \{\pm 1\}$ to $c(\Gamma^\emptyset \sqcup \Gamma^\emptyset \sqcup \Gamma^\emptyset) = \{ \pm 1\}$, hence each $\D_{\iota_{ij}}$ is multiplication by $\pm 1$.

Now we use functoriality of $\D$, repeatedly using $\D_{\iota_1} = - \D_{\iota_2}$. Since $\iota_{12} \circ \iota_1 = \iota_{31} \circ \iota_2$ we have $\D_{\iota_{12}} = - \D_{\iota_{31}}$. Since $\iota_{12} \circ \iota_2 = \iota_{32} \circ \iota_2$ we have $\D_{\iota_{12}} = \D_{\iota_{32}}$. And since $\iota_{31} \circ \iota_1 = \iota_{32} \circ \iota_1$ we have $\D_{\iota_{31}} = \D_{\iota_{32}}$. Putting these equalities together gives $-\D_{\iota_{31}} = \D_{\iota_{12}} = \D_{\iota_{32}} = \D_{\iota_{31}}$, a contradiction as $\D_{\iota_{31}}$ is an isomorphism $\Z \To \Z$.
\end{proof}

We may now choose an isomorphism $\V(\Sigma^\emptyset \sqcup \Sigma^\emptyset, V^\emptyset \sqcup V^\emptyset) \cong \Z$ so that $D_{\iota_1} = \D_{\iota_2} : \Z \To \Z$ is the identity. We then define $s(\Gamma^\emptyset \sqcup \Gamma^\emptyset) = 1$ so $\D_{\iota_1} s(\Gamma^\emptyset) = \D_{\iota_2} s(\Gamma^\emptyset) = s(\Gamma^\emptyset \sqcup \Gamma^\emptyset)$.

For any sutures $\Gamma$ on any occupied $(\Sigma,V)$, we now define $s(\Gamma)$ as follows. Take a vacuum inclusion $(\phi, \Gamma_c):(\Sigma^\emptyset, V^\emptyset) \To (\Sigma,V)$ which sends $\Gamma^\emptyset \mapsto \Gamma$, and define $s(\Gamma) = \D_{\phi, \Gamma_c} s(\Gamma^\emptyset)$. We now show that these $s(\Gamma)$ are well defined, so that we have something close to an optimistic twisted SQFT.
\begin{lem}
Consider two distinct decorated morphisms $(\phi, \Gamma_c)$, $(\phi', \Gamma'_c)$ from the vacuum to $(\Sigma,V)$, sending $\Gamma^\emptyset \mapsto \Gamma$. Then $\D_{\phi, \Gamma_c} s(\Gamma^\emptyset) = \D_{\phi', \Gamma'_c} s(\Gamma^\emptyset)$.
\end{lem}

We write $s(\Gamma) = \D_{\phi, \Gamma_c} s(\Gamma^\emptyset)$ and $s'(\Gamma) = \D_{\phi', \Gamma'_c}(c(\Gamma^\emptyset))$; we must show $s(\Gamma) = s'(\Gamma)$.

\begin{proof}
By decorated isotopy of $(\phi, \Gamma_c)$ and $(\phi', \Gamma'_c)$ as necessary we may assume that the image discs of $\phi$ and $\phi'$ on $\Sigma$ are disjoint. Then $\phi, \phi'$ both factor through a decorated morphism $(\zeta, \Gamma_\zeta)$ from the double vacuum $(\Sigma^\emptyset \sqcup \Sigma^\emptyset, V^\emptyset \sqcup V^\emptyset)$ to $(\Sigma,V)$, where $\zeta$ includes the two disjoint discs into $\Sigma$. That is, we have decorated morphisms

\begin{center}
\begin{tikzpicture}[scale=1]
\draw (0,0) node {$(\Sigma^\emptyset, V^\emptyset)$};
\draw (5,0) node {$(\Sigma,V)$};
\draw (2.5,3) node {$(\Sigma^\emptyset \sqcup \Sigma^\emptyset, V^\emptyset \sqcup V^\emptyset)$};

\draw [->] (0.7,0.1) to node [above] {$(\phi, \Gamma_c)$} (4.3,0.1);
\draw [->] (0.7,-0.1) to node [below] {$(\phi', \Gamma'_c)$} (4.3,-0.1);

\draw [->] (0.1,0.3) to node [above left] {$\iota_1$} (2.1,2.7);
\draw [->] (0.3,0.3) to node [below right] {$\iota_2$} (2.3,2.7);

\draw [->] (2.7, 2.7) to node [above right] {$(\zeta, \Gamma_\zeta)$} (4.8,0.3);

\end{tikzpicture}
\end{center}
where $(\zeta, \Gamma_\zeta) \circ \iota_1 = (\phi, \Gamma_c)$, $(\zeta, \Gamma_\zeta) \circ \iota_2 = (\phi', \Gamma'_c)$. (This is a complicated way of saying that two disjoint embeddings of a disc can be regarded as an embedding of two discs!)

From lemma \ref{lem:sanguine_inclusion} we have $\D_{\iota_1} = \D_{\iota_2}$ and we have chosen an isomorphism $\V(\Sigma^\emptyset, V^\emptyset)$ so that $\D_{\iota_1} = \D_{\iota_2} = 1$. We now have, using functoriality of $\D$,
\[
s(\Gamma) = \D_{\phi, \Gamma_c} s(\Gamma^\emptyset) = \D_{\zeta,\Gamma_\zeta} \D_{\iota_1} s(\Gamma^\emptyset) 
= \D_{\zeta, \Gamma_\zeta} \D_{\iota_2} s(\Gamma^\emptyset) = \D_{\phi', \Gamma'_c} s(\Gamma^\emptyset) = s'(\Gamma).
\]
\end{proof}

(Note that this proof relied upon being able to replace $(\phi, \Gamma_c)$ and $(\phi', \Gamma'_c)$ with decorated-isotopic morphisms. As such it relies on the maps $\D_{\phi, \Gamma_c}$ depending only on the decorated-isotopy class of $(\phi, \Gamma_c)$. We will discuss this requirement later in section \ref{sec:definition}.)

\begin{lem}
The suture elements $s(\Gamma)$ are respected: for any sutures $\Gamma$ on $(\Sigma,V)$, and decorated morphism $(\phi, \Gamma_c): (\Sigma,V) \To (\Sigma',V')$,
\[
\D_{\phi,\Gamma_c} s(\Gamma) = s(\Gamma \cup \Gamma_c).
\]
\end{lem}

\begin{proof}
Taking a small disc on $(\Sigma,V)$ intersecting $\Gamma$ in an arc, and its image under $\phi$ on $\Sigma'$, we obtain vacuum inclusions into $(\Sigma,V)$ and $(\Sigma',V')$, which send sutures $\Gamma^\emptyset$ on the vacuum to $\Gamma$ and $\Gamma'$ respectively, so that the following diagram commutes.

\begin{center}
\begin{tikzpicture}[scale=1]
\draw (2.5,3) node {$(\Sigma^\emptyset, V^\emptyset)$};
\draw (0,0) node {$(\Sigma,V)$};
\draw (5,0) node {$(\Sigma',V')$};

\draw [->] (0.7,0) to node [above] {$(\phi, \Gamma_c)$} (4.3,0);
\draw [->] (2.2,2.7) to node [above left] {$\iota$} (0.2,0.3);
\draw [->] (2.7,2.7) to node [above right] {$\iota'$} (4.8,0.3);
\end{tikzpicture}
\end{center}
We then have, by definition $s(\Gamma') = \D_{\iota'} s(\Gamma^\emptyset) = \D_{\phi,\Gamma_c} \D_\iota s(\Gamma^\emptyset) = \D_{\phi,\Gamma_c} s(\Gamma)$.
\end{proof}

We can now complete the proof of nonexistence of sanguine twisted SQFT along similar lines as the optimistic case.
\begin{proof}[Proof of proposition \ref{prop:no_sanguine_TSQFT}]
As in the proof of proposition \ref{prop:no_optimistic_TSQFT} Let $(\Sigma,V)$ be a quadrangulated disc with 6 vertices. We have $\V(\Sigma,V;0)$ 2-dimensional over $\Z$ with ordered basis $(s(\Gamma_{-+}), s(\Gamma_{+-}))$.

Considering again the rotation morphism $R$ and the corresponding module homomorphism in grading $0$, since $R$ sends $\Gamma_{-+} \mapsto \Gamma_{+-}$ and $R^3$ sends sutures to themselves, we again obtain $s(\Gamma_{-+}) + s(\Gamma_{+-}) + s(\Gamma_\pm) = 0$. 

And considering $(\Phi, \Gamma_c): (\Sigma,V) \To (\Sigma^\square, V^\square)$ again we find $\D_{\Phi, \Gamma_c}$ sends
\[
s(\Gamma_{-+}) \mapsto s(\Gamma_+), \quad
s(\Gamma_{+-}) \mapsto 0, \quad
- s(\Gamma_{-+}) - s(\Gamma_{+-}) \mapsto s(\Gamma_+).
\]
Adding these gives $2s(\Gamma_+) = 0$, a contradiction as $s(\Gamma_+) \in c(\Gamma_+) = \{\pm \1 \} \subset \Z$.
\end{proof}

This completes the proof of proposition \ref{prop:necessity_of_ambiguity}. Twisted SQFT is therefore consigned to eternal ambiguity; but only up to units, which is not a particularly costly price to pay.

\subsection{Definition of twisted SQFT}
\label{sec:definition}

Having learned our lesson from the previous section, we cannot obtain a bona fide functor from $\DOS$, the decorated occupied surface category, to a graded module category. We consider graded module homomorphisms up to units, and suture elements up to units. A twisted SQFT is thus a `` functor up to units''. 

When we speak of an element of an $R$-module $M$ ``up to units'', we really mean the orbit of an element of $M$ under the multiplicative action of the group of units $R^\times$. We require each $c(\Gamma)$ to be such an orbit, i.e. a $\Z[H_1(\Sigma)]^\times$-orbit in $\V(\Sigma,V)$.

The following lemma shows that, even when module homomorphisms are ambiguous up to units, they never mix up such orbits.
\begin{lem}
\label{lem:preserving_unit_orbits}
Let $M, M'$ be graded modules over rings $R,R'$, and let $f: M \To M'$ be a graded module homomorphism over a ring homomorphism $\bar{f}$. If $m_1, m_2 \in M$ lie in the same $R^\times$-orbit, then $f(m_1), f(m_2)$ lie in the same $R'^\times$-orbit. That is, $f$ sends an $R^\times$-orbit of $M$ into an $R'^\times$-orbit of $M'$.
\end{lem}

\begin{proof}
If $m' = rm$ for some $r \in R^\times$, then $f(m') = f(rm) = \bar{f}(r) \; f(m)$, where $\bar{f}(r)$ is a unit, as any ring homomorphism preserves units.
\end{proof}

We now define a twisted sutured quadrangulated field theory. 
\begin{defn}
\label{defn:twisted_sqft}
A \emph{twisted sutured quadrangulated field theory} (twisted SQFT) is a pair $(\mathcal{D},c)$, where
\begin{enumerate}
\item
$\mathcal{D}$ associates:
\begin{enumerate}
\item
to an occupied surface $(\Sigma,V)$ a graded $\Z[H_1(\Sigma)]$-module $\V(\Sigma,V)$, and 
\item
to a decorated-isotopy class of decorated morphism $(\phi, \Gamma_c) \; : \; (\Sigma,V) \To (\Sigma',V')$, a graded module homomorphism $\D_{\phi, \Gamma_c} : \V(\Sigma,V) \To \V(\Sigma',V')$ over the ring homomorphism $\phi_*: \Z[H_1(\Sigma)] \To \Z[H_1(\Sigma')]$, well defined up to multiplication by units in $\Z[H_1(\Sigma')]$;
\end{enumerate}
\item
$c$ assigns to each (isotopy class of) sutures $\Gamma$ on $(\Sigma,V)$ a $\Z[H_1(\Sigma)]^\times$-orbit $c(\Gamma)$ of $\V(\Sigma,V)$.
\end{enumerate}
The pair $(\D,c)$ satisfy the following conditions:
\begin{enumerate}
\item
$\D$ is ``a functor up to units''. That is:
\begin{enumerate}
\item
To an identity morphism $(\Sigma,V) \To (\Sigma,V)$, $\D$ assigns the identity map $\V(\Sigma,V) \To \V(\Sigma,V)$, up to units.
\item
To the composition of two decorated morphisms $(\Sigma,V) \stackrel{(\phi, \Gamma_c)}{\To} (\Sigma',V') \stackrel{(\phi',\Gamma'_c)}{\To} (\Sigma'',V'')$, $\D$ assigns the composition $\V(\Sigma,V) \stackrel{\D_{\phi,\Gamma_c}}{\To} \V(\Sigma',V') \stackrel{\D_{\phi',\Gamma'_c}}{\To} \V(\Sigma'',V'')$ of graded module homomorphisms, up to units. 
\end{enumerate}
\item
Quadrangulations give tensor decompositions. For any (isotopy class of) quadrangulation $Q$ of $(\Sigma,V)$ with squares $(\Sigma_i^\square, V_i^\square)$,
\[
\V(\Sigma,V) \cong \Z[H_1(\Sigma)] \otimes \bigotimes_i \V(\Sigma_i^\square, V_i^\square)
\]
where the tensor products are over $\Z$. (This includes the ``null quadrangulation'' on an occupied vacuum $(\Sigma^\emptyset, V^\emptyset)$, which has $\V(\Sigma^\emptyset, V^\emptyset) \cong \Z$.) If $\Gamma$ is a basic set of sutures on $(\Sigma,V)$ restricting to $\Gamma_i$ on $(\Sigma_i^\square, V_i^\square)$ then under this isomorphism $c(\Gamma) = \Z[H_1(\Sigma)]^\times \otimes \bigotimes_i c(\Gamma_i)$.
\item
Suture elements are respected: for any sutures $\Gamma$ on $(\Sigma,V)$,
\[
\mathcal{D}_{\phi,\Gamma_c}(c(\Gamma)) \subseteq c(\Gamma \cup \Gamma_c).
\]
\item
Basic sutures give bases: $c(\Gamma_+) = \{ \pm \1\} $ and $c(\Gamma_-) = \{ \pm \0 \}$ for two elements $\0, \1 \in \V(\Sigma^\square, V^\square)$, which form a free basis for $\V(\Sigma^\square, V^\square)$ as a $\Z$-module.
\item
Euler class gives grading. Set the grading $e(\0) = -1$, $e(\1) = 1$ on $\V(\Sigma^\square, V^\square)$ and extend to obtain a $\Z$-grading on any $\V(\Sigma,V)$ (any element of any coefficient ring has grading $0$). Then for any set of sutures $\Gamma$, every element of $c(\Gamma)$ has grading $e(\Gamma)$.
\end{enumerate}
\end{defn}

We now elaborate on some aspects of this definition, and note some simple properties.

\emph{Graded module homomorphisms up to units.}
As we have mentioned, the graded module homomorphisms $\D_{\phi, c}$ are defined ``up to units'', and should properly be considered as equivalence classes up to multiplication by units. But there is some further subtlety arising from gradings. We may post-multiply a graded module homomorphism by a unit to obtain another homomorphism --- but in fact, we may multiply by a separate unit in each graded summand. That is, given a graded module homomorphism $f = \bigoplus_e f_e  : \V(\Sigma,V) \To \V(\Sigma',V')$ with graded summands $f_e: \V(\Sigma,V;e) \To \V(\Sigma',V';e+\deg (f))$, we may multiply each $f_e$ by a separate unit $z_e \in \Z[H_1(\Sigma')]^\times$ and the result is another graded module homomorphism $\bigoplus_e z_e f_e$. When we say ``up to units'', we mean that $f$ is only well defined up to multiplying by a unit in each summand in this way. These form an equivalence class of graded module homomorphisms and strictly speaking $\D_{\phi, \Gamma_c}$ is this collection of homomorphisms.

\emph{``Functor up to units'' and categorical considerations.}
The first property of $\D$ expresses precisely ``functoriality up to units''. In the untwisted $\Z_2$ case, there are no units other than $1$, so ``up to units'' becomes redundant, and SQFT is a bona fide functor. It would be more satisfactory if $\D$ could be expressed cleanly as a bona fide functor between appropriate ``quotient'' categories; this can be done, but involves greater technical detail, and we leave it to subsequent work.

\emph{Naturality of $\V(\Sigma,V)$.}
Considering decorated morphisms which are \emph{homeomorphisms}, functoriality up to units gives $\V(\Sigma,V)$ properties of ``naturality up to units'' (see \cite{Juhasz_Thurston12} for a full discussion of naturality). In particular, if $\phi: (\Sigma,V) \To (\Sigma',V')$ is a homeomorphism then $\D_{\phi, \emptyset}$ is an isomorphism. (More precisely, an equivalence class up to units of isomorphisms.) This follows since $\D$ respects composition and identity up to units, and since $(\phi, \emptyset)$ has inverse decorated morphism $(\phi^{-1}, \emptyset)$.

\emph{Decorated-isotopy classes.}
We have specified, unlike in \cite{Me12_itsy_bitsy}, that the homomorphisms $\D_{\phi, \Gamma_c}$ only depend on the \emph{decorated-isotopy class} of $(\phi, \Gamma_c)$. In the $\Z_2$ case the uniqueness of suture elements made the invariance of homomorphisms under decorated isotopy clear \cite[lemma 8.3]{Me12_itsy_bitsy}. It turns out also in the twisted case that imposing decorated-isotopy invariance is redundant, but this requires further work which we defer to the sequel.

\emph{Free modules and bases.}
Writing $\V(\Sigma,V) \cong \Z[H_1(\Sigma)] \otimes \bigotimes_i \V(\Sigma_i^\square, V_i^\square)$, expresses $\V(\Sigma,V)$ as a free module over $\Z[H_1(\Sigma)]$. Each $\V(\Sigma_i^\square, V_i^\square)$ is a free 2-dimensional $\Z$-module; their tensor product over $\Z$ is a $2^{I(\Sigma,V)}$-dimensional $\Z$-module (recall $I(\Sigma,V)$ is the number of squares); taking the tensor product with $\Z[H_1(\Sigma)]$ gives a free $2^{I(\Sigma,V)}$-dimensional $\Z[H_1(\Sigma)]$-module. Given a quadrangulation $Q$ of $(\Sigma,V)$, we naturally obtain a basis of $\V(\Sigma,V)$ over $\Z[H_1(\Sigma,V)]$. Let the $I(\Sigma,V)$ squares of $Q$ be $(\Sigma_i^\square, V_i^\square)$, and take a basis $\0$, $\1$ of each $\V(\Sigma_i^\square, V_i^\square)$. Then a basis of $\V(\Sigma,V)$ over $\Z[H_1(\Sigma)]$ is given by 
\[
\e_1 \otimes \e_2 \otimes \cdots \otimes \e_{I(\Sigma,V)} = | \e_1 \e_2 \cdots \e_{I(\Sigma,V)} \rangle,
\]
where each $\e_i$ is $\0$ or $\1$. These $2^{I(\Sigma,V)}$ basis elements consist of a representative of each $c(\Gamma)$, over all basic sets of sutures $\Gamma$. When $(\Sigma,V)$ is the vacuum with the null quadrangulation, we have $\V(\Sigma,V) \cong \Z$ and we write a basis over $\Z$ simply as $1$ (not $\1$).

\emph{Bases for graded components.}
Writing the graded decomposition of $\V(\Sigma,V)$ as $\bigoplus_e \V(\Sigma,V;e)$, we then obtain a basis for each $\V(\Sigma,V;e)$. Let $n_0, n_1$ denote the number of $\0$'s and $\1$'s respectively among $\e_1, \ldots, \e_{I(\Sigma,V)}$. Then $\V(\Sigma,V;e)$ has basis over $\Z[H_1(\Sigma)]$ given by those $| \e_1 \e_2 \cdots \e_{I(\Sigma,V)} \rangle$ satisfying $n_1 - n_0 = e$. These basis elements are precisely those which are suture elements for basic sutures with Euler class $e$. It follows that $\V(\Sigma,V;e)$ is nontrivial precisely for $e$ of the same parity as $I(\Sigma,V)$ and satisfying $|e| \leq I(\Sigma,V)$. This corresponds to the fact that $e(\Gamma)$ satisfies the same properties, for nontrivial sutures $\Gamma$.

\emph{Sutures up to homeomorphism vs isotopy.}
Two sets of sutures $\Gamma,\Gamma'$ on $(\Sigma,V)$ may be homeomorphic in the sense that there is a homeomorphism $\phi: \Sigma \To \Sigma$ taking $\Gamma$ to $\Gamma'$; but such $\Gamma,\Gamma'$ are in general not isotopic. For instance, if $\Gamma$ is a set of sutures on a disc, and we rotate the disc, the resulting sutures $\Gamma'$ will in general be non-isotopic. Suture elements $c(\Gamma)$ are invariant under isotopy of sutures, but in general not under homeomorphism; so $c(\Gamma), c(\Gamma')$ will in general be different.

\emph{Keeping track of suture element orbits.}
We saw in lemma \ref{lem:preserving_unit_orbits} that each $\D_{\phi, \Gamma_c}: \V(\Sigma,V) \To \V(\Sigma',V')$ sends a $\Z[H_1(\Sigma)]^\times$-orbit into a $\Z[H_1(\Sigma')]^\times$-orbit. Note, however, that $\phi_*: \Z[H_1(\Sigma)] \To \Z[H_1(\Sigma')]$ might not be surjective, and so the image of a $c(\Gamma)$ under $\D_{\phi,\Gamma_c}$ might not be an entire $\Z[H_1(\Sigma')]^\times$-orbit. This is why we write $c(\Gamma \cup \Gamma_c) \subseteq c(\Gamma')$, rather than equality. Given the ambiguity in $\D_{\phi, \Gamma_c}$, however, we may post-multiply $\D_{\phi, \Gamma_c}$ by a unit to hit any given element of $c(\Gamma')$.

\emph{Gradings of decorated morphisms.}
Finally, we note that the degree of each $\D_{\phi, \Gamma_c}$ is given by the amount by which is shifts Euler classes, i.e. the degree of the decorated morphism $\deg(\phi, \Gamma_c)$.

\subsection{First properties of suture elements}
\label{sec:first_properties}

We can now develop some basic properties of suture elements. 

Some of the arguments used in the optimistic and sanguine cases apply immediately for bona fide twisted SQFT. For instance, the proof of lemma \ref{lem:sanguine_vacuum_sutures_1} immediately gives a proof that vacuum sutures have suture elements $\pm 1$.
\begin{lem}
Let $(\Sigma^\emptyset, V^\emptyset)$ be the vacuum, and $\Gamma^\emptyset$ the vacuum sutures. Then $\V(\Sigma^\emptyset, V^\emptyset) \cong \Z$ and $c(\Gamma^\emptyset) = \{ \pm 1 \}$.
\qed
\end{lem}

As for trivial sutures, the proofs of lemmas \ref{lem:optimistic_trivial_sutures_zero} and \ref{lem:sanguine_trivial_sutures_zero} show that suture elements are zero.
\begin{lem}
Suppose $\Gamma$ is a set of sutures on occupied $(\Sigma,V)$ which contains a contractible closed loop. Then $c(\Gamma) = \{0\}$.
\qed
\end{lem}

We can also demonstrate a relation between suture elements of bypass triples. As mentioned earlier, in mod-2 SQFT one can show that if $\Gamma_0, \Gamma_1, \Gamma_2$ is a bypass triple then $c(\Gamma_0) + c(\Gamma_1) + c(\Gamma_2) = 0$. A similar result was shown by Massot \cite[Lemma 18]{Massot09} in the twisted case. We prove it in our context, using methods similar to \cite[sec. 8.4]{Me12_itsy_bitsy}.

\begin{prop}
In a twisted SQFT, if $\Gamma_0, \Gamma_1, \Gamma_2$ form a bypass triple, then there exist $c_i \in c(\Gamma_i)$ such that
\[
c_0 + c_1 + c_2 = 0.
\]
\end{prop}

\begin{proof}
Consider, as in the proof of proposition \ref{prop:no_optimistic_TSQFT} and figure \ref{fig:optimistic_sutures}, the disc with 6 vertices $(\Sigma,V)$, the quadrangulation $(\Sigma^\square_1, V^\square_1)$, $(\Sigma^\square_2, V^\square_2)$, the sets of sutures $\Gamma_{-+}, \Gamma_\pm, \Gamma_{+-}$, and the decorated morphism $R$ which rotates sutures $120^\circ$ anticlockwise. The corresponding $\D_R$ has degree $0$ and we consider its $0$-graded component. Now $V(\Sigma,V;0)$ has basis $(| \0 \1 \rangle, | \1 \0 \rangle )$, with respect to which $\D_R$ is given by a $2 \times 2$ integer matrix up to sign. Since $R$ takes $\Gamma_{-+} \mapsto \Gamma_{+-}$, and $R^3$ sends $\Gamma_{-+} \mapsto \Gamma_{-+}$, $\Gamma_{+-} \mapsto \Gamma_{+-}$, the $0$-graded summand of $\D_R$ has a $2 \times 2$ matrix representative $M$ such that
\[
M = \begin{bmatrix} 0 & \alpha \\ \pm 1 & \beta \end{bmatrix},
\quad
M^3 = \begin{bmatrix} \pm 1 & 0 \\ 0 & \pm 1 \end{bmatrix},
\]
for some $\alpha,\beta \in \Z$. There are four such matrices, and all of them have $\alpha = \pm 1$, $\beta = \pm 1$.

Now $\Gamma_{-+}, \Gamma_\pm, \Gamma_{+-}$ form a bypass triple and have suture elements $c(\Gamma_{-+}) = \pm |\0 \1 \rangle$, $c(\Gamma_{+-}) = \pm | \1 \0 \rangle $, and $c(\Gamma_\pm) = \pm (\alpha |\0 \1 \rangle + \beta | \1 \0 \rangle ) $. As each of $\alpha, \beta \in \{ \pm 1\}$, we have representatives
\[
-\alpha |\0 \1 \rangle \in c(\Gamma_{-+}),
\quad
-\beta | \1 \0 \rangle \in c(\Gamma_{+-}),
\quad
\alpha | \0 \1 \rangle + \beta | \1 \0 \rangle \in c(\Gamma_\pm)
\]
which sum to $0$.

Now any bypass triple of sutures $\Gamma_0,\Gamma_1,\Gamma_2$ on an occupied surface $(\Sigma',V')$ can be obtained via a decorated morphism $(\phi, \Gamma_c) : (\Sigma,V) \To (\Sigma',V')$, which takes the triple $\Gamma_{-+}, \Gamma_\pm, \Gamma_{+-}$ to $\Gamma_0, \Gamma_1, \Gamma_2$ respectively. Taking a representative graded module homomorphism $\D_{\phi, \Gamma_c}$, the three suture elements $-\alpha |\0 \1 \rangle$, $- \beta |\1 \0 \rangle$ and $\alpha |\0\1\rangle + \beta |\1\0 \rangle$ are taken by $\D_{\phi, \Gamma_c}$ to suture elements of $\Gamma_0, \Gamma_1, \Gamma_2$ which sum to zero. 
\end{proof}

The bypass relation can be used to simplify basic sutures, and express suture elements of general sutures in terms of basis elements. 

Mod 2, when there is no ambiguity, we can take a quadrangulation $Q$ and, given any sutures $\Gamma$, we can successively simplify $\Gamma$ along the internal edges of the $Q$ by bypass surgeries. In this way, $c(\Gamma)$ can be written as a sum of basis suture elements.

In the twisted case, we have the difficulty that signs and units become uncertain each time we apply the bypass relation. However, in \cite{Me10_Sutured_TQFT, Me11_torsion_tori} we developed methods for keeping track of signs of certain representatives of certain suture elements. In a subsequent paper, we will apply those methods to the present context; in the next section, we show what can be achieved by reducing twisted SQFT mod 2.

\subsection{Reduction to mod 2 SQFT}
\label{sec:reduction_mod_2}

It turns out that we may simply drop all the group ring paraphernalia from twisted SQFT and reduce mod 2, and the result is untwisted SQFT. In this section, we prove this result.

Recall that the coefficient ring $\Z[H_1(\Sigma)]$ is a Laurent polynomial ring: specifically, if $H_1(\Sigma)$ has generators $A_1, A_2, \ldots, A_m$, then $\Z[H_1(\Sigma)]$ is a Laurent polynomial ring in the variables $q_1 = e^{A_1}, \ldots, q_m = e^{A_m}$, using exponential notation.

There is thus a canonical ring homomorphism $\Z[H_1(\Sigma)] \To \Z$ obtained by setting all $q_i \mapsto 1$ or $A_i \mapsto 0$; equivalently, by taking sums of coefficients of Laurent polynomials. Composing with reduction mod 2, we obtain a canonical ring homomorphism
\[
r_\Sigma: \Z[H_1(\Sigma)] \To \Z_2.
\]
Note that $r_\Sigma$ takes all units of $\Z[H_1(\Sigma)]$ to units of $\Z_2$, i.e. to $1$.

Now we show how to reduce graded modules, graded module homomorphisms, and suture elements.

Each $\V(\Sigma,V)$ is a free $\Z[H_1(\Sigma)]$-module; a quadrangulation gives a basis $|\0 \cdots \0 \rangle$, $\ldots$, $| \1 \cdots \1 \rangle$. We define the reduced graded $\Z_2$-module (vector space) $\V_r(\Sigma,V)$ to be the free $\Z_2$-module on the same basis, with the same gradings. There is then a canonical ``reduction mod 2'' graded module homomorphism
\[
R_{\Sigma,V}: \V(\Sigma,V) \To \V_r(\Sigma,V)
\quad
\text{over the ring homomorphism}
\quad
r_\Sigma: \Z[H_1(\Sigma)] \To \Z_2,
\]
which has degree $0$ and is the identity on basis elements.

From a graded module homomorphism $f: \V(\Sigma,V) \To \V(\Sigma',V')$ of degree $n$, over a ring homomorphism $\bar{f}: \Z[H_1(\Sigma)] \To \Z[H_1(\Sigma')]$, with summands $f = \bigoplus_e f_e$, where $f_e: \V(\Sigma,V;e) \To \V(\Sigma',V';e+n)$, we define a graded linear $\Z_2$-vector space map $f^r: \V_r(\Sigma,V) \To \V_r (\Sigma',V')$ (or equivalently, a graded $\Z_2$-module homomorphism over the identity ring homomorphism $\Z_2 \To \Z_2$), with summands $f^r = \bigoplus_e f^r_e$, where $f^r_e: \V_r (\Sigma,V;e) \To \V_r(\Sigma',V';e+n)$, such that the following diagram of graded module homomorphisms commutes.
\[
\begin{array}{ccc}
\V(\Sigma,V;e) & \stackrel{f_e}{\To} & \V(\Sigma',V';e+n) \\
\downarrow R_{\Sigma,V} && \downarrow R_{\Sigma',V'} \\
\V_r(\Sigma,V;e) & \stackrel{f^r_e}{\To} & \V_r(\Sigma',V';e+n)
\end{array}
\]
As each $\V(\Sigma,V;e)$ is free and $R_{\Sigma,V}$ is the identity on basis elements, we can define each $f^r_e$ on each basis element to be the mod-2 reduced image of $f_e$ applied to the same basis element. This gives a graded linear $\Z_2$-vector space map.

Given a graded module homomorphism $\D_{\phi, \Gamma_c}$, which is well-defined up to units in each summand, we thus obtain a reduced graded linear map $\D_{\phi,\Gamma_c}^r$ between graded $\Z_2$-vector spaces, which is well-defined up to units in $\Z_2^\times$ in each summand. As $\Z_2^\times = \{1\}$, the reduced graded linear map $\D_{\phi, \Gamma_c}^r$ is in fact well-defined.

Finally, for any $c(\Gamma) \subset \V(\Sigma,V)$, the set $c(\Gamma)$ is a $\Z[H_1(\Sigma)]^\times$-orbit. Its reduced image $R_{\Sigma,V} c(\Gamma)$ is a $\Z_2^\times$-orbit in $\V_r(\Sigma,V)$, i.e. a single element, which we can call $c_r(\Gamma)$.

We now show that this ``reduced'' structure gives an (untwisted) SQFT, in the sense of \cite[defn. 8.1]{Me12_itsy_bitsy}, recapitulated in definition \ref{def:SQFT}.

(Technically, we have defined twisted SQFT to associate $\D_{\phi, \Gamma_c}$, and hence $\D^r_{\phi,\Gamma_c}$, to a decorated isotopy class of decorated morphism, while in (untwisted) SQFT they are associated to a specific decorated morphism. As mentioned earlier, the difference is redundant, but strictly speaking, given a decorated morphism $(\phi, \Gamma_c)$, we associate to it the graded module homomorphism of its decorated isotopy class.)

\begin{prop}
\label{prop:mod-2_reduction}
The assignments $(\Sigma,V) \rightsquigarrow \V_r(\Sigma,V)$, $(\phi,\Gamma_c) \rightsquigarrow \D^r_{\phi,\Gamma_c}$ and $\Gamma \rightsquigarrow c_r(\Gamma)$ form a sutured quadrangulated field theory.
\end{prop}

\begin{proof}
The assignments are all of the correct type. We first show that $\D^r$ is a functor, using the fact that $\D$ is a ``functor up to units''. For an identity morphism $\phi: (\Sigma,V) \To (\Sigma,V)$ we have $\D_{\phi,\emptyset}$ is the identity up to units; so $\D^r_{\phi,\emptyset}$ is the identity. For a composition $(\Sigma,V) \stackrel{(\phi,\Gamma_c)}{\To} (\Sigma',V') \stackrel{(\phi',\Gamma'_c)}{\To} (\Sigma',V'')$, the graded module homomorphism of the composition $\D_{(\phi', \Gamma'_c) \circ (\phi,\Gamma_c)}$ is equal to the composition of the graded module homomorphisms $\D_{\phi',\Gamma'_c} \circ \D_{\phi,\Gamma_c}$, up to units; reducing mod 2, the reduced map $\D^r_{(\phi', \Gamma'_c) \circ (\phi,\Gamma_c)}$ is equal to the composition $\D^r_{\phi',\Gamma'_c} \circ \D^r_{\phi,\Gamma_c}$. So we have a functor $\D^r$.

The property of twisted SQFT that quadrangulations give tensor decompositions of modules, once reduced mod 2, immediately gives the required tensor decomposition property of mod-2 SQFT. The fact that $\D_{\phi,\Gamma_c} c(\Gamma) \subseteq c(\Gamma \cup \Gamma_c)$, reducing mod 2 (where suture elements become single elements) gives immediately $\D^r_{\phi,\Gamma_c} c_r(\Gamma) = c_r (\Gamma \cup \Gamma_c)$. The property that basic sutures give bases in twisted SQFT immediately says the same in mod-2 SQFT. And the Euler grading property also carries over immediately.

Thus the definition of a twisted SQFT is satisfied.
\end{proof}

\subsection{Properties inherited from SQFT}
\label{sec:inheritance}

Proposition \ref{prop:mod-2_reduction} tells us that a twisted SQFT, reduced mod 2, has the properties of untwisted SQFT. We immediately obtain the following results; their SQFT analogues were all mentioned in section \ref{sec:background}.

\begin{lem}
If $\Gamma_1$, $\Gamma_2$ are sets of sutures on $(\Sigma_1,V_1)$, $(\Sigma_2,V_2)$ then, reduced mod 2, $c_r(\Gamma_1 \sqcup \Gamma_2) = c_r(\Gamma_1) \otimes c_r(\Gamma_2) \in \V_r((\Sigma_1, V_1) \sqcup (\Sigma_2,V_2)) = \V_r (\Sigma_1, V_1) \otimes \V_r (\Sigma_2, V_2)$.
\end{lem}

\begin{proof}
Proposition \ref{prop:mod-2_reduction} and \cite[prop. 8.7]{Me12_itsy_bitsy}.
\end{proof}

\begin{lem}
If $\phi: (\Sigma,V) \To (\Sigma',V')$ is a standard gluing then reduced map, $\D^r_{\phi,\emptyset} : \V_r(\Sigma,V) \To \V_r(\Sigma',V')$ is the identity $\bigotimes_i \V_r(\Sigma_i^\square, V_i^\square) \To \bigotimes_i \V_r (\Sigma_i^\square, V_i^\square)$, where $(\Sigma_i^\square, V_i^\square)$ are the squares of a quadrangulation.
\end{lem}

\begin{proof}
Proposition \ref{prop:mod-2_reduction} and \cite[lemma 8.2]{Me12_itsy_bitsy}.
\end{proof}

\begin{lem}
Any graded module homomorphism in twisted SQFT, reduced mod 2, is a composition of digital creation and general digital annihilation operators.
\end{lem}

\begin{proof}
Proposition \ref{prop:mod-2_reduction} and \cite[thm. 1.1]{Me12_itsy_bitsy}.
\end{proof}

\begin{lem}
If $\Gamma$ is a confining set of sutures then every element of $c(\Gamma)$, reduced mod 2, is $0$.
\end{lem}

\begin{proof}
Mod-2, this is \cite[thm. 8.8]{Me12_itsy_bitsy}, following Massot \cite{Massot09}. In any case it is then immediate from proposition \ref{prop:mod-2_reduction}.
\end{proof}

Some versions of the above properties hold more generally in twisted SQFT, without reducing mod 2, but require further background. For instance, the graded module homomorphism of a standard gluing in twisted SQFT is clearly (plus or minus) the identity on each basis element, but showing it is (plus or minus) the identity on each graded summand requires further work. As mentioned in section \ref{sec:first_properties}, we will apply the methods of \cite{Me10_Sutured_TQFT, Me11_torsion_tori} to resolve these issues in the sequel.

However, we have not yet answered an even more basic question: whether a twisted SQFT exists. The rest of this paper is devoting to answering that question in the affirmative with sutured Floer homology.

\section{Sutured Floer homology with twisted coefficients}
\label{sec:twisted_SFH}

\subsection{Background: the untwisted case}
\label{sec:SFH_background}

We begin with a brief summary of sutured Floer homology, following \cite{Ju06}. Initially we consider the untwisted case; we defer twisted coefficients to section \ref{sec:introducing_twisted_coeffs}.

Sutured Floer homology is an invariant of balanced sutured 3-manifolds, introduced by Juh\'{a}sz in \cite{Ju06}, based on that hat-version $\widehat{HF}$ of Ozsv\'{a}th-Szab\'{o}'s Heegaard Floer homology \cite{OS04Closed}.

A \emph{sutured 3-manifold} $(M, \Gamma)$ is (for our purposes) a 3-manifold with no closed components, and $(\partial M, \Gamma)$ a sutured surface, with each boundary component having nonempty sutures. (Note however that the notion of sutured 3-manifold is more standard than ``sutured surface'' in this paper, and predates it: see e.g. \cite{Gabai83}). Our definition is not identical to Gabai's original definition.) 

Thus, $\Gamma$ is an embedded oriented 1-submanifold of $\partial M$, containing curves on each component of $\partial M$, such that
\begin{enumerate}
\item $\partial M \backslash \Gamma = R_+ \cup R_-$, where $R_\pm$ are oriented as $\pm \partial M$, and
\item $\overline{\partial R_\pm} = \Gamma$ as oriented 1-manifolds.
\end{enumerate}
A sutured 3-manifold is \emph{balanced} if $\chi(R_+) = \chi(R_-)$. 

Sutured Floer homology (SFH) associates to a balanced sutured 3-manifold $(M, \Gamma)$ a module $SFH(M,\Gamma; R)$ over a coefficient ring $R$. The coefficient ring may be $\Z_2$, $\Z$ or the group ring $\Z[H_2(M)]$; we respectively call these three cases \emph{mod-2}, \emph{signed} and \emph{twisted} SFH. When the coefficient ring is clear or unimportant we omit it and simply write $SFH(M,\Gamma)$.

A \emph{sutured Heegaard diagram} is an oriented surface $S$ (possibly disconnected) with two sets of pairwise disjoint simple closed curves $\{\alpha_i\}_{i=1}^k$ and $\{\beta_i\}_{i=1}^l$ in the interior of $S$. From a sutured Heegaard diagram we can construct a 3-manifold $M$ by thickening $S$ into $S \times [0,1]$ and gluing thickened discs to each $\alpha_i \times \{0\}$ and $\beta_i \times \{1\}$. This $M$ naturally has sutures $\Gamma$ given by $\partial S \times \{\frac{1}{2}\}$, $R_+$ given by $S \times \{1\}$ surgered along the $\beta_i$ (together with $\partial S \times (\frac{1}{2},1]$), and $R_-$ given by $S \times \{0\}$ surgered along the $\alpha_i$ (together with $\partial S \times [0, \frac{1}{2})$). We say that $(S, \alpha, \beta)$ is a \emph{Heegaard decomposition} of $(M, \Gamma)$. Juh\'{a}sz shows \cite[proposition 2.9]{Ju06} that $(M, \Gamma)$ is balanced iff $k=l$, $S$ has no closed components, and the elements of $\alpha$ and $\beta$ are both linearly independent in $H_1(S, \Q)$. We call such a diagram \emph{balanced}. All sutured Heegaard diagrams in this paper are balanced, hence of the form $(S, \alpha, \beta)$ where $\alpha = \{\alpha_i\}_{i=1}^k$, $\beta = \{\beta_i\}_{i=1}^k$. 

Every balanced sutured 3-manifold has a Heegaard decomposition (necessarily balanced), and any two Heegaard diagrams for a balanced sutured $(M, \Gamma)$ are diffeomorphic after a finite sequence of Heegaard moves.

From a balanced sutured Heegaard diagram $(S, \alpha, \beta)$, we can consider the components of $S \backslash \left( \bigcup \alpha \cup \bigcup \beta \right)$. Some of these components contain points of $\partial S$ and are \emph{boundary-adjacent}; the others are disjoint from $\partial S$ and we call them \emph{internal}. Let $D_1, \ldots, D_m$ be the internal components, oriented from $S$, and let $D(S,\alpha,\beta)$ be the free abelian group on $\{D_1, \ldots, D_m\}$; so $D(S,\alpha,\beta) \cong \Z^m$. Elements of $D(S,\alpha,\beta)$ are called \emph{domains}. A domain has well-defined coefficients, which are integers; a domain is \emph{positive} if all its coefficients are positive; similarly a domain may be negative, non-positive or non-negative. A domain also has a well-defined boundary, which is an integer combination of arcs of $\alpha$ and $\beta$ curves. 

A domain is \emph{periodic} if its boundary is an integer combination of whole $\alpha$ and $\beta$ curves. The periodic domains form a subgroup of $D(S, \alpha, \beta)$ which we denote by $\P(S,\alpha,\beta)$. A periodic domain can be regarded as a homology between $\alpha$ and $\beta$ curves, so taking boundaries of periodic domains gives a homomorphism $\P(S,\alpha,\beta) \To H_1(\alpha) \cap H_1(\beta)$ (where we regard $H_1(\alpha), H_1(\beta) \subset H_1(S)$) --- which is injective (as $H_2(S) = 0$) and surjective (as any homology between $\alpha$ and $\beta$ curves is given by a periodic domain), hence an isomorphism. 

A balanced sutured Heegaard diagram is \emph{admissible} if every nonzero periodic domain has both positive and negative coefficients. Every balanced Heegaard diagram is isotopic to an admissible one.

Sutured Floer homology (for our purposes) considers holomorphic curves in the \emph{symmetric product} $\Sym^k S$ of a Heegaard surface, with boundary conditions imposed by the $\alpha$ and $\beta$ curves. 

The symmetric product $\Sym^k S = \frac{S^{\times k}}{S_k}$ is a $2k$-manifold given by the quotient of the $k$-fold Cartesian product $S^{\times k}$ by the action of the symmetric group $S_k$ permuting coordinates. We denote the image of the point $(x_1, \ldots, x_k) \in S^{\times k}$ in $\Sym^k S$ as a set $\{x_1, \ldots, x_k\}$. In $\Sym^k S$ there are $k$-dimensional tori $\T_\alpha$ and $\T_\beta$ given by the products of the $\alpha$ and $\beta$ curves; that is,
\[
\T_\alpha = \frac{\alpha_1 \times \cdots \times \alpha_k}{S_k}
\quad \text{and} \quad
\T_\beta = \frac{\beta_1 \times \cdots \times \beta_k}{S_k}.
\]
Note that as the $\alpha_i$ are disjoint, $\T_\alpha$ is just an embedded torus in $\Sym^k S$ homeomorphic to $(S^1)^k$; likewise for $\T_\beta$. For generic $\alpha$ and $\beta$, $\T_\alpha \cap \T_\beta$ will be a finite set of points in $\Sym^k S$.

Let $\Disc = \{z \in \C \; : \; |z| \leq 1 \}$ be the unit disc, and let $\partial \Disc_+ = \{z \in \partial \Disc \; : \; \Re z \geq 0\}$, $\partial \Disc_- = \{z \in \partial \Disc \; : \; \Re z \leq 0 \}$ be its right and left boundaries respectively. Given $x,y \in \T_\alpha \cap \T_\beta$, a \emph{Whitney disc} from $x$ to $y$ is a continuous map $u: \Disc \To \Sym^k S$ satisfying $u(i) = y$, $u(\partial \Disc_+) \subset \T_\alpha$, $u(\partial \Disc_-) \subset \T_\beta$ and $u(-i) = x$. There is a well-defined notion of homotopy of a Whitney disc from $x$ to $y$, and $\pi_2(x,y)$ denotes the set of homotopy classes of such Whitney discs.

A Whitney disc $u$ determines a domain $D(u) \in D(\Sigma,\alpha,\beta)$ as follows. For an internal component $D_i$ of $S \backslash \left( \bigcup \alpha \cup \bigcup \beta \right)$, choose a point $z_i \in D_i$ and let $n_{z_i} (u)$ be the algebraic intersection number of $u$ (or a transverse perturbation thereof) with $\{ z_i \} \times \Sym^{k-1} S$. Then we define $D(u) = \sum_{i=1}^m n_{z_i} (u) \; D_i$. This domain does not change under homotopy of a Whitney disc, and hence a homotopy class $\phi \in \pi_2(x,y)$ determines a domain $D(\phi)$.

Consider $x,y \in \T_\alpha \cap \T_\beta$ and a domain $D \in D(\Sigma,\alpha,\beta)$. The boundary $\partial D$ consists of an integer combination of arcs in $\alpha$ and $\beta$; each $\alpha_i \cap \partial D$ (resp. $\beta_i \cap \partial D$) is an integer combination of arcs in $\alpha_i$ (resp. $\beta_i$). For each $i$, all of the sets $x \cap \alpha_i$, $x \cap \beta_i$, $y \cap \alpha_i$ and $y \cap \beta_i$ consist of single points. We say $D$ \emph{connects $x$ to $y$} if
\[
\partial (\alpha_i \cap \partial D) = (x \cap \alpha_i) - (y \cap \alpha_i)
\quad \text{and} \quad
\partial (\beta_i \cap \partial D) = (x \cap \beta_i) - (y \cap \beta_i).
\]
We denote by $D(x,y)$ the set of domains connecting $x$ to $y$. If $\phi \in \pi_2(x,y)$ then $D(\phi) \in D(x,y)$; thus we have a map $D: \pi_2 (x,y) \To D(x,y)$.

The sutured Floer chain complex $CF(\Sigma,\alpha,\beta)$ is freely generated, as a module over the coefficient ring $R$, by the points of $\T_\alpha \cap \T_\beta$. The differential $\partial: CF(S,\alpha,\beta) \To CF(S,\alpha,\beta)$ counts certain index-1 holomorphic representatives of Whitney discs. 

Specifically, given $\phi \in \pi_2(x,y)$, and fixing an appropriate almost complex structure on $\Sym^k S$, let $\M(\phi)$ denote the moduli space of holomorphic representatives of $\phi$, i.e. holomorphic maps $\Disc \To \Sym^k S$ which are Whitney discs from $x$ to $y$ in the homotopy class of $\phi$. The expected dimension of $\M(\phi)$ (if transversality holds) is given by the Maslov class $\mu(\phi)$, which is determined by $\phi$ alone; alternatively it is determined by the domain $D(\phi)$, and we may write $\mu(D)$. There is an $\R$-action on $\M(\phi)$ by holomorphic automorphisms of $\Disc$ fixing $\pm i$; denote the quotient by $\widehat{\M}(\phi)$. When $\mu(\phi) = 1$ then the reduced moduli space $\widehat{\M}(\phi)$ consists of finitely many signed points.

The differential counts index-1 holomorphic Whitney discs from a given $x \in \T_\alpha \cap \T_\beta$ to various other points $y \in \T_\alpha \cap \T_\beta$. We define
\[
\partial x = \sum_{y \in \T_\alpha \cap \T_\beta} y \sum_{\substack{\phi \in \pi_2 (x,y) \\\ \mu(\phi) = 1}} \# \widehat{\M}(\phi).
\]
Admissibility of $(S,\alpha,\beta)$ guarantees that for every $x,y$, the set of \emph{non-negative} domains in $D(x,y)$ is finite; the domain of any holomorphic Whitney disc is of this type, by positivity of intersection. One can show that $\partial^2 = 0$ and the homology depends only on $(M, \Gamma)$, not on the choice of Heegaard decomposition, almost complex structure or any other choices. This homology $SFH(M,\Gamma;R)$ is the \emph{sutured Floer homology} of $(M, \Gamma)$ with coefficients in $R$. When $R = \Z_2$, we simply count curves modulo $2$. Over $\Z$, we use the fact that points of $\widehat{\M}(\phi)$ are signed and perform a signed count.

Each $x \in \T_\alpha \cap \T_\beta$ has an associated \emph{spin-c} structure $\s(x)$. If there is a Whitney disc connecting $x$ and $y$ then $\s(x) = \s(y)$. In fact, a Whitney disc connecting $x$ to $y$ exists, i.e. $\pi_2(x,y) \neq \emptyset$, iff $\s(x) = \s(y)$. The space $\Spin^c(M,\Gamma)$ is affine over $H^2(M, \partial M) \cong H_1(M)$.
	
The chain complex $(CF(S,\alpha,\beta),\partial)$ thus splits as a direct sum over spin-c structures, $(CF(S, \alpha, \beta), \partial) \cong \bigoplus_\s (CF(S, \alpha, \beta, \s), \partial_\s)$, where the summand $CF(S, \alpha, \beta, \s)$ is the submodule generated by intersection points $x \in \T_\alpha \cap \T_\beta$ with $\s(x) = \s$, and $\partial_\s$ is the restriction of $\partial$ to this subgroup.

A sutured surface $(\Sigma, \Gamma)$ in a 3-manifold naturally describes a \emph{contact structure} in a neighbourhood of the surface, i.e. a non-integrable 2-plane field. The surface is \emph{convex} and $\Gamma$ is its dividing set \cite{Gi91}. Hence we can regard a sutured 3-manifold as a 3-manifold with a germ of a contact structure on its boundary. We may then speak of a contact structure on $(M, \Gamma)$ as an extension of this structure near the boundary.

Sutured Floer homology provides invariants of contact structures. Let $\xi$ denote a contact structure on a balanced sutured 3-manifold $(M, \Gamma)$. With $R = \Z_2$, $\Z$ or $\Z[H_2(M)]$, there is a subset $c(\xi) \subset SFH(-M,-\Gamma;R)$ naturally associated to $\xi$, which is an $R^\times$-orbit of a single element of $SFH(-M,-\Gamma;R)$. (The minus signs refer to reversed orientation but are essentially irrelevant for our purposes.) When $R = \Z_2$, this $c(\xi)$ reduces to a single element of $SFH(-M, -\Gamma; \Z_2)$. See \cite{HKM09, Ghiggini_Honda08}.

\subsection{SQFT from SFH}
\label{sec:SQFT_from_SFH}

As shown in \cite{Me12_itsy_bitsy}, SFH with $\Z_2$ coefficients naturally gives an SQFT, by the following construction which we now recall. From an occupied surface $(\Sigma,V)$, we may define a balanced sutured 3-manifold $(\Sigma \times S^1, V \times S^1)$. The circle $S^1$ is oriented as usual. The sutures are the curves $\{v\} \times S^1$ where $v \in V$. We orient each suture positively or negatively around $S^1$ accordingly as $v \in V_+$ or $V_-$.

We set
\[
\V(\Sigma,V) = SFH(-\Sigma \times S^1, -V \times S^1),
\]
which Honda--Kazez--Mati\'{c} proved is equal to
\[
\bV^{\otimes I(\Sigma,V)} = \left( \Z_2 \0 \oplus \Z_2 \1 \right)^{\otimes I(\Sigma,V)},
\]
where $\bV$ is a standard 2-dimensional $\Z_2$-vector space with basis elements $\0, \1$ having grading $-1$ and $1$ respectively.

Note that we can equivalently use $V \times S^1$ or $F \times S^1$ as sutures. As mentioned in section \ref{sec:background} an occupied surface $(\Sigma,V)$ is equivalent to a sutured background $(\Sigma,F)$; $V$ and $F$ determine each other. We can give points of $F$ signs by declaring that $f \in F$ is positive (resp. negative) if, proceeding along the oriented boundary $\partial \Sigma$ from $f$, the next point of $V$ encountered is negative (resp. positive). We can write $F = F_+ \sqcup F_-$ and form a sutured 3-manifold $(\Sigma \times S^1, F \times S^1)$ homeomorphic to $(\Sigma \times S^1, V \times S^1)$. Thus, at our convenience, we may use $V \times S^1$ or $F \times S^1$ as sutures; we say that $\Sigma \times S^1$ has \emph{$V$-sutures} or \emph{$F$-sutures} accordingly.

A set of sutures $\Gamma$ on $(\Sigma,V)$ gives a contact structure $\xi_\Gamma$ on $\Sigma \times I$, and gluing $\Sigma \times \{0\}$ to $\Sigma \times \{1\}$ gives a contact structure on $\Sigma \times S^1$. The suture element is then defined as $c(\Gamma) = c(\xi_\Gamma)$. On the solid (``square'') torus $(\Sigma^\square \times S^1, V^\square \times S^1)$ there are (up to isotopy) precisely two tight contact structures, corresponding to the two basic sets of sutures on the occupied square.

Work of Honda--Kazez--Mati\'{c} \cite{HKM08} essentially then gives, for any decorated occupied surface morphism, a linear map of sutured Floer homology, as required of the functor $\mathcal{D}$; it respects suture/contact elements. A gluing theorem of Juh\'{a}sz \cite{Ju08} gives tensor decompositions from quadrangulations, and computations in \cite{HKM08} or \cite{Me09Paper, Me10_Sutured_TQFT} show that the square's vector space is 2-dimensional with the desired basis. Sutured Floer homology is naturally graded by spin-c structures, which correspond to Euler classes of contact structures, which determine Euler classes of sutures.

Thus, we obtain SQFT from SFH. We now turn to twisted coefficients, and using similar ideas, try to produce a twisted SQFT from SFH.

\subsection{Introducing twisted coefficients}
\label{sec:introducing_twisted_coeffs}

Twisted coefficients were considered by Ozsv\'{a}th--Szab\'{o} \cite{OS04Prop} for Heegaard Floer homology, and by Massot \cite{Massot09} and Ghiggini--Honda  \cite{Ghiggini_Honda08} for sutured Floer homology. As Ghiggini--Honda note, the extension of sutured Floer homology to the case of twisted coefficients is analogous to the closed case. 

To explain how twisted coefficients apply to $SFH$, and extend some results that we will need, we follow closely the arguments of \cite{OS04Closed} and \cite{OS04Prop}, adapting arguments there to the case of $SFH$.

Roughly, the extra feature in the definition of SFH with twisted coefficients is that the differential $\partial$ contains an extra term $e^{A(\phi)}$, where $A(\phi) \in H_2(M)$ is associated to the homotopy class $\phi \in \pi_2(x,y)$ of a Whitney disc. But to define $A(\phi)$ we first need to discuss the homology and homotopy of various spaces involved.

\subsection{Homology of the symmetric product}
\label{sec:homology_symmetric_product}

We now consider the symmetric product $\Sym^k S$ in some detail, adapting relevant statements of \cite{OS04Closed, OS04Prop} from the closed to the sutured case.

There are two maps
\[
i \; : \; S \To \Sym^k S,
\quad
j \; : \; H_1 ( \Sym^k S ) \To H_1 (S)
\]
defined as follows. Choose a basepoint $x_0 \in S$. Then $i$ is defined by $x \mapsto \{x, x_0, \ldots, x_0\}$. To define $j$ we note that a curve in general position in $\Sym^k S$ gives a map from a $k$-fold cover of $S^1$ to $S$; then $j$ is the corresponding map on homology. As a cobordism between curves in general position in $\Sym^k S$ gives a map from a branched cover of a surface to $S$, $j$ is well-defined.

\begin{lem}[Cf. Lemma 2.6 of \cite{OS04Closed}]
The maps $i$ and $j$ induce inverse isomorphisms on $H_1$. Further, the abelianization map $\pi_1(\Sym^k S) \To H_1(\Sym^k S)$ is an isomorphism. Thus
\[
H_1(S) \cong H_1(\Sym^k S) \cong \pi_1 (\Sym^k S).
\]
\end{lem}

\begin{proof}
That $i$ and $j$ are inverse on $H_1$ is clear; hence they are isomorphisms. It remains to show $\pi_1(\Sym^k S)$ is abelian; equivalently, any null-homologous $\gamma: S^1 \To \Sym^k S$ is null-homotopic. Using the map $j$, a null-homologous $\gamma$ gives a null-homologous map $\hat{\gamma}: C \To S$, where $C$ is a $k$-fold cover of $S^1$. Hence there is a map $\bar{\gamma}: F \To S$, where $F$ is a 2-manifold with boundary $C$ and $\bar{\gamma}|_C = \hat{\gamma}$. As in the proof in \cite{OS04Closed}, increasing the genus of $F$ if necessary, we can assume $F$ is a $k$-fold branched cover of the disc $\Disc$, with covering map $\pi: F \To \Disc$. Then the map $\Disc \To \Sym^k S$ defined by $z \mapsto i(\bar{\gamma}(\pi^{-1}(z)))$ is a null-homotopy of $\gamma$.
\end{proof}

In $H_1 (S)$ there are $\Z$ subgroups generated by each $[\alpha_i]$ and $[\beta_i]$. In a balanced sutured Heegaard diagram, the curves of $\alpha$ are linearly independent in $H_1 (S)$, and similarly for $\beta$, so we have inclusions of $\Z^k$ subgroups $H_1(\alpha), H_1(\beta) \subset H_1(S)$. 

In $\Sym^k S$ we have the inclusion $\T_\alpha \hookrightarrow \Sym^k S$, and a subgroup of $H_1(\Sym^k S)$ given by its image. But $H_1(\T_\alpha)$ is generated by loops of the form $\alpha_i \times 0 \times \cdots \times 0$, where $0$ denotes a constant loop. The images of these loops in $H_1(\Sym^k S)$ are $i_* [\alpha_i]$, for $i = 1, \ldots, k$, which are linearly independent since $i_*$ is an isomorphism and the $[\alpha_i]$ are linearly independent in $H_1(S)$. Thus the inclusion $\T_\alpha \hookrightarrow \Sym^k S$ induces an injective map on $H_1$. As both the torus $\T_\alpha$ (homeomorphic to $(S^1)^k$) and $\Sym^k S$ have abelian fundamental groups, the inclusion also induces an injective map on $\pi_1$. The same applies for $\T_\beta$. We record these results.

\begin{lem}
\label{lem:subgroup_isomorphisms}
There are abelian group inclusions induced by inclusions of spaces
\[
H_1(\alpha), H_1(\beta) \subset H_1(S), \quad
H_1 (\T_\alpha), H_1(\T_\beta) \subset H_1(\Sym^k S), \quad
\pi_1 (\T_\alpha), \pi_1(\T_\beta) \subset \pi_1 (\Sym^k S).
\]
Under the isomorphisms $H_1(S) \cong H_1(\Sym^k S) \cong \pi_1 (\Sym^k S)$ induced by $i,j$ and abelianization, we have $H_1(\alpha) \cong H_1(\T_\alpha) \cong \pi_1(\T_\alpha) \cong \Z^k$ and $H_1(\beta) \cong H_1(\T_\beta) \cong \pi_1(\T_\beta) \cong \Z^k$.
\qed
\end{lem}

On the other hand, $M$ is obtained by gluing discs to $S \times [0,1]$ along $\alpha \times \{0\}$ and $\beta \times \{1\}$, so we have
\[
H_1(M) \cong \frac{H_1 (S)}{H_1(\alpha) + H_1(\beta)} \cong \frac{ H_1 (\Sym^k S) }{ H_1 (\T_\alpha) + H_1 (\T_\beta) }.
\]

\subsection{Homotopy classes of Whitney discs}
\label{sec:homotopy_classes_Whitney_discs}

We now consider the set $\pi_2(x,y)$ of homotopy classes of Whitney discs, adapting \cite[sec. 2.4]{OS04Closed} to the sutured case.

The set $\pi_2(x,y)$ has some algebraic structure: a Whitney disc $u$ from $x$ to $y$, and a Whitney disc $v$ from $y$ to $z$, can be spliced together to give a Whitney disc $u*v$ from $x$ to $z$. This operation $*$ descends to homotopy classes: from $\phi \in \pi_2(x,y)$ and $\psi \in \pi_2(y,z)$ we obtain $\phi * \psi: \pi_2(x,z)$. 

Any $\pi_2(x,x)$ in fact forms an abelian group: it is not difficult to construct an identity and inverse of a Whitney disc from $x$ to $x$; and in a similar way as for any second homotopy group, the operation $*$ is commutative. Moreover, any $\phi \in \pi_2(x,y)$ has an inverse $\phi^{-1} \in \pi_2(y,x)$ such that $\phi * \phi^{-1} \in \pi_2(x,x)$ and $\phi^{-1} * \phi \in \pi_2 (y,y)$ are both the identity.

We will first consider Whitney discs from an intersection point to itself, i.e. $\pi_2(x,x) \cong H_2(M)$, before moving on to $\pi_2(x,y)$ in general.

Let $\Omega(\T_\alpha, \T_\beta)$ be the space of paths in $\Sym^k S$ from $\T_\alpha$ to $\T_\beta$, with basepoint the constant path $x$. We then have $\pi_2(x,x) \cong \pi_1 \Omega(\T_\alpha, \T_\beta)$. Evaluating at the endpoints of such a path gives a map $\ev: \Omega(\T_\alpha, \T_\beta) \To \T_\alpha \times \T_\beta$. We then have a Serre fibration
\[
\Omega \Sym^k S \To \Omega(\T_\alpha, \T_\beta) \stackrel{\ev}{\To} \T_\alpha \times \T_\beta.
\]
The homotopy long exact sequence of this fibration includes maps
\[
\begin{array}{ccccccc}
\pi_1(\Omega \Sym^k s) &\To& \pi_1 \Omega(\T_\alpha, \T_\beta) &\To& \pi_1(\T_\alpha \times \T_\beta) &\To& \pi_0 \Omega \Sym^k S.
\end{array}
\]
We simplify some of these terms. First, $\pi_1 \Omega \Sym^k S \cong \pi_2 \Sym^k S = 0$, by the argument in the proof of \cite[prop. 2.7]{OS04Closed}, since $S$ deformation retracts to a wedge of circles. We have seen $\pi_1 \Omega(\T_\alpha, \T_\beta) \cong \pi_2(x,x)$.  And $\pi_0 \Omega \Sym^k S \cong \pi_1 (\Sym^k S) \cong H_1 (\Sym^k S)$. So the above exact sequence becomes
\[
0 \To \pi_2(x,x) \To \pi_1 (\T_\alpha \times \T_\beta) \stackrel{f}{\To} H_1(\Sym^k S).
\]
Thus $\pi_2(x,x) \cong \ker f$. The map $f$ takes a pair of (homotopy classes of) loops in $\gamma_\alpha: S^1 \To \T_\alpha$ and $\gamma_\beta: S^1 \To \T_\beta$, based at $x$, to the element $[\gamma_\alpha] - [\gamma_\beta]$ in $H_1(\Sym^k S)$. Thus $\ker f \cong \pi_1(\T_\alpha) \cap \pi_1(\T_\beta)$, where we consider $\pi_1(\T_\alpha)$, $\pi_1(\T_\beta)$ as subgroups of $\pi_1(\Sym^k S)$, using lemma \ref{lem:subgroup_isomorphisms}.

On the other hand, if we consider the long exact sequence of the spaces $\alpha \cup \beta \hookrightarrow S \hookrightarrow M$, we obtain
\[
0 \cong H_2(S) \To H_2(M) \To H_1(\alpha \cup \beta) \stackrel{g}{\To} H_1(S) \To H_1(M).
\]
Hence $\ker g \cong H_2(M)$; but also, as $g$ is induced by inclusion, $\ker g \cong H_1(\alpha) \cap H_1(\beta)$. The isomorphism $H_1(\alpha) \cap H_1(\beta) \cong H_2(M)$ can be understood geometrically as follows: a homology class in $H_1(\Sigma)$ that is both a linear combination of $[\alpha_i]$'s and a linear combination of $[\beta_i]$'s is a 1-cycle on $S$ that bounds surfaces on both sides of $\Sigma$; hence it gives an 2-cycle in $M$. Considering lemma \ref{lem:subgroup_isomorphisms} again, we have $H_1(\alpha) \cap H_1(\beta) \cong H_1(\T_\alpha) \cap H_1(\T_\beta)$ so $\ker g \cong \ker f$.  Putting together the isomorphisms $\pi_2(x,x) \cong \ker f \cong \ker g \cong H_2(M) \cong H_1(\alpha) \cap H_1(\beta)$, we have proved the following.
\begin{lem}
There is an isomorphism $\pi_2(x,x) \cong H_2(M)$, which takes $\phi = [u] \in \pi_2(x,x)$, where $u: \Disc \To \Sym^k S$ is a Whitney disc from $x$ to $x$, to an element of $H_2(M)$ as follows. The restrictions 
$u|_{\partial \Disc_+}$ and $u|_{\partial \Disc_-}$ give homologous loops in $H_1(\T_\alpha)$ and $H_1(\T_\beta)$, hence an element of $H_1(\alpha) \cap H_1(\beta) \cong H_2(M)$.
\qed
\end{lem}

We can now turn from $\pi_2(x,x)$ to $\pi_2(x,y)$ in general. Recall that $\pi_2(x,y) \neq \emptyset$ iff the spin-c structures of $x$ and $y$ are equal, $\s(x) = \s(y)$. 

Given a point $x \in \T_\alpha \cap \T_\beta$ with spin-c structure $\s$, a \emph{complete set of paths for $\s$} is a choice of homotopy classes $\theta_y \in \pi_2(x,y)$, for each $y \in \T_\alpha \cap \T_\beta$ with spin-c structure $\s$. (Cf. \cite[defn. 2.12]{OS04Closed}.) We take $\theta_x$ to be the constant Whitney disc at $x$.

For any $y,z \in \T_\alpha \cap \T_\beta$ with spin-c structure $\s$, we then obtain a map $\pi_2(y,z) \To \pi_2(x,x)$ given by $u \mapsto \theta_y * u * \theta_z^{-1}$. Since Whitney discs have inverses with respect to the splicing operation $*$, the complete set of paths for $\s$ gives a bijection $\pi_2(y,z) \cong \pi_2(x,x)$. Following this with the isomorphism $\pi_2(x,x) \cong H_2(M)$ of the previous lemma, we have a bijection $A_{y,z} : \pi_2(y,z) \To H_2(M)$.

Indeed, for any $x,y \in \T_\alpha \cap \T_\beta$ (same spin-c structure or not), we can define a map $A_{x,y}: \pi_2 (x,y) \To H_2 (M)$ as follows. If $x,y$ are in the same spin-c class then we use the map just described; if not, we take $A_{x,y}$ to be the zero map.

The collection of maps $A_{x,y}$, over all $x,y \in \T_\alpha \cap \T_\beta$, forms an \emph{additive assignment} similar to \cite[defn. 2.12]{OS04Closed}. That is, we have a collection of functions $A_{x,y}: \pi_2(x,y) \To H_2(M)$, for each $x,y \in \T_\alpha \cap \T_\beta$, such that for any $x,y,z \in \T_\alpha \cap \T_\beta$ and any $\phi \in \pi_2(x,y)$, $\psi \in \pi_2(y,z)$ we have
\[
A_{x,y} (\phi) + A_{y,z} (\psi) = A_{x,z} (\phi * \psi).
\]
In fact, each $\phi \in \pi_2(x,y)$ implicitly includes the data of $x$ and $y$, so we may simply write $A(\phi)$ instead of $A_{x,y}(\phi)$. Thus $A(\phi) + A(\psi) = A(\phi * \psi)$.

\subsection{Twisted coefficient sutured Floer homology}
\label{sec:twisted_coefficient_SFH}

We can now define sutured Floer homology with twisted coefficients, adapting \cite[sec. 8]{OS04Prop} for our purposes. As discussed by Ghiggini--Honda \cite{Ghiggini_Honda08}, extending sutured Floer homology to twisted coefficients is a straightforward generalisation of the closed case.

The twisted chain complex $CF(S, \alpha, \beta)$ is defined from an admissible sutured Heegaard decomposition $(S,\alpha,\beta)$ of $(M,\Gamma)$, as in the untwisted case, but is now freely generated by $\T_\alpha \cap \T_\beta$ over the coefficient ring $\Z[H_2(M)]$. (In the closed case Ozsv\'{a}th--Szab\'{o} define it over $\Z[H^1(M)]$; and obviously $H^1(M) \cong H_2(M)$ by Poincar\'{e} duality in the closed case.) When we wish to emphasise the coefficients we write $CF(S,\alpha,\beta; \Z[H_2(M)])$.

The differential is defined as:
\[
\partial x = \sum_{y \in \mathbb{T}_\alpha \cap \mathbb{T}_\beta} y \sum_{\substack{\phi \in \pi_2(x,y) \\\ \mu(\phi) = 1}} \# \widehat{\M}(\phi) \; e^{A(\phi)} .
\]
This definition is similar to that in the untwisted case, except that we insert an additional factor $e^{A(\phi)} \in \Z[H_2(M)]$, using the additive assignment $A: \pi_2(x,y) \To H_2(M)$, for each homotopy class of Whitney disc $\phi$ counted in the differential. Again $\pi_2(x,y)$ denotes the space of homotopy classes of Whitney discs between $x$ and $y$, $\widehat{\M}(\phi)$ is the moduli space of holomorphic discs in the homotopy class of $\phi$ (modulo the $\R$-action), $\mu(\phi)$ the Maslov class, and $\# \widehat{\M}(\phi)$ a signed count of points. Again the chain complex splits as a direct sum over spin-c structures, $(CF(S,\alpha,\beta),\partial) \cong \bigoplus_\s (CF(S,\alpha,\beta,\s),\partial_\s)$.

Ozsv\'{a}th--Szab\'{o} prove \cite[thm. 8.2]{OS04Prop} in the closed case that the homology $\widehat{HF}(S,\alpha,\beta)$ of the chain complex is independent of all choices involved --- Heegaard decomposition, almost complex structures, etc. --- and so obtain an invariant $\widehat{HF}(M)$ of a closed 3-manifold. The argument carries over immediately in the sutured case to show that the homology of $SFH(S,\alpha,\beta)$ is similarly independent, and thus we obtain an invariant $SFH(M, \Gamma; \Z[H_2(M)])$ of a balanced sutured 3-manifold $(M,\Gamma)$, which is a $\Z[H_2(M)]$-module. As in the untwisted case, it splits as a direct sum over spin-c structures, $SFH(M,\Gamma;\Z[H_2(M)]) \cong \bigoplus_\s SFH(M,\Gamma,\s;\Z[H_2(M)])$.

\subsection{Surface decompositions with twisted coefficients}
\label{sec:decompositions_twisted}

In SQFT we consider cutting and gluing occupied surfaces along arcs; correspondingly, we need to consider cutting and gluing sutured 3-manifolds along surfaces, and its effect on $SFH$ with twisted coefficients.

In \cite{Ju08}, Juh\'{a}sz considered various types of sutured manifold decompositions, and their effects on $SFH$ with integer coefficients. We need one of his results, extended to twisted coefficients, for our purposes. In fact, we need specific details of it, such as lemma \ref{lem:decomp_homology} below, and a version (corollary \ref{cor:Juhasz_decomp_adapted}) specially adapted to our purposes. 

Therefore, we will state a generalisation of Juh\'{a}sz' theorem to twisted coefficients, and consider its proof in some detail, closely following the original proof, though there are significant extra details.

(We note that there are other similar gluing results in the literature; however they are not as explicit as the result we need. In \cite{HKM08} Honda--Kazez--Mati\'{c} give a gluing isomorphism theorem in just the case we need, for $\Z$ coefficients (lemma 7.9). This theorem is based on the same authors' work in \cite{HKM09} (theorems 6.1, 6.2), which gives an alternative proof of a weaker version of Juh\'{a}sz' result.)

First, some language, following \cite{Ju08}. Let $(M, \Gamma)$ be a balanced sutured 3-manifold. A \emph{decomposing surface} $U$ in $(M, \Gamma)$ is a properly embedded oriented surface with no closed components, such that each component of $\partial U \subset \partial M$ either runs along a component of $\Gamma$, or intersects $\Gamma$ transversely \cite[defn. 2.7]{Ju08}. (Note this differs slightly from Juh\'{a}sz' definition: firstly, we have no toroidal sutures; and secondly, we require sutured manifolds to have nonempty sutures on each boundary component, so we do not allow $U$ to have closed components.) Cutting along $U$ (more precisely, removing a product neighbourhood $N(U)$) gives a \emph{sutured manifold decomposition} $(M, \Gamma) \stackrel{U}{\rightsquigarrow} (M',\Gamma')$, where $M' = M \backslash N(U)$. In $\partial M'$, two copies $U'_+$ and $U'_-$ of $U$ exist, along which a normal vector respectively points out of or in to $M'$. The sutures on $M'$ are given by
\begin{align*}
\Gamma' &= (\Gamma \cap M') \cup (U'_+ \cap R_-(\Gamma)) \cup (U'_- \cap R_+(\Gamma)),\\
R_+(\Gamma') &= ((R_+(\Gamma) \cap M') \cup U'_+) \backslash \Gamma',\\
R_- (\Gamma') &= ((R_-(\Gamma) \cap M') \cup U'_-) \backslash \Gamma'.
\end{align*}
(Thus, the sutures on $U'_\pm$ are all boundary-parallel.) We define a map $\iota \; : \; M' \hookrightarrow M$ to re-glue the two copies $U'_\pm$ back together into $U$.

An oriented simple closed curve $C \subset R_\pm$ is \emph{boundary-coherent} if $[C] \neq 0 \in H_1(R_\pm)$, or if $[C]=0 \in H_1(R_\pm)$ and $C$ is oriented as the boundary of its interior (which lies in the oriented $R_\pm$). 

A spin-c structure $\s$ on $(M, \Gamma)$ is \emph{outer} with respect to $U$ if there is a unit vector field $v$ on $M$ representing $\s$ such that for all $p \in U$, $v_p \neq - (\nu_U)$, where $\nu_U$ is the unit normal vector field on $U$ with respect to some Riemannian metric on $M$. Otherwise, $\s$ is \emph{inner}. Let $O_U$ denote the set of outer spin-c structures with respect to $U$.

Only the most general theorem of \cite{Ju08} (theorem 1.3) is applicable in our context, as we often have $H_2(M) \cong H_1(\Sigma)$ nonzero, and $U$ non-separating. Generalised to twisted coefficients it is as follows. This is a precise version of theorem \ref{thm:gluing_rough} above.

\begin{thm}
\label{thm:Juhasz_gluing}
Let $(M, \Gamma)$ be a balanced sutured manifold and let $(M, \Gamma) \stackrel{U}{\rightsquigarrow} (M', \Gamma')$ be a sutured manifold decomposition. Suppose that $U$ is open and that for every component $W$ of $R_\pm(\Gamma)$, the set of closed components of $U \cap W$ consists of parallel oriented boundary-coherent simple closed curves. Then there exist admissible sutured Heegaard decompositions $(S, \alpha, \beta)$ of $(M, \Gamma)$ and $(S', \alpha', \beta')$ of $(M', \Gamma')$ such that the following chain complexes are isomorphic:
\[
\left( \Z[H_2(M)] \otimes_{\Z[H_2(M')]} CF(S', \alpha', \beta'; \Z[H_2(M')]), 1 \otimes \partial \right) \cong \bigoplus_{\s \in O_U} \left( CF (S, \alpha, \beta, \s; \Z[H_2(M)]), \partial \right).
\]
Hence
\[
SFH(M', \Gamma'; \Z[H_2(M)]) \cong \bigoplus_{\s \in O_U} SFH(M, \Gamma, \s; \Z[H_2(M)]).
\]
In particular, $SFH(M', \Gamma'; \Z[H_2(M)])$ is a direct summand of $SFH(M, \Gamma; \Z[H_2(M)])$.
\end{thm}

Here the tensor product is taken with respect to  $\iota_* :  \Z[H_2(M')] \To \Z[H_2(M)]$, which expresses $\Z[H_2(M)]$ as a module over $\Z[H_2(M')]$. Note that we really want an isomorphism involving $SFH(M', \Gamma'; \Z[H_2(M')])$, while the last isomorphism involves $SFH(M', \Gamma'; \Z[H_2(M)])$, which is the homology of the chain complex obtained by using $\Z[H_2(M)]$, as a $\Z[H_2(M')]$-module via $\iota_*$, rather than $\Z[H_2(M')]$ itself. In corollary \ref{cor:Juhasz_decomp_adapted} we obtain an isomorphism of this desired type.

Before proceeding to the proof, we need some preliminary definitions and discussion.

A \emph{Heegaard diagram adapted to a decomposing surface} $U \subset (M, \Gamma)$ is \cite[defn. 4.3]{Ju08} a quadruple $(S, \alpha, \beta, P)$ where:
\begin{enumerate}
\item
$(S, \alpha, \beta)$ is a balanced Heegaard diagram for $(M, \Gamma)$
\item
$P \subset S$ is a ``quasi-polygon'', i.e. a closed subsurface of $S$ (oriented like $S$), whose boundary consists of edges and vertices (however a boundary component can consist of a single closed edge).
\item
The vertices of $P$ (i.e. intersections of adjacent edges) are precisely $P \cap \partial S$.
\item
$\partial P = A \cup B$, where $A$ and $B$ are unions of pairwise disjoint edges of $P$ (i.e. the edges of $\partial P$ alternate between $A$ and $B$).
\item
The edges of $A$ avoid $\beta$, and the edges of $B$ avoid $\alpha$, i.e. $\alpha \cap B = \beta \cap A = \emptyset$.
\item
$U$ is given, up to isotopy through decomposing surfaces, by rounding the corners of $P \times \{\frac{1}{2}\} \cup (A \times [\frac{1}{2},1]) \cup (B \times [0, \frac{1}{2}])$.
\end{enumerate}

A decomposing surface $U$ in $(M, \Gamma)$ is \emph{good} if it has nonempty boundary, and each component of $\partial U$ intersects both $R_+(\Gamma)$ and $R_- (\Gamma)$ \cite[defn. 4.6]{Ju08}. The diagram $(S, \alpha, \beta, P)$ is called \emph{good} if $A$ and $B$ have no closed components. See figure \ref{fig:nice_diagram} for an illustration, where $P$ is a square. With $U$ placed as in (vi) above, we see that:
\begin{enumerate}
\item
$U$ is homeomorphic to $P$. In particular, $U$ deformation retracts onto $P \times \{\frac{1}{2}\}$ by retracting $A \times [\frac{1}{2}, 1]$ and $B \times [0, \frac{1}{2}]$ back to $A \times \{\frac{1}{2}\}$ and $B \times \{\frac{1}{2}\}$ respectively.
\item
$U \cap \Gamma$ is given by $(P \cap \partial S) \times \{\frac{1}{2}\}$.
\item
$U \cap R_+(\Gamma)$ is given by $A \times \{1\} \cup \partial A \times (\frac{1}{2}, 1]$. In particular, a component of $\partial U$ intersects $R_+(\Gamma)$ iff the corresponding component of $\partial P$ contains an edge of $A$.
\item
$U \cap R_-(\Gamma)$ is given by $B \times \{0\} \cup \partial B \times [0, \frac{1}{2})$. In particular, a component of $\partial U$ intersects $R_- (\Gamma)$ iff the corresponding component of $\partial P$ contains an edge of $B$.
\end{enumerate}
See figure \ref{fig:nice_decomposition}. Hence, $U$ is good iff a Heegaard diagram adapted to $U$ is good.

\begin{figure}
\begin{center}
\def\svgwidth{250pt}
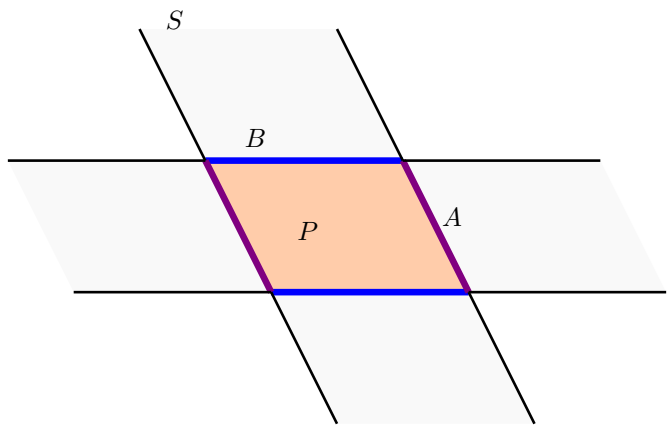
\caption{A Heegaard diagram adapted to a decomposing surface. As $P$ (the darkly shaded region) is a square with $A$  (purple) and $B$ (blue) each consisting of two arcs, the diagram is good.}
\label{fig:nice_diagram}
\end{center}
\end{figure}

\begin{figure}
\begin{center}
\def\svgwidth{300pt}
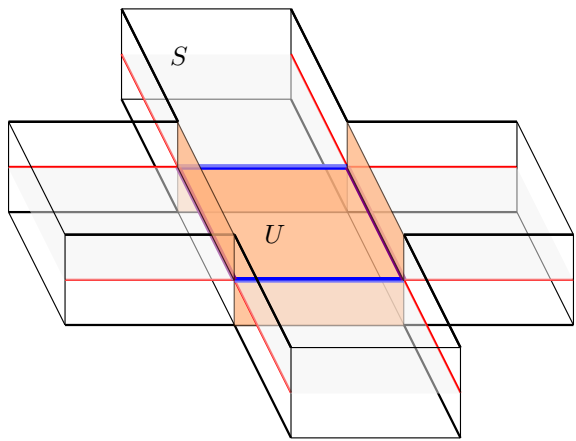
\caption{A good decomposing surface and its relation to the adapted Heegaard diagram of figure \ref{fig:nice_diagram}. Sutures are in red, $S \times \{1/2\}$ light grey, and $U$ is drawn as $P \times \{\frac{1}{2}\} \cup (A \times [\frac{1}{2},1]) \cup (B \times [0, \frac{1}{2}])$. }
\label{fig:nice_decomposition}
\end{center}
\end{figure}

The sutured manifold decomposition $(M, \Gamma) \stackrel{U}{\rightsquigarrow} (M', \Gamma')$ corresponds to an operation on sutured Heegaard diagrams as follows \cite[defn. 5.1]{Ju08}; see figure \ref{fig:decomposition_cut}. From a diagram $(S, \alpha, \beta, P)$ adapted to $U$, we remove $P$ from $S$, and then glue two separate copies $P_A, P_B$ of $P$ back in to $S \backslash P$, gluing $P_A$ to $S \backslash P$ along $A$, and $P_B$ to $S \backslash P$ along $B$. The result is a surface $S' = (S \backslash P) \cup P_A \cup P_B$. There is a natural smooth map $p: S' \To S$ which is the identity on $S' \backslash (P_A \cup P_B) \cong S \backslash P$ and which maps $P_A \cup P_B \to P$ in $2$-to-$1$ fashion. We define curves $\alpha' = \{\alpha'_i\}_{i=1}^k$ and $\beta' = \{\beta'_i\}_{i=1}^k$ on $S'$ as follows. On $P$ there may be both $\alpha$ and $\beta$ curves; we draw the $\alpha$ curves on $P_A$ and the $\beta$ curves on $P_B$. As $\alpha \cap B = \beta \cap A = \emptyset$, combining these curves with the existing $\alpha$ and $\beta$ curves on $S \backslash P$, gives closed curves $\alpha'_i$, $\beta'_i$ as desired. Note $p: S' \To S$ restricts to give homeomorphisms $\alpha'_i \cong \alpha_i$, $\beta'_i \cong \beta_i$. We obtain a sutured Heegaard diagram $(S', \alpha', \beta')$, which is a sutured Heegaard decomposition of $(M', \Gamma')$ \cite[prop. 5.2]{Ju08}. If $(S, \alpha, \beta)$ is admissible then so is $(S', \alpha', \beta')$ \cite[proof of prop. 7.6]{Ju08}.

\begin{figure}
\begin{center}
\def\svgwidth{250pt}
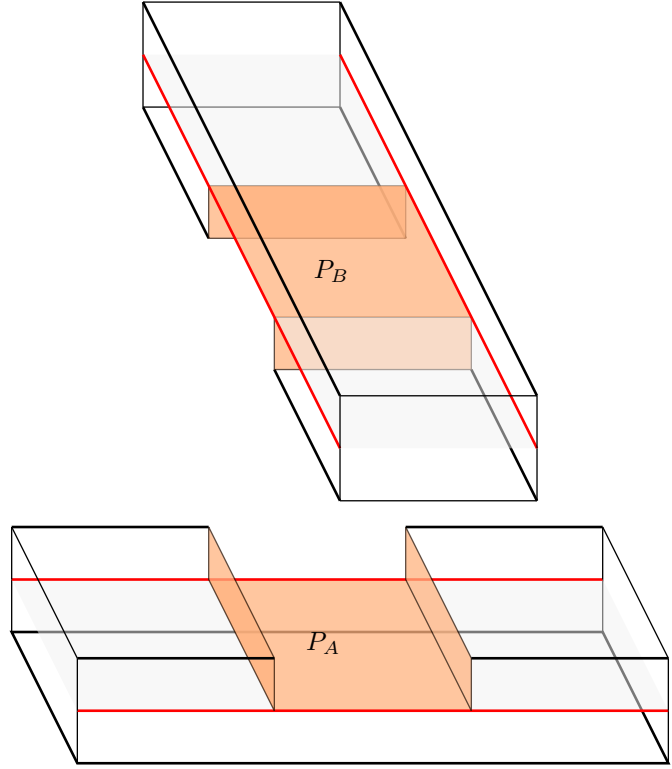
\caption{Decomposition of the sutured 3-manifold along $U$, and its effect on the sutured Heegaard surface: $P$ is split into two copies $P_A$ and $P_B$.}
\label{fig:decomposition_cut}
\end{center}
\end{figure}

In fact, up to homotopy the map $p: S' \To S$ extends to the map $\iota: M' \cong M \backslash N(U) \To M$, in the following sense. We may isotope $S'$ in $M'$ relative to boundary, so that $P_A \subset R_+(\Gamma')$ and $P_B \subset R_- (\Gamma')$, and then the restriction of $\iota$ to $S'$ is given by $p$. In particular, the maps induced by $p$ from $H_1 (S') \To H_1(S)$, $H_1(\alpha') \To H_1(\alpha)$ and $H_1(\beta') \To H_1(\beta)$ are equal to those induced by $\iota$, and so we have the following.

\begin{lem}
\label{lem:decomp_homology}
The map $p: S' \To S$ induces a map
\[
\bar{p} \; : \; H_2(M') \cong H_1(\alpha') \cap H_1(\beta') \To H_1(\alpha) \cap  H_1(\beta) \cong H_2(M)
\]
which is equal to $\iota_*$.
\qed
\end{lem}

Given $(S, \alpha, \beta, P)$ adapted to $U$, each point of $\alpha_i \cap \beta_j$ either lies in $P$ or $S \backslash P$. An intersection point $x \in \T_\alpha \cap \T_\beta$ is called \emph{inner} if $x \cap P \neq \emptyset$, i.e. $x$ contains a point of some $\alpha_i \cap \beta_j$ in $P$; otherwise $x$ is \emph{outer} \cite[defn. 5.4]{Ju08}. We think of the quasi-polygon $P$ as forming an ``inner'' part of $S$. A point $x \in \T_\alpha \cap \T_\beta$ is inner or outer, respectively as its spin-c structure is inner or outer with respect to $U$, i.e. $\s(x) \in O_U$ \cite[lem. 5.5]{Ju08}. 

Following Sarkar--Wang \cite{SaWa}, a sutured Heegaard diagram $(S, \alpha, \beta)$ is \emph{nice} if every internal component of $S \backslash (\alpha \cup \beta)$ is a bigon or square. A diagram $(S, \alpha, \beta, P)$ adapted to a decomposing surface $U$ is \emph{nice} if every component of $S \backslash (\alpha \cup \beta \cup A \cup B)$ disjoint from $\partial S$ is a bigon or square \cite[defn. 6.1]{Ju08}. When Heegaard diagrams are nice, we can count holomorphic curves directly from the combinatorics of the Heegaard diagram.

We can now proceed to the proof of theorem \ref{thm:Juhasz_gluing}. Most of the proof in \cite{Ju08} carries over without change. We summarise the proof, detailing the extra steps required for twisted coefficients.

\begin{proof}
A decomposing surface $U$ satisfying the hypotheses of the theorem can be made good by an isotopy which preserves the set of outer spin-c structures \cite[lem. 4.5]{Ju08}. We can then find a Heegaard diagram $(S, \alpha, \beta, P)$ adapted to $U$ \cite[prop. 4.4]{Ju08}, which is necessarily good, and isotope the $\alpha$ curves to make $(S, \alpha, \beta)$ admissible, keeping $(S, \alpha, \beta, P)$ adapted to $U$ \cite[prop. 4.8]{Ju08}.

Denote the subgroups of $CF(S, \alpha, \beta; \Z[H_2(M)])$ generated by outer and inner intersection points by $O_P$ and $I_P$ respectively; so $CF(S, \alpha, \beta; \Z[H_2(M)] ) \cong O_P \oplus I_P$ and $O_P \cong \bigoplus_{\s \in O_U} CF(S, \alpha, \beta, \s; \Z[H_2(M)])$. As $\pi_2(x,y) \neq \emptyset$ iff $\s(x)=\s(y)$, the differential $\partial$ respects this splitting and we have \cite[cor. 5.8]{Ju08}
\[
CF(S, \alpha, \beta; \Z[H_2(M)] ) = \left( O_P, \partial|_{O_P} \right) \oplus \left( I_P, \partial|_{I_P} \right).
\]

Turning to $(S', \alpha', \beta')$, we see that intersection points $\alpha'_i \cap \beta'_j$ correspond bijectively (via $p: S' \To S$) to the intersection points $\alpha_i \cap \beta_j$ outside $P$; so $\T_{\alpha'} \cap \T_{\beta'}$ corresponds bijectively to outer intersection points of $\T_\alpha \cap \T_\beta$. Hence the generators of $CF(S', \alpha', \beta')$ (over any coefficient ring) correspond bijectively to those of $O_P$.

Using certain permissible moves (isotopies or handleslides), the diagram $(S, \alpha, \beta, P)$ can be made nice, while remaining adapted to $U$ and admissible \cite[thm. 6.4]{Ju08}; moreover $(S, \alpha, \beta)$ is also nice \cite[prop. 7.5]{Ju08}. The holomorphic curves counted in the differentials on $CF(S', \alpha', \beta')$ (over any coefficient ring) and $O_P$ correspond bijectively; indeed this bijection is essentially induced by the map $p: S' \To S$ \cite[prop. 7.6]{Ju08}. Over $\Z$ coefficients, this implies that
\[
SFH(M', \Gamma'; \Z) \cong H(CF(S', \alpha', \beta'; \Z)) \cong H(O_P, \partial |_{O_P}; \Z) \cong \bigoplus_{\s \in O_U} SFH(M, \Gamma, \s; \Z).
\]

Over twisted coefficients, however, $CF(S',\alpha',\beta')$ has coefficient ring $\Z[H_2(M')]$, while $(O_P, \partial|_{O_P}) \subset CF(S,\alpha,\beta)$ has coefficient ring $\Z[H_2(M)]$. Using $\iota_*: H_2(M') \To H_2(M)$, we may tensor $CF(S',\alpha',\beta')$ by $\Z[H_2(M)]$ to obtain two complexes over the same coefficient ring. The two complexes are then free modules over $\Z[H_2(M)]$, with a bijection between generators induced by $p$, hence isomorphic as $\Z[H_2(M)]$-modules:
\[
\Z[H_2(M)] \otimes_{\Z[H_2(M')]} CF(S', \alpha', \beta'; \Z[H_2(M')]) \cong O_P.
\]

It remains to show the differentials are isomorphic. Consider a holomorphic disc with homotopy class $\phi' \in \pi_2(x',y')$ contributing to $\partial x'$, where $x' \in \T_{\alpha'} \cap \T_{\beta'}$ is a generator of $CF(S', \alpha', \beta' ; \Z[H_2(M')] )$. Then $\partial x'$ contains a term $\pm e^{A'(\phi')} y'$, where $A': \pi_2(x',y') \To  H_2(M')$ is the assignment described in section \ref{sec:homotopy_classes_Whitney_discs}. After tensoring with $\Z[H_2(M)]$ using $\iota_*$, the term in $\partial x'$ becomes $\pm e^{\iota_* A'(\phi')} y'$.

The corresponding holomorphic disc contributing to the differential in $O_P$ has a homotopy class $\phi \in \pi_2(x,y)$. Here $x = p^k(x')$, $y = p^k(y')$, where $p^k : \Sym^k S' \To \Sym^k S$ is the map naturally induced by $p: S' \To S$. Then $\partial x$ contains a corresponding term $\pm e^{A(\phi)} y$ with the same sign $\pm$. Here $A: \pi_2(x,y) \To H_2(M)$ is a corresponding assignment. It remains to show that $\iota_* A'(\phi') = A(\phi)$. 

Recall from section \ref{sec:homotopy_classes_Whitney_discs} that $A': \pi_2(x',y') \To H_2(M')$ is defined by using a complete set of paths to give an isomorphism $\pi_2(x',y') \cong \pi_2(z',z')$, where $z' \in \T_{\alpha'} \cap \T_{\beta'}$ is some basepoint in the same spin-c class as $x'$ and $y'$; then we use the boundaries of a periodic domain to define an element of $H_1(\alpha') \cap H_1(\beta') \cong H_2(M')$. Similarly, $A: \pi_2(x,y) \To H_2(M)$ uses a complete set of paths for an isomorphism $\pi_2(x,y) \cong \pi_2(z,z)$, where $z = p^k (z')$, and then takes the boundary of a periodic domain to find an element of $H_1(\alpha) \cap H_1(\beta) \cong H_2(M)$.

We may compose a Whitney disc $u: \Disc \To \Sym^k S'$ from $x'$ to $y'$ with $p^k$ to obtain a Whitney disc $p^k \circ u: \Disc \To \Sym^k S$ from $x$ to $y$. Moreover a homotopy class $\phi \in \pi_2(x',y')$ determines a well-defined homotopy class $p^k \circ \phi \in \pi_2(x,y)$. So we obtain a map $p_*: \pi_2(x',y') \To \pi_2(x,y)$. In fact this map gives a bijection between the holomorphic discs in $CF(S',\alpha',\beta'; \Z[H_2(M')])$ and $O_P$, so $p_* (\phi') = \phi$.

We choose complete sets of paths on $(S', \alpha', \beta')$ and $(S, \alpha, \beta)$ so they correspond under $p^k$. In particular, we take $\theta'_{y'} \in \pi_2(z',y')$, $\theta'_{x'} \in \pi_2(z',x')$ and $\theta_y \in \pi_2(z,y)$, $\theta_x \in \pi_2(z,x)$ so that $p_* \theta'_{x'} = \theta_x$ and $p_* \theta'_{y'} = \theta_y$.

To find $A'(\phi')$, we take $\theta'_{y'} * \phi' * \theta_{x'}^{'-1} \in \pi_2(x',x')$, consider its domain $D(\theta'_{y'} * \phi' * \theta_{x'}^{'-1}) \in D(S',\alpha',\beta')$ and from its boundary obtain an element of $H_1(\alpha') \cap H_1(\beta') \cong H_2(M')$. Similarly, $A(\phi)$ is given via the boundary of the domain $D(\theta_y * \phi * \theta_x^{-1})$.

Now consider $\iota_* A'(\phi')$. By lemma \ref{lem:decomp_homology}, up to homotopy, $\iota$ restricts to $S'$ as $p$. So $\iota_*$ maps $D(\theta'_{y'} * \phi' * \theta_{x'}^{'-1})$ to $D( p_*(\theta'_{y'} * \phi' * \theta_{x'}^{'-1})) = D(\theta_y * \phi * \theta_x^{-1})$. It follows that $\iota_* A'(\phi') = A(\phi)$ as desired.
\end{proof}

We now obtain a useful corollary of the above theorem adapted to our purposes in defining twisted SQFT.
\begin{cor}
\label{cor:Juhasz_decomp_adapted}
Let $(M, \Gamma) \stackrel{U}{\rightsquigarrow} (M', \Gamma')$ be a sutured manifold decomposition, and $(S, \alpha, \beta, P)$ a sutured Heegaard diagram adapted to $U$, such that:
\begin{enumerate}
\item 
$(S, \alpha, \beta, P)$ is good
\item
$(S, \alpha, \beta)$ is admissible
\item
The map $\iota_*: H_2(M') \To H_2(M)$ is injective with image a direct summand of $H_2(M)$.
\item
$P$ contains no points of $\alpha \cap \beta$.
\end{enumerate}
Then taking a tensor product with $\Z[H_2(M)]$ over $\iota_*: \Z[H_2(M')] \To \Z[H_2(M)]$ gives an isomorphism
\[
\Z[H_2(M)] \otimes_{\Z[H_2(M')]} SFH(M', \Gamma'; \Z[H_2(M')])
\cong
SFH(M, \Gamma; \Z[H_2(M)]).
\]
\end{cor}

\begin{proof}
As $(S, \alpha, \beta)$ is admissible, the twisted coefficient chain complex $CF(S, \alpha, \beta; \Z[H_2(M)])$ is well defined, with homology $SFH(M,\Gamma;\Z[H_2(M)])$. As $P$ contains no points of $\alpha \cap \beta$, all intersection points are outer; hence all nontrivial spin-c summands of $SFH(M,\Gamma;\Z[H_2(M)])$ are outer.

In order to have bijections between holomorphic curves counted in differentials, we need \emph{nice} Heegaard diagrams --- which are not assumed. So, as $(S, \alpha, \beta, P)$ is good and $(S, \alpha, \beta)$ is admissible, as above, by permissible moves on $\alpha$ and $\beta$ curves we  obtain a good and nice surface diagram $(S, \alpha_N, \beta_N, P)$ adapted to $U$ such that $(S, \alpha_N, \beta_N)$ is admissible and nice \cite[thm. 6.4, prop. 7.5]{Ju08}. However there might now be intersection points $\alpha_N \cap \beta_N$ in $P$, so we have a possibly nontrivial decomposition into outer and inner complexes:
\[ 
(CF(S, \alpha_N, \beta_N ; \Z[H_2(M)] ), \partial) \cong (O_P, \partial|_{O_P}) \oplus (I_P, \partial|_{I_P}).
\]
As all inner spin-c summands of $SFH(M,\Gamma;\Z[H_2(M)])$ are zero, however, we know $SFH(M,\Gamma;\Z[H_2(M)]) = H(O_P, \partial|_{O_P})$.

Cutting out $P$ and gluing in $P_A, P_B$ we obtain a sutured Heegaard diagram $(S', \alpha'_N, \beta'_N)$ for $(M', \Gamma')$, which is admissible. Then $CF(S', \alpha'_N, \beta'_N ; \Z[H_2(M')])$ is free over $\Z[H_2(M')]$ with generators in bijective correspondence with those of $O_P$. Moreover there is now a bijective correspondence between holomorphic curves counted in the chain complexes $CF(S', \alpha'_N, \beta'_N; \Z[H_2(M')])$ and $(O_P, \partial|_{O_P})$. As above, after tensoring with $\Z[H_2(M)]$ we obtain isomorphic differentials:
\[
(O_P, \partial|_{O_P}) \cong (\Z[H_2(M)] \otimes_{\Z[H_2(M')]} CF(S', \alpha'_N, \beta'_N), 1 \otimes \partial' ).
\]
Hence
\[
SFH(M, \Gamma; \Z[H_2(M)]) = H(O_P, \partial|_{O_P})
\cong H( \Z[H_2(M)] \otimes_{\Z[H_2(M')]} CF(S', \alpha'_N, \beta'_N ), 1 \otimes \partial' ).
\]
Now as $\iota_*$ injects $H_2(M')$ as a direct summand of $H_2(M)$, we have $\ker \partial|_{O_P} \cong \Z[H_2(M)] \otimes_{\Z[H_2(M')]} (\ker \partial')$ and $\im \partial|_{O_P} \cong \Z[H_2(M)] \otimes_{\Z[H_2(M')]} (\im \partial')$. Thus $H(\partial|_{O_P}) \cong \Z[H_2(M)] \otimes_{\Z[H_2(M')]} H(\partial')$; that is,
\begin{align*}
H \left( \Z[H_2(M)] \otimes_{\Z[H_2(M')]} CF(S', \alpha'_N, \beta'_N), 1 \otimes \partial' \right) 
&\cong \Z[H_2(M)] \otimes_{\Z[H_2(M')]} H \left( CF(S', \alpha'_N, \beta'_N), \partial' \right) \\
&\cong \Z[H_2(M)] \otimes_{\Z[H_2(M')]} SFH(M', \Gamma'; \Z[H_2(M')])
\end{align*}
giving the desired result.
\end{proof}

\subsection{Contact elements and gluing with twisted coefficients}
\label{sec:contact_gluing}

As mentioned above in section \ref{sec:SFH_background}, twisted $SFH$ provides invariants of contact structures. A contact structure $\xi$ on $(M, \Gamma)$ gives a contact invariant $c(\xi) \subset SFH(M,\Gamma;\Z[H_2(M)])$, which is the $\Z[H_2(M)]^\times$-orbit of a single element. So $c(\xi)$ is ``a contact element'', ambiguous up to multiplication by units in the coefficient ring. (Alternatively we could take $c(\xi) \in SFH(M, \Gamma; \Z[H_2(M)]) / \Z[H_2(M)]^\times$.)

Contact invariants in Heegaard Floer homology were introduced for closed manifolds in \cite{OSContact}, extended to the sutured case with integer coefficients in \cite{HKM09}, and to twisted coefficients in \cite{Ghiggini_Honda08}. 

Ghiggini--Honda \cite{Ghiggini_Honda08} extended a TQFT-type property of SFH and contact invariants to the case of twisted coefficients. In our language of module homomorphisms we may state it as follows. 
\begin{thm} \cite[Theorem 12]{Ghiggini_Honda08}
\label{thm:twisted_SFH_TQFT_map}
Let $\iota \; : \; (M', \Gamma') \hookrightarrow (M, \Gamma)$ be an inclusion of balanced sutured manifolds, with $M'$ lying in the interior of $M$. Let $\xi$ be a contact structure on $(M \backslash \Int \; M', \Gamma \cup \Gamma')$. Then there is a natural module homomorphism 
\[
\Phi_\xi \; : \; SFH(-M', -\Gamma' ; \Z[H_2(M')] ) \To SFH(-M, -\Gamma; \Z[H_2(M)] ),
\]
over the ring homomorphism $\iota_* : \Z[H_2(M')] \To \Z[H_2(M)]$. For any contact structure $\xi'$ on $(M', \Gamma')$, we have $\Phi_\xi c(\xi') \subseteq c(\xi \cup \xi')$.
\end{thm}

Saying that $\Phi_\xi$ is a module homomorphism over $\iota_*$ means that 
\[
\text{if } A \in H_2(M') \text{ and } x \in SFH(-M',-\Gamma';\Z[H_2(M')]), \text{ then } \Phi_\xi ( e^A \cdot x ) = e^{\iota_* A} \Phi_\xi x.
\]

\section{Heegaard decompositions for twisted SQFT}
\label{sec:TSQFT_Heegaard}
\label{sec:twisted_SFH_on_products}

We finally turn to the precise balanced sutured 3-manifolds we need for twisted SQFT: those of the form $(M, \Gamma) = (\Sigma \times S^1, V \times S^1)$ where $(\Sigma,V)$ is an occupied surface. We give explicit Heegaard decompositions and computations of $SFH$ with twisted coefficients for these manifolds, and demonstrate that they form a twisted SQFT.

Note that the coefficient ring is the group ring of $H_2(M) \cong H_1(\Sigma)$ --- this is immediate from the K\"{u}nneth formula or by deformation retraction of $\Sigma$ onto a 1-skeleton. So $SFH(M,\Gamma)$ is a $\Z[H_1(\Sigma)]$-module.

As mentioned in section \ref{sec:what_this_paper_does}, the Heegaard decompositions discussed here are essentially those of \cite{Massot09}. In proposition 14 of that paper, Massot gives the following calculation of twisted $SFH$:
\[
SFH(-\Sigma \times S^1, - V \times S^1) \cong {\bf V}^{\otimes I(\Sigma,V)} \cong \left( \Z[H_1(\Sigma)]_{(-1)} \oplus \Z[H_1(\Sigma)]_{(1)} \right)^{\otimes I(\Sigma,V)}
\]
where ${\bf V}$ is a free 2-dimensional module on $\Z[H_1(\Sigma)]$, and subscripts $(\pm1)$ graded copies of $\Z[H_1(\Sigma)]$. 

Given a quadrangulation, the number of tensor factors is the same as the number of squares. In fact we will obtain an explicit correspondence between generators of the Heegaard Floer chain complex and basic sutures on the quadrangulated surface. (We did not show this directly in \cite{Me12_itsy_bitsy}, as in that simpler case we could rely on abstract gluing theorems and use contact/suture elements.) Thus, our computations show directly how $SFH$ relates to topology. 

While \cite{Massot09} gives an explicit Heegaard decomposition, details relating to admissibility are omitted; and it seems that these details turn out to be rather subtle. However, they are illuminated by considering quadrangulations of occupied surfaces --- and, additionally, certain ribbon graph-type objects which we call \emph{tape graphs}, which arise as ``spines'' of occupied surfaces. Tape graphs may be interesting in their own right. In any case, we give Heegaard decompositions which are explicitly admissible.

These details, and the demonstration that they form a twisted SQFT, occupy the rest of this paper.

\subsection{Spines of occupied surfaces}
\label{sec:spines_of_occupied_surfaces}

We introduce a notion of \emph{spine} of an occupied surface, derived from a quadrangulation.

\begin{defn}
Let $Q$ be a quadrangulation of an occupied surface $(\Sigma,V)$. The \emph{positive spine} $G_Q^+$ (resp. negative spine $G_Q^-$) of $Q$ is the graph $G$, embedded in $\Sigma$, obtained by joining the positive (resp. negative) vertices of each square of $Q$ by a diagonal.
\end{defn}
In the following we will focus on $G_Q^+$ for definiteness, but similar results obviously apply also to $G_Q^-$.

Note that a spine may have loops and multiple edges. For instance, if $\Sigma$ has positive genus and $N = 1$, then $G_Q^+$ consists of multiple loops at one vertex. 

The vertices of $G_Q^+$ are precisely $V_+$ and there are $N$ of them; and there is one edge for each square of $Q$, hence $I(\Sigma,V) = N - \chi(\Sigma)$ edges. Thus $\chi(G_Q^\pm) = \chi(\Sigma)$; in fact we have the following.

\begin{lem}
The surface $\Sigma$ deformation retracts onto $G_Q^+$. Moreover, cutting $\Sigma$ along $G_Q^+$ produces $N$ discs, each of which has precisely $1$ negative vertex.
\end{lem}

To be precise, cutting $\Sigma$ along $G_Q^+$ produces open discs; taking closure of these gives discs with signed vertices inherited from $(\Sigma,V)$. Within $(\Sigma,V)$, the union of one of these open discs and the adjacent edges and vertices may not be a disc, as in figure \ref{fig:annular_spine}.

\begin{figure}
\begin{center}
\def\svgwidth{100pt}
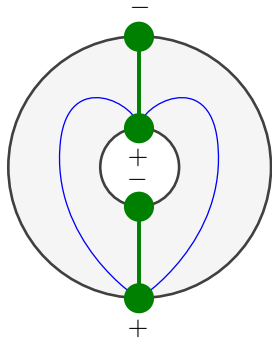
\caption{An annular occupied surface $(\Sigma,V)$ with two vertices on each boundary component. The green decomposition arcs give a quadrangulation with two squares. The blue arcs are diagonals forming the spine $G_Q^+$.}
\label{fig:annular_spine}
\end{center}
\end{figure}

\begin{proof}
Each edge of $G_Q^+$ cuts a square into two triangles; each triangle has two positive vertices and one negative vertex. After cutting along $G_Q^+$, the triangles are only glued along edges which run from $V_+$ to $V_-$. Thus, around each negative vertex $v$ there is a nonempty set of triangles, each of which has two positive vertices, glued along the edges emanating from $v$. As $v \in \partial \Sigma$, these triangles form a disc $D_v$ with one negative vertex $v$.

Further, the two edges adjacent to $v$ on the boundary of $D_v$ lie in $\partial \Sigma$, while all other boundary edges of $D_v$ run through the interior of $\Sigma$. So $D_v$ deformation retracts onto the union of edges between positive vertices, and together these give a deformation retraction of $\Sigma$ onto $G_Q^+$.
\end{proof}

As the only simply connected occupied surfaces are discs, we note that $\Sigma$ is a disc iff $G_Q^+$ is a tree.

Also, the following converse holds.

\begin{lem}
Let $G$ be an embedded graph in $(\Sigma,V)$, with vertices in $V_+$, which cuts $\Sigma$ into discs, each of which contains precisely one negative vertex. Then $G$ is the positive spine of a quadrangulation, unique up to isotopy.
\end{lem}

Again, to be precise, we mean here that $G$ cuts $\Sigma$ into open discs, and taking their closure, we have precisely one negative vertex on each.

\begin{proof}
In each disc, the two edges adjacent to the negative vertex cannot be part of $G$. But all other edges must come from $G$.

In each disc, take a set of arcs joining the negative vertex to each non-adjacent positive vertex, and cutting the disc into triangles, each of which has one negative and two positive vertices. Let the set of all these arcs be $Q$; we show $Q$ is a quadrangulation. Cutting $\Sigma$ along $G$ and $Q$ produces these triangles, and regluing along an edge of $G$ glues two triangles together into an occupied square; as vertices alternate in sign around the boundary of such a square, none of its boundary edges come from $G$. Thus, regluing all edges of $G$ we obtain a collection of squares; so $Q$ cuts $\Sigma$ into squares and is a quadrangulation. The edges of $G$ are precisely the diagonals of those squares joining positive vertices, so $G=G_Q^+$.
\end{proof}

\subsection{Spines as ribbon graphs}
\label{sec:tape_graphs}

The spine of a quadrangulation is a graph, but it has more structure from its embedding in $\Sigma$. In particular, for each vertex $v$ of $G_Q^+$, the edges of $G_Q^+$ incident to $v$ have a natural total ordering --- in anticlockwise order. Precisely, this ordering can be given from the unit tangent space $UT_v \Sigma$, which is homeomorphic to a closed interval; an orientation on $\Sigma$ orders this interval.

This structure is similar to that of a \emph{ribbon} or \emph{fat graph}. But in ribbon graphs, the edges emanating from a vertex only have a \emph{cyclic} ordering, not a total ordering. Note that if the graph contains a loop, then both ends of an edge are at the same vertex; so the total ordering is really on \emph{half-edges}, splitting each edge into two at some arbitrary interior point. Each half-edge is incident to precisely one vertex. While the $d$ edges incident to a vertex $v$ of degree $d$ may not all be distinct, $v$ is always incident to precisely $d$ distinct half-edges. It is these half-edges around $v$ which are totally ordered.

We shall call our variant of a ribbon graph a \emph{tape graph}. 
\begin{defn}
A \emph{tape graph} is a finite graph, together with, for each vertex $v$, a total ordering of the half-edges incident to $v$.
\end{defn}

It will be useful to introduce several definitions for tape graphs.

At each vertex $v$ of a tape graph $G$, the total ordering on half-edges at $v$ has a first and last half-edge. We call the first and last half-edges at $v$ \emph{barrier half-edges} at $v$; we call their edges the \emph{barrier edges} at $v$. Other half-edges at $v$ are called \emph{internal}.

When we draw tape graphs in the plane, we draw the edges in order anticlockwise around each vertex, but use the following device to denote the total ordering. We enlarge (``blow up'') each vertex $v$ and draw it as an interval $I_v$. Edges are drawn in order anticlockwise around the interval, all emanating from one side of the interval, and totally ordered along the interval. So the barrier edges are the first and last edges emanating from the interval. See figure \ref{fig:blowing_up}

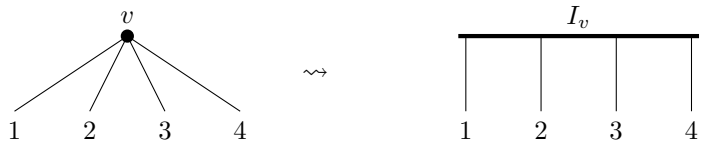
\begin{figure}
\begin{center}

\begin{tikzpicture}[
scale=0.5, 
boundary/.style={ultra thick}
]

\draw (0,2.5) node {$v$};
\fill [black] (0,2) circle (5pt);
\draw (0,2) -- (-3,0);
\draw (-3,-0.5) node {$1$};
\draw (0,2) -- (-1,0);
\draw (-1,-0.5) node {$2$};
\draw (0,2) -- (1,0);
\draw (1,-0.5) node {$3$};
\draw (0,2) -- (3,0);
\draw (3,-0.5) node {$4$};

\draw (5,1) node {$\rightsquigarrow$};

\draw (12,2.5) node {$I_v$};
\draw [boundary] (8.8,2) -- (15.2,2);
\draw (9,2) -- (9,0);
\draw (9,-0.5) node {$1$};
\draw (11,2) -- (11,0);
\draw (11,-0.5) node {$2$};
\draw (13,2) -- (13,0);
\draw (13,-0.5) node {$3$};
\draw (15,2) -- (15,0);
\draw (15,-0.5) node {$4$};

\end{tikzpicture}

\caption{Blowing up a vertex $v$ on a tape graph. Edges incident to $v$ are labelled in order.}
\label{fig:blowing_up}
\end{center}
\end{figure}

There is a well-defined \emph{thickening} of a tape graph. Edges are thickened into rectangles (``tapes''), vertices are thickened into discs, and thickened edges are glued to thickened vertices according to their ordering. A thickened tape graph is a surface with boundary $\Sigma$, and we may regard the boundary as consisting of \emph{sides} of thickened edges. We may thus speak of a \emph{side} of an edge, and a \emph{half-side} of a half-edge. If $\Sigma$ is orientable then sides (and half-sides) of edges inherit orientations from $\partial \Sigma$; the two sides of an edge are oriented in opposite directions. See figure \ref{fig:thickening}.

Furthermore, the half-sides at a vertex $v$ inherit an ordering; for each half-edge $h$ at $v$, as in figure \ref{fig:thickening} we write its two half-sides as $h_-, h_+$ where $h_- < h_+$. If $\Sigma$ is orientable, then for any edge $e$, with half-edges $h,h'$, the two sides of $e$ are $h_- \cup h'_+$ and $h'_- \cup h_+$; moreover the orientation on these half-sides induced by $\partial \Sigma$ points from $h_-$ to $h'_+$ and from $h'_-$ to $h_+$.

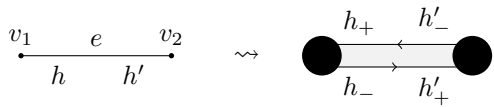
\begin{figure}
\begin{center}

\begin{tikzpicture}[
scale=0.5, 
boundary/.style={ultra thick}
]

\draw (0,0.5) node {$v_1$};
\fill [black] (0,0) circle (2pt);
\draw (0,0) -- node [above] {$e$} (4,0);
\draw (4,0.5) node {$v_2$};
\draw (1,-0.5) node {$h$};
\draw (3,-0.5) node {$h'$};
\fill [black] (4,0) circle (2pt);

\draw (6,0) node {$\rightsquigarrow$};

\fill [gray!10] (8,0.3) -- (12,0.3) -- (12,-0.3) -- (8,-0.3) -- cycle;
\fill [black] (8,0) circle (15pt);
\fill [black] (12,0) circle (15pt);
\draw [->] (12,0.3) -- node [above] {$h'_-$} (10,0.3);
\draw (10,0.3) -- node [above] {$h_+$} (8,0.3);
\draw [->] (8,-0.3) -- node [below] {$h_-$} (10,-0.3);
\draw (10,-0.3) -- node [below] {$h'_+$} (12,-0.3);
\end{tikzpicture}

\caption{Thickening a tape graph. Half-edges $h,h'$ on the graph become half-sides $h_\pm, h'_\pm$.}
\label{fig:thickening}
\end{center}
\end{figure}

We can use this idea to define orientability of a general tape graph $G$ (not necessarily a spine of a quadrangulation). We say $G$ is \emph{oriented} if, for each edge $e$, with half-edges $h,h'$, the two half-sides $h_- \cup h'_+$ form one side of $e$, and $h'_- \cup h_+$ form the other side. This corresponds precisely to when the orderings on edges at each vertex are compatible with an orientation on the thickened graph; a tape graph thickens into an oriented surface with boundary iff it is oriented.

At a vertex $v$ of degree $d$, let the incident half-edges be, in order, $h_1, \ldots, h_d$; so the incident half-sides are, in order, $h_{1-}, h_{1+}, \ldots, h_{d-}, h_{d+}$. We call the first and last half-sides $h_{1-}, h_{d+}$ the \emph{barrier half-sides} at $v$; and the sides to which they belong the \emph{barrier sides} at $v$; other sides and half-sides are called \emph{internal}. 

When $G = G_Q^+$ is the spine of a quadrangulation on $(\Sigma,V)$, a small neighbourhood of a thickened vertex $I_v$ is split by the $d$ incident half-edges into $d+1$ components, which we call \emph{wedges} at $v$. (We regard the half-edges as cutting through $I_v$.) These wedges have a natural total ordering, and we may label them in order $w_0, w_1, \ldots, w_d$. We call $w_0, w_d$ \emph{barrier wedges}, and other wedges \emph{internal}. Each internal wedge $w_i$ is bounded by $h_i$ and $h_{i+1}$; or more precisely, by the half-sides $h_{i+}$ and $h_{(i+1)-}$. As for the barrier wedges, $w_0$ is bounded by the end of the blown-up $I_v$, and $h_{1-}$; while $w_d$ is bounded by the other end of $I_v$, and $h_{d+}$. See figure \ref{fig:wedges}.  

\begin{figure}
\begin{center}

\begin{tikzpicture}[
scale=1, 
boundary/.style={ultra thick}
]

\draw [boundary] (-2.5,0) -- (2.5,0);
\draw (0,.5) node {$I_v$};
\draw (-2,0) -- (-2,-2);
\draw (0,0) -- (0,-2);
\draw (2,0) -- (2,-2);
\draw (-2,-2.5) node {$h_1$};
\draw (-2.5,-1.7) node {$h_{1-}$};
\draw (-1.5,-1.7) node {$h_{1+}$};
\draw (0,-2.5) node {$h_2$};
\draw (-0.5,-1.7) node {$h_{2-}$};
\draw (0.5,-1.7) node {$h_{2+}$};
\draw (2,-2.5) node {$h_3$};
\draw (1.5,-1.7) node {$h_{3-}$};
\draw (2.5,-1.7) node {$h_{3+}$};
\draw (2.5,-0.5) node {$w_3$};
\draw (1,-0.5) node {$w_2$};
\draw (-1,-0.5) node {$w_1$};
\draw (-2.5,-0.5) node {$w_0$};

\end{tikzpicture}

\caption{Wedges at a vertex of degree $3$.}
\label{fig:wedges}
\end{center}
\end{figure}
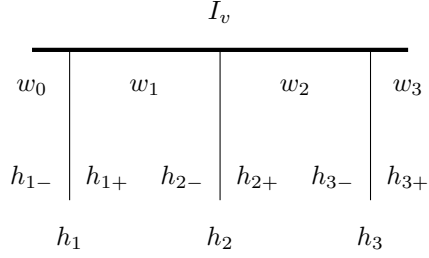

Any spine $G_Q^+$ of a quadrangulation $Q$ on an occupied surface $(\Sigma, V)$ is a tape graph; however that the converse is not true.

For instance, consider the oriented tape graph $G$ shown in figure \ref{fig:tape_graph_not_spine}. Suppose $G$ is the spine of a quadrangulation on an occupied surface $(\Sigma,V)$. Then $\Sigma$ is homeomorphic to the thickening of $G$, and each component of $\partial \Sigma$ contains a positive vertex, so each boundary component of the thickened $G$ must contain a vertex and its adjacent barrier half-sides. As the boundary component $C$ of the thickened $G$ shown fails this condition, $G$ is not the spine of any quadrangulation.

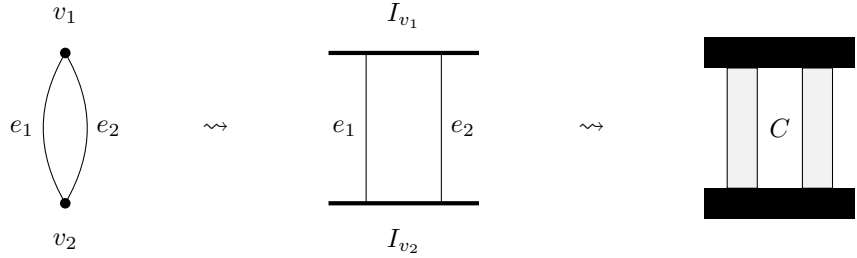
\begin{figure}
\begin{center}

\begin{tikzpicture}[
scale=1, 
boundary/.style={ultra thick}
]

\draw (0,1.5) node {$v_1$};
\draw (0,-1.5) node {$v_2$};
\fill [black] (0,-1) circle (2pt);
\fill [black] (0,1) circle (2pt);
\draw (0,-1) to [bend left=30] node [left] {$e_1$} (0,1);
\draw (0,-1) to [bend right=30] node [right] {$e_2$} (0,1);

\draw (2,0) node {$\rightsquigarrow$};

\draw (4.5,1.5) node {$I_{v_1}$};
\draw (4.5,-1.5) node {$I_{v_2}$};
\draw [boundary] (3.5,1) -- (5.5,1);
\draw [boundary] (3.5,-1) -- (5.5,-1);
\draw (4,1) -- node [left] {$e_1$} (4,-1);
\draw (5,1) -- node [right] {$e_2$} (5,-1);

\draw (7,0) node {$\rightsquigarrow$};

\draw [black, fill = gray!10] (8.8,0.8) -- (9.2,0.8) -- (9.2,-0.8) -- (8.8,-0.8) -- cycle;
\draw [black, fill = gray!10] (9.8,0.8) -- (10.2,0.8) -- (10.2,-0.8) -- (9.8,-0.8) -- cycle;
\fill [black] (8.5,1.2) -- (10.5,1.2) -- (10.5,0.8) -- (8.5,0.8) -- cycle;
\fill [black] (8.5,-1.2) -- (10.5,-1.2) -- (10.5,-0.8) -- (8.5,-0.8) -- cycle;
\draw (9.5,0) node {$C$};

\end{tikzpicture}

\caption{A tape graph which is not the spine of any quadrangulated surface: the graph, its blowup, and thickening.}
\label{fig:tape_graph_not_spine}
\end{center}
\end{figure}

However, the condition found above is necessary and sufficient to be a spine of a quadrangulation.
\begin{prop}
\label{prop:tape_graph_quadrangulation}
An oriented tape graph $G$ is the spine of a quadrangulation $Q$ on an occupied surface iff every boundary component of the thickening of $G$ contains a vertex and its adjacent barrier half-sides.
\end{prop}

Let $C$ be a boundary component of a thickened oriented tape graph $G$. So $C$ is an oriented circle, consisting of a cyclic sequence of $k$ vertices $v_i$ and $k$ sides $s_i$, in order $v_0, s_0, v_1, s_1, \ldots, v_{k-1}, s_{k-1}, v_k = v_0$, where indices are taken modulo $k$. For each $i$, the sides $s_{i-1}, s_i$ are adjacent at the vertex $v_i$, and precisely one of the following occurs:
\begin{enumerate}
\item
the sides $s_{i-1}, s_i$ are consecutive at $v_i$, with $s_{i-1} < s_i$; or
\item
the sides $s_{i-1}, s_i$ are the barrier sides at $v_i$, with $s_{i-1}$ maximal and $s_i$ minimal, so $s_{i-1} > s_i$.
\end{enumerate}
Call $v_i$ a \emph{breakpoint} of $C$ if the second case occurs. 
The proposition can now be restated: an orientable $G$ is a spine iff each boundary component of its thickening contains a breakpoint.

\begin{proof}
The argument above shows that if $G$ is a spine then every boundary component contains a breakpoint.

Conversely, suppose $G$ is orientable and each every boundary component of its thickening contains a breakpoint. Consider a boundary component $C$ of the thickened $G$, which is a cyclic sequence of vertices and sides of edges, containing at least one breakpoint.

Breaking $C$ at its breakpoints splits $C$ into (bona fide, non-cyclic) subsequences (possibly just one subsequence). Each subsequence $C'$ is of the form $v_0, s_0, \ldots, v_{k-1}, s_{k-1}, v_k$, where: $k \geq 1$; $v_0, v_l$ are breakpoints; $s_0$ is the minimal side at $v_0$; $s_{k-1}$ is the maximal side at $v_k$; and for $1 \leq i \leq l-1$, the sides $s_{i-1} < s_i$ are consecutive sides at $v_i$. Every side of every edge in $G$ is contained in precisely one such subsequence.

For such a subsequence $C'$, we take $k$ oriented triangles $\Delta_0, \ldots, \Delta_{k-1}$. Each $\Delta_i$ has one negative vertex $\bar{u}_i$ and two positive vertices; let $\bar{e}_i$ be the positive-to-positive edge opposite $\bar{u}_i$. We glue the triangles into a fan, respecting signs, identifying all the $\bar{u}_i$ into a single vertex $\bar{u}$, and gluing the $\Delta_i$ in anticlockwise order around $\bar{u}$. The oriented boundary of the fan consists of an edge $\bar{e}_0^-$ of $\Delta_0$, the $k$ edges $\bar{e}_0, \ldots, \bar{e}_{k-1}$, and an edge $\bar{e}_{k-1}^+$ of $\Delta_{k-1}$. See figure \ref{fig:triangle_fan}. We glue this fan of triangles to $C'$, gluing each edge $\bar{e}_i$ to the side $s_i$.

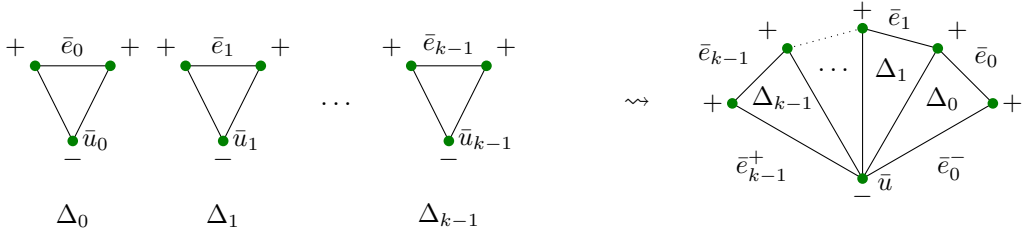
\begin{figure}
\begin{center}

\begin{tikzpicture}[
scale=1, 
boundary/.style={ultra thick}
]

\coordinate [label = above left:{$+$}] (0tl) at (0,0.5);
\coordinate [label = above right:{$+$}] (0tr) at (1,0.5);
\coordinate [label = below: {$-$}, label = right: {$\bar{u}_0$}] (0b) at (0.5,-0.5);
\coordinate [label = above left: {$+$}] (1tl) at (2,0.5);
\coordinate [label = above right:{$+$}] (1tr) at (3,0.5);
\coordinate [label = below: {$-$}, label = right: {$\bar{u}_1$}] (1b) at (2.5,-0.5);
\coordinate [label = above left: {$+$}] (ftl) at (5,0.5);
\coordinate [label = above right:{$+$}] (ftr) at (6,0.5);
\coordinate [label = below: {$-$}, label = right: {$\bar{u}_{k-1}$}] (fb) at (5.5,-0.5);
\coordinate [label = below: {$-$}] (gb) at (11,-1);
\coordinate [label = right: {$+$}] (gv0) at ($ (11,-1) + (30:2) $);
\coordinate [label = above right: {$+$}] (gv1) at ($ (11,-1) + (60:2) $);
\coordinate [label = above: {$+$}] (gv2) at ($ (11,-1) + (90:2) $);
\coordinate [label = above left: {$+$}] (gvsl) at ($ (11,-1) + (120:2) $);
\coordinate [label = left: {$+$}] (gvl) at ($ (11,-1) + (150:2) $);

\draw (0tl) -- node [above] {$\bar{e}_0$} (0tr) -- (0b) -- cycle;
\draw (0.5,-1.5) node {$\Delta_0$};
\draw (1tl) -- node [above] {$\bar{e}_1$} (1tr) -- (1b) -- cycle;
\draw (2.5,-1.5) node {$\Delta_1$};
\draw (4,0) node {$\ldots$};
\draw (ftl) -- node [above] {$\bar{e}_{k-1}$} (ftr) -- (fb) -- cycle;
\draw (5.5,-1.5) node {$\Delta_{k-1}$};

\draw (8,0) node {$\rightsquigarrow$};

\draw (gb) -- node [below right] {$\bar{e}_0^-$} (gv0) -- node [above right] {$\bar{e}_0$} (gv1) -- (gb) (gv1) -- node [above] {$\bar{e}_1$} (gv2) -- (gb) -- (gvsl) -- node [above left] {$\bar{e}_{k-1}$} (gvl) -- node [below left] {$\bar{e}_{k-1}^+$} (gb);
\draw [dotted] (gv2) -- (gvsl);
\draw [xshift = 11 cm, yshift = -1 cm] (105:1.5) node {$\ldots$};
\draw ($ (gb) + (0.3,-0.05) $) node {$\bar{u}$};
\draw ($ (gb) + (45:1.5) $) node {$\Delta_0$};
\draw ($ (gb) + (75:1.5) $) node {$\Delta_1$};
\draw ($ (gb) + (135:1.5) $) node {$\Delta_{k-1}$};

\foreach \point in {(0tl), (0tr), (0b), (1tl), (1tr), (1b), (ftl), (ftr), (fb), (gb), (gv0), (gv1), (gv2), (gvsl), (gvl)}
\fill [green!50!black] \point circle (2pt);

\end{tikzpicture}

\caption{Gluing triangles together into a fan, to construct a surface from its spine.}
\label{fig:triangle_fan}
\end{center}
\end{figure}

Gluing fans of triangles on to $G$ in this way, we obtain an oriented surface $\Sigma$. Precisely one triangle is glued to each side of each edge of $G$, so $\Sigma$ is a triangulated surface. The boundary of $\Sigma$ consists of the vertices of $G$, and the pairs of edges $e_{k-1}^+, e_0^-$ of each fan of triangles, which are joined at the negative vertices $v_-$. Thus the boundary contains vertices $V$ which alternate in sign, so $(\Sigma,V)$ is an occupied surface. The triangles are glued in pairs along positive-to-positive edges (i.e. the edges of $G$) to form a quadrangulation $Q$ of $(\Sigma,V)$, and $G$ is the positive spine of $Q$, $G = G_Q^+$.
\end{proof}

\subsection{Heegaard blocks}
\label{sec:Heegaard_blocks}

We will build Heegaard decompositions of $(\Sigma \times S^1, V \times S^1)$ out of building blocks which we call \emph{Heegaard blocks}, corresponding to the squares of a quadrangulation $Q$. That is, letting $(\Sigma^\square, V^\square)$ denote an occupied square, we construct a Heegaard decomposition of $(\Sigma^\square \times S^1, V^\square \times S^1)$. We will show in section \ref{sec:building_Heegaard} that the Heegaard decompositions of squares of $Q$ piece together to give a Heegaard decomposition of $(\Sigma \times S^1, V \times S^1)$.

We draw $\Sigma^\square \times S^1$ as $\Sigma^\square \times [-1,1]$, with $\Sigma^\square$ horizontal and $[-1,1]$ vertical, and with top and bottom $\Sigma^\square \times \{-1,1\}$ identified, so $S^1 \cong \R/2\Z$.

As noted in section \ref{sec:SQFT_from_SFH}, we may endow $\Sigma^\square \times S^1$ with $V$-sutures or $F$-sutures; the two are equivalent. We use $F$-sutures, so the sutures $\Gamma = F \times S^1$, where $F$ are the midpoints of edges of the square. Points of $F$ are signed as described in section \ref{sec:SQFT_from_SFH}, and sutures $F \times S^1$ oriented according to these signs (sutures go up/down in our pictures as $F$ is positive/negative). The sutures $\Gamma$ split $\partial (\Sigma^\square \times S^1)$ into four annuli; each annulus contains precisely one of the circles $\{v\} \times S^1$ where $v \in V$, so that $V_\pm \times S^1 \subset R_\pm$.

We now describe a \emph{positive Heegaard block}, a Heegaard decomposition of $(\Sigma^\square \times S^1, F \times S^1)$. Let $\gamma_+$ denote the standard basic sutures on the square; they cut $\Sigma^\square$ into three connected components. There are two quarter-circle neighbourhoods $N(V_+)$ of the positive vertices $V_+$, which together comprise $R_+(\gamma_+)$; the remaining component is $R_-(\gamma_+)$. Denote by $e_+$ the diagonal between positive vertices, and let $e'_+ \subset e_+$ be a shorter segment joining the two components of $\gamma_+$. See figure \ref{fig:Heegaard_block_top_view}; this is roughly the ``top view'' of a Heegaard block.

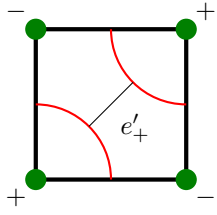
\begin{figure}
\begin{center}

\begin{tikzpicture}[
scale=2, 
boundary/.style={ultra thick}, 
vertex/.style={draw=green!50!black, fill=green!50!black},
suture/.style={thick, draw=red}
]

\coordinate [label = above left:{$-$}] (tl) at (0,1);
\coordinate [label = below left:{$+$}] (bl) at (0,0);
\coordinate [label = below right:{$-$}] (br) at (1,0);
\coordinate [label = above right:{$+$}] (tr) at (1,1);

\draw [boundary] (tl) -- (bl) -- (br) -- (tr) -- cycle;
\draw [suture] (0.5,0) arc (0:90:0.5);
\draw [suture] (0.5,1) arc (180:270:0.5);
\draw [black] ($ (bl) + (45:0.5) $) -- node [below right] {$e'_+$} ($ (tr) + (-135:0.5) $);

\foreach \point in {(tl), (bl), (br), (tr)}
\fill [green!50!black] \point circle (2pt);

\end{tikzpicture}

\caption{The shortened diagonal $e'_+$ between basic sutures on a square.}
\label{fig:Heegaard_block_top_view}
\end{center}
\end{figure}

Now consider $\Sigma^\square \times S^1$. The sutures $\gamma_+$ lift to two \emph{sheets} $\gamma_+ \times S^1$, which split $\Sigma^\square \times S^1$ into three connected components: two \emph{nooks}, which are neighbourhoods $N(V_+) \times S^1$ of $V_+ \times S^1$; and a \emph{bulk} region $R_-(\gamma_+) \times S^1$. The path $e'_+ \times \{0\}$ connects one sheet to the other. A neighbourhood $N(e'_+ \times \{0\})$ of this path forms a \emph{tunnel} connecting the two nooks. 

Define two subsets of $\Sigma^\square \times S^1$.
\begin{enumerate}
\item
$A = \left( R_-(\gamma_+) \times S^1 \right) \backslash N \left( e'_+ \times \{0\} \right)$. That is, $A$ consists of the bulk, with the tunnel drilled out. It is a genus 2 handlebody with $R_- \subset \partial A$.
\item
$B = \left( N(V_+) \times S^1 \right) \cup N \left( e'_+ \times \{0\} \right)$. That is, $B$ consists of the two nooks and tunnel joining them. It is a genus 2 handlebody with $R_+ \subset \partial B$.
\end{enumerate}
We have $A \cap B = \emptyset$ and $\bar{A} \cup \bar{B} = \Sigma^\square \times S^1$. The Heegaard surface $S$ is their common boundary, $S = \partial A \cap \partial B$. We orient $S$ as the boundary of $A$. Thus $S$ consists of the two sheets, with a 1-handle or \emph{tube} attached around $e'_+ \times \{0\}$; it is properly embedded in $\Sigma^\square \times S^1$, with oriented boundary $\partial S = \Gamma = F \times S^1$. It has genus $0$ and $4$ boundary components. See figure \ref{fig:Heegaard_block}.

\begin{figure}
\begin{center}
\def\svgwidth{240pt}
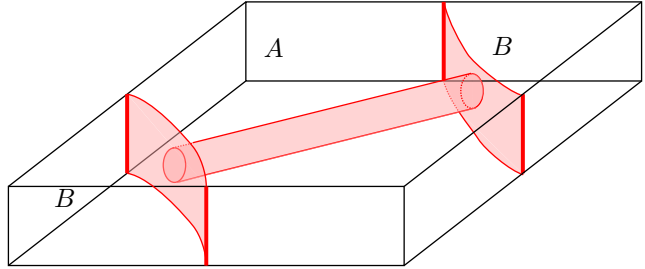
\end{center}
\caption{A Heegaard surface in a Heegaard block.}
\label{fig:Heegaard_block}
\end{figure}

To see the $\alpha$ curve of the Heegaard decomposition, consider the annulus $e'_+ \times S^1$ in the bulk region, lying over the shortened diagonal $e'_+$. If the tunnel has radius $r$, then this annulus intersects $\bar{A}$ in a rectangular-shaped disc $D_\alpha = e'_+ \times [r, 2-r]$ (recall $S^1 \cong \R/2\Z$), shaded light blue in the diagram. We set $\alpha = \partial D_\alpha$. This $\alpha$ is the union of four closed intervals: two vertical intervals $\partial e'_+ \times [r, 2-r]$, and two horizontal intervals $e'_+ \times \{r,2-r\}$. 

The disc $D_\alpha$ (green in figure \ref{fig:pos_Heegaard_block}) splits $A$ into two components ($A'$ and $A''$), which are neighbourhoods of $V_- \times S^1$. In fact, $A'$ deformation retracts onto $S \cup D_\alpha$, as does $A''$. Together, $A'$ and $A''$ deformation retract to cover the whole of $S \cup D_\alpha$; $A'$ covers one side of $D_\alpha$, and $A''$ the other.

The definition of $\beta$ is simpler: $\beta$ a simple closed curve around the tube, bounding a disc $D_\beta$ in the tunnel. In figure \ref{fig:pos_Heegaard_block} $D_\beta$ is shaded in light blue. The disc $D_\beta$ blocks the tunnel, splitting $B$ into two components ($B'$ and $B''$), which are neighbourhoods of the two nooks. Each component deformation retracts onto $S \cup D_\beta$. Together, $B'$ and $B''$ retract to cover the whole of $S \cup D_\beta$; $B'$ covers one side of $D_\beta$, and $B''$ covers the other.

We observe that $\Sigma^\square \times S^1$ is homeomorphic to a thickening $S \times [0,1]$ of $S$, with thickened discs glued to $\alpha \times \{0\}$ and $\beta \times \{1\}$. Under this homeomorphism, $R_-$ is given by $S \times \{0\}$ surgered along $\alpha \times \{0\}$, and $R_+$ is given by $S \times \{1\}$ surgered along $\beta \times \{1\}$. So $(S,\alpha,\beta)$ gives a sutured Heegaard decomposition of $(\Sigma^\square \times S^1, F \times S^1)$. Note $|\alpha \cap \beta| = 2$.

While this construction started with the basic positive sutures $\gamma_+$ on the square, we could just as well perform the same construction starting from $\gamma_-$. We can then define a \emph{negative Heegaard block}, in addition to the positive one described above. For our purposes, however, we only need positive Heegaard blocks as described above.

\begin{figure}
\begin{center}
\def\svgwidth{300pt}
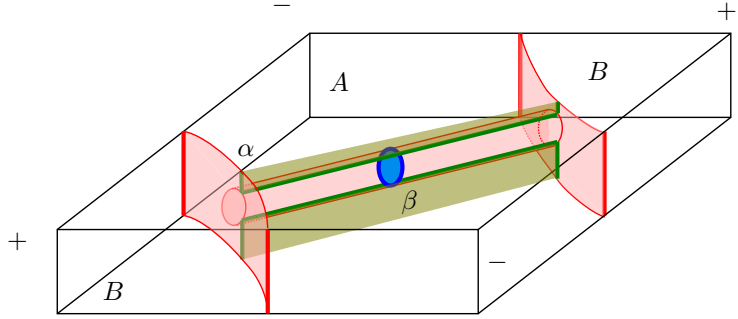
\caption{A positive Heegaard block. The $\alpha$ curve is in green; disc $D_\alpha$ pale green; $\beta$ in blue; $D_\beta$ pale blue.}
\label{fig:pos_Heegaard_block}
\end{center}
\end{figure}

\subsection{Building a Heegaard decomposition from basic sutures with Heegaard blocks}
\label{sec:building_Heegaard}

Suppose we now have an occupied surface $(\Sigma, V, F)$ with quadrangulation $Q$ consisting of squares $(\Sigma^\square_i, V^\square_i, F_i)$. Consider the corresponding balanced sutured 3-manifold with $F$-sutures $(\Sigma \times S^1, F \times S^1)$; as in the previous section, sutures $\Gamma = F \times S^1$ are oriented so that $V_\pm \times S^1 \subset R_\pm$. For each $i$ we have a positive Heegaard block $(S_i, \alpha_i, \beta_i)$, which is a sutured Heegaard decomposition of $(\Sigma^\square_i, F_i \times S^1)$. We now show how these Heegaard blocks piece together to give a Heegaard decomposition of $(\Sigma \times S^1, F \times S^1)$.

We can glue the squares $(\Sigma^\square_i, V^\square_i, F_i)$ to obtain $(\Sigma,V)$. Each time we glue a square to another, we identify a boundary edge $e$ with another; this gluing identifies vertices of the same sign, and a suture endpoint $f$ with another of opposite sign. Combining this gluing with the identity on $S^1$, we can identify the annular face $e \times S^1$ of a Heegaard block with another; this gluing identifies a suture $f \times S^1$ with another with opposite orientation, resulting in a sutured 3-manifold with $2$ fewer sutures.

The Heegaard surfaces $S_i$ also glue together. Each $S_i$ is oriented and properly embedded in $(\Sigma^\square_i \times S^, F_i \times S^1)$ with oriented boundary $F_i \times S^1$. Each gluing identifies one boundary component of some $S_i$ with another boundary component of some $S_j$ by an orientation-reversing map. The resulting surface is thus also oriented and properly embedded, with boundary given by the resulting sutures. Upon gluing, a bulk-minus-tunnel region $A_i$ now connects to another $A_j$; and a $B_i$ now connects to a $B_j$; but $A$ and $B$ regions never intersect.

When finally all squares are glued together into $(\Sigma,V,F)$, and all Heegaard blocks are glued into $(\Sigma \times S^1, F \times S^1)$, we obtain an oriented surface $S$, glued together from the $S_i$, with two sets of pairwise disjoint curves $\alpha = \{\alpha_i\}_{i=1}^{I(\Sigma,V)}$ and $\beta = \{\beta_i\}_{i=1}^{I(\Sigma,V)}$, and oriented boundary $\partial S = \Gamma = F \times S^1$. This surface $S$ splits $\Sigma \times S^1$ into two connected components: a ``bulk" region $A$ containing all the $A_i$, and a ``nook-and-tunnel" region $B$ containing all the $B_i$. 

The bulk region $A$ contains the discs $D_{\alpha_i}$. The deformation retractions of each $A_i$ onto $S_i \cup D_{\alpha_i}$ combine to give a deformation retraction of $A$ onto $S \cup \bigcup_{i=1}^{I(\Sigma,V)} D_{\alpha_i}$. Similarly, the nook-and-tunnel region $B$ contains the discs $D_{\beta_i}$ blocking all the tunnels, and the deformation retractions of each $B_i$ onto $S_i \cup D_{\beta_i}$ together give a deformation retraction of $B$ onto $S \cup \bigcup_{i=1}^{I(\Sigma,V)} D_{\beta_i}$. 

Thus the sutured 3-manifold $(\Sigma \times S^1, F \times S^1)$ is homeomorphic to a thickening $S \times [0,1]$ of $S$, with thickened discs glued to $\alpha \times \{0\}$ and $\beta \times \{1\}$, and oriented sutures $\Gamma = F \times S^1 = \partial S$. Moreover, $R_-$ is given by $S \times \{0\}$ surgered along the $\alpha_i$, and $R_+$ is given by $S \times \{1\}$, surgered along the $\beta_i$. So $(S, \alpha, \beta)$ is a sutured Heegaard decomposition of $(\Sigma \times S^1, F \times S^1)$. 

We have $|\alpha| = |\beta| = I(\Sigma,V)$, and $|\alpha_i \cap \beta_j| = 2 \delta_{ij}$. We record the above.
\begin{prop}
\label{prop:Heegaard_block_decomposition}
Given a quadrangulation $Q$ of an occupied surface $(\Sigma,V,F)$, there is a naturally associated Heegaard decomposition $(S, \alpha, \beta)$ of $(\Sigma \times S^1, F \times S^1)$ constructed out of Heegaard blocks, with $|\alpha| = |\beta| = I(\Sigma,V)$ and $|\alpha_i \cap \beta_j| = 2 \delta_{ij}$.
\qed
\end{prop}

As an aside, we can easily deduce the topology of $S$. Each $S_i$ is a $4$-punctured sphere, so $\chi(S_i) = -2$, so $\chi(S) = -2I(\Sigma,V) = -2N + 2 \chi(\Sigma)$. As $S$ has $2N$ boundary components and is orientable, $S$ has genus $1 - \chi(\Sigma)$.

We also note that, while we used positive Heegaard blocks for this construction, could choose a positive or negative Heegaard block arbitrarily in each square. The $\alpha$ and $\beta$ curves in general would be a mixture of ``bulk blocking" and ``tunnel blocking" curves, and the complementary components of $S$ would each a mixture of bulk and nook-and-tunnel regions.

\subsection{Heegaard decompositions and spines}
\label{sec:Heegaard_and_spines}

We now examine the relationship between the Heegaard decomposition $(S, \alpha, \beta)$ of $(\Sigma \times S^1, F \times S^1)$ described above, associated to the quadrangulation $Q$ of $(\Sigma, V, F)$, and the positive spine $G_Q^+$ of $Q$. The terminology of tape graphs conveniently describes several aspects of the Heegaard decomposition. 

Recall $S \subset \Sigma \times S^1$ is a collection of annular sheets around $V_+ \times S^1$, surgered by 1-handles (tubes) along (shortened) positive-to-positive diagonals of squares of $Q$, i.e. edges of $G_Q^+$.

Under the projection $\pi: \Sigma \times S^1 \To \Sigma$ onto the first coordinate, the sheets of $S$ project to closed arcs (intervals) $I_v$ properly embedded in $\Sigma$, bounding a neighbourhood of each $v \in V_+$. (In fact the $I_v$ together form sutures on $(\Sigma,V,F)$ with positive basic sutures on each square, and are oriented accordingly.) A tube/handle of $S$ running from $v_1$ to $v_2$, where $v_1, v_2 \in V_+$, projects to  a thickened edge in $\Sigma$ from $I_{v_1}$ to $I_{v_2}$.

Thus $\pi(S)$ can be regarded as a blown up, thickened positive spine $G_Q^+$, where each vertex $v$ is enlarged to an interval $I_v$, and edges drawn in order from one side of $I_v$. The orientation on $\Sigma$ totally orders rays emanating from $I_v$ anticlockwise, and this corresponds to the orientation of $I_v$ as a suture.

Each $\alpha$ curve projects under $\pi$ to a closed interval running along a thickened edge, down the middle, splitting it into its two sides. Each $\beta$ curve projects to a closed interval running across a thickened edge, splitting it into its two half-edges.

Consider now the components of $S \backslash (\alpha \cup \beta)$. These correspond precisely to the components of $\pi(S) \backslash (\pi(\alpha) \cup \pi(\beta))$, i.e. the components of the thickened spine $G_Q^+$ cut along $\pi(\alpha)$ and $\pi(\beta)$. Cutting along both $\pi(\alpha)$ and $\pi(\beta)$ curves splits each  thickened edge into its $4$ half-sides, and splits each ``blown-up vertex'' (interval) $I_v$ at the the centre of each thickened edge. The thickened $G_Q^+$ is therefore cut into precisely one piece for each \emph{wedge} of $G_Q^+$. 

Under this correspondence, the boundary-adjacent components (resp. internal) of $S \backslash (\alpha \cup \beta)$ correspond to barrier (resp. internal) wedges of $G_Q^+$.

Also, the wedges of $G_Q^+$ correspond precisely to the \emph{edges} of the quadrangulation $Q$, including both boundary and internal edges. The boundary edges correspond to the barrier wedges of $G_Q^+$, and the internal edges of $Q$ correspond to the internal wedges of $G_Q^+$.

\begin{prop}
\label{prop:wedge_component_bijection}
There are natural bijections between:
\begin{enumerate}
\item
the connected components of $S \backslash (\bigcup \alpha \cup \bigcup \beta)$
\item
the wedges of $G_Q^+$
\item
the edges of the $Q$ (both boundary and internal edges).
\end{enumerate}
Under these bijections, the boundary-adjacent (resp. internal) components of $S \backslash (\bigcup \alpha \cup \bigcup \beta)$ correspond to the barrier (resp. internal) wedges of $G_Q^+$ and the boundary (resp. internal) edges of $Q$.
\qed
\end{prop}

\subsection{Periodic domains}
\label{sec:periodic_domains}

For a Heegaard decomposition $(S, \alpha, \beta)$ of $(M, \Gamma) = (\Sigma \times S^1, F \times S^1)$ associated to a quadrangulation $Q$ of $(\Sigma, V, F)$, we now give an explicit description of the periodic domains.

Recall from section \ref{sec:SFH_background} that the periodic domains of $(S,\alpha,\beta)$ form a subgroup $\P(S, \alpha, \beta) \subset D(S, \alpha, \beta)$ with an isomorphism $\P (S, \alpha, \beta) \To H_1(\alpha) \cap H_1(\beta)$ given by taking boundaries. We also have $H_1(\alpha) \cap H_1(\beta) \cong H_2(M)$ (section \ref{sec:homotopy_classes_Whitney_discs}), and $H_2(M) = H_2(\Sigma \times S^1) \cong H_1(\Sigma)$. And as $\Sigma \simeq G_Q^+$ we have $H_1(\Sigma) \cong H_1(G_Q^+)$. So we have isomorphisms
\[
\P(S, \alpha, \beta) \cong H_1(\alpha) \cap H_1(\beta) \cong H_2(\Sigma \times S^1) \cong H_1(\Sigma) \cong H_1(G_Q^+).
\]

We will describe periodic domains in terms of the tape graph $G_Q^+$. Let $\Tree$ be a maximal tree of $G_Q^+$, so $H_1(G_Q^+)$ has free generators corresponding to the edges of $G_Q^+ \backslash \Tree$. Let $e$ be such an edge, with endpoints $v, w$. In $\Tree$, there is a unique path (without backtracking) from $v$ to $w$; combined with $e$, we obtain a simple loop $l_e$ (oriented arbitrarily) in $G_Q^+$. The $l_e$ form a basis for $H_1(G_Q^+)$.

A loop $l_e$ on $G_Q^+$ traverses some number $s$ of edges $e_i$ and vertices $v_i$: let these be
\[
v_0, e =  e_0, v_1, e_1, \ldots, v_{s-1}, e_{s-1}, v_s = v_0, e_s = e_0,
\]
with indices taken mod $s$. Each $e_i$ inherits an orientation from $l_e$. As $l_e$ is simple, the vertices and edges are all distinct. At each vertex $v_i$, there is a set of well-defined set of internal wedges between the two incident edges $e_{i-1}$ and $e_i$. Let the set of all these internal wedges, taken together over all $v_i$, be $W_e$.

Traversing each $e_i$ according to its orientation, we call the two half-edges of $e_i$ the \emph{start} half-edge (including $v_{i-1}$) and \emph{finish} half-edge (including $v_i$). The two sides of $e_i$ can be described as \emph{left} and \emph{right}. Every time we turn from an edge into a (blown-up) vertex or from a vertex onto an edge, we turn in a definite direction --- left or right. (Think of the edges and ``blown-up'' vertices of the tape graph $G_Q^+$ as roads, intersecting at right angles, along which we ride.) Note that we turn left (resp. right) from an edge $e_{i-1}$ into a vertex $v_i$, iff we proceed along the blown-up $I_{v_i}$ in the decreasing (resp. increasing) direction of the total ordering there, iff we turn left (resp. right) out of $v_i$ onto $e_i$. Correspondingly, at each vertex $v_i$ we may say we are \emph{on the left} (resp. \emph{right}). See figure \ref{fig:left_and_right}.

\begin{figure}
\begin{center}

\begin{tikzpicture}[
scale=1, 
boundary/.style={ultra thick}
]

\draw [boundary] (-2.5,0) -- (2.5,0);
\draw (0,.5) node {$I_{v_1}$};
\draw (-2,0) -- (-3,-2);
\draw (0,0) -- (0,-1);
\draw (2,0) -- (3,-1);
\draw [boundary] (-5.5,-2) -- (-2.5,-2);
\draw (-4,-2.5) node {$I_{v_2}$};
\draw (-5,-2) -- (-5,-1);
\draw [dashed, red, >->] (2.8, -1) -- (2, -0.2) -- (-2, -0.2) -- (-2.8, -1.8) -- (-4.8, -1.8) -- (-4.8,-1);

\end{tikzpicture}

\caption{Traversing these edges and vertices as per the dashed red line, we are on the left at $v_1$ and on the right at $v_2$.}
\label{fig:left_and_right}
\end{center}
\end{figure}

Each wedge $w \in W_e$ lies at some vertex $v_i$, where $l_e$ is on the left or right. Let $L_e$ and $R_e$ denote the wedges of $W_e$ at vertices on the left and right respectively.

By proposition \ref{prop:wedge_component_bijection}, the internal wedges of $G_Q^+$ are in natural bijection with the internal connected components of $S \backslash \left( \bigcup \alpha \cup \bigcup \beta \right)$, i.e. the generators of $D(S, \alpha, \beta)$. Write $D_w$ for generator of $D(S,\alpha,\beta)$ corresponding to the wedge $w$. Define a domain $D_e$ by
\[
D_e = \sum_{w \in R_e} D_w - \sum_{w \in L_e} D_w \in D(S, \alpha, \beta).
\]

\begin{lem}
For each edge $e$ of $G_Q^+ \backslash \Tree$, $D_e$ is a periodic domain.
\end{lem}

To prove the lemma, we introduce orientations and terminology for $\alpha$ and $\beta$ curves. 

Each oriented edge $e_i$ of the loop $l_e$ has an $\alpha$ and a $\beta$ curve around the corresponding tube of $S$; we denote them $\alpha_{e_i}$ and $\beta_{e_i}$. Recall (section \ref{sec:Heegaard_blocks}) that an $\alpha$ curve consists of $4$ arcs in $\Sigma \times S^1$: $2$ vertical arcs $\partial e'_+ \times [r,2-r]$ and $2$ horizontal arcs $e'_+ \times \{r,2-r\}$ (where $e'_+$ is a shortened diagonal of a square and $r$ is tube radius). Orient $\alpha_{e_i}$ so that the $e'_+ \times \{r\}$ arc is oriented like $e_i$, i.e. from $v_i \times \{r\}$ to $v_{i+1} \times \{r\}$. Orient $\beta_{e_i}$ so that at each point of $\alpha_i \cap \beta_i$, tangent vectors to $(\alpha_i, \beta_i)$ form an oriented basis for the tangent space to $S$, agreeing with the orientation on $S$ (i.e. as the boundary of the ``bulk-minus-tunnels'' region $A$). For those $\alpha, \beta$ curves corresponding to edges not on $l_e$, we orient them arbitrarily. Note these orientations depend on the loop $l_e$.

Each $\alpha_{e_i}$ is cut by $\beta_{e_i}$ into two arcs, corresponding to the two half-edges (start and finish) of $e_i$: call them the \emph{start arc} $\alpha_{e_i}^S$ (closer to $v_{i-1}$) and \emph{finish arc} $\alpha_{e_i}^F$ (closer to $v_i$) respectively. Each $\beta_{e_i}$ is cut by $\alpha_{e_i}$ into two arcs, corresponding to the two sides (left and right) of $e_i$: call them the \emph{left arc} $\beta_{i_i}^L$ and \emph{right arc} $\beta_{e_i}^R$ respectively.

\begin{proof}
We compute $\partial D_e$, showing it is an integer combination of closed $\alpha$ and $\beta$ curves (not arcs).

The wedges of $W_e$ at a vertex $v_i$ on the loop $l_e$ are the wedges between edges $e_{i-1}$ and $e_i$. The loop $l_e$ is on the left (resp. right) at $v_i$ iff $e_i < e_{i-1}$ (resp. $e_{i-1} < e_i$). The sum of these wedges has boundary given by the finish arc $\alpha_{e_{i-1}}^F$, the start arc $\alpha_{e_i}^S$, an arc of $\beta_{i-1}$ and an arc of $\beta_i$ (left or right respectively as $v_i$ is on the left or right), and full $\beta$ curves $\beta_{e'}$ for each edge $e'$ incident to $v_i$ between $e_{i-1}$ or $e_i$. The $\beta_{e'}$ lie on edges that are not part of the loop $l_e$, so are oriented arbitrarily. See figure \ref{fig:on_the_left}.

\begin{figure}
\begin{center}
\def\svgwidth{400pt}
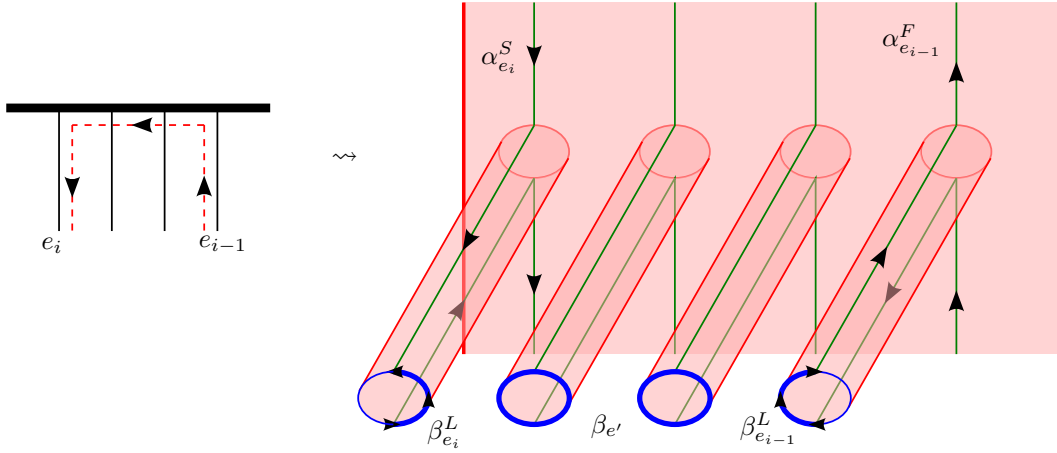
\caption{Domains associated to a vertex in the loop $l_e$; the vertex here is on the left.}
\label{fig:on_the_left}
\end{center}
\end{figure}

Recall $S$ is oriented as the boundary of the bulk region, and domains $D_w$ inherit orientations from $S$. If $v_i$ is on the right we obtain
\[
\partial \sum_{w \in W_e \text{ at } v_i} D_w = \alpha_{e_{i-1}}^F - \beta_{e_{i-1}}^R + \sum_{e_{i-1} < e' < e_i \text{ at } v_i} \left( \pm \beta_{e'} \right) + \alpha_{e_i}^s + \beta_{e_i}^R.
\]
If $v_i$ is on the left we obtain
\[
- \partial \sum_{w \in W_e \text{ at } v_i} D_w = \alpha_{e_{i-1}}^F + \beta_{e_{i-1}}^L + \sum_{e_i < e' < e_{i-1} \text{ at } v_i} \left( \pm \beta_{e'} \right) + \alpha_{e_i}^S - \beta_{e_i}^L.
\]
Now $\partial D_e$ is the sum of the expressions above over all vertices $v_i$ on $l_e$. Each $\alpha_{e_i}^S$ at $v_i$ is summed with $\alpha_{e_i}^F$ at $v_{i+1}$ to give $\alpha_{e_i}$. As for the $\beta$-arcs: when an edge $e_i$ runs between vertices $v_i, v_{i+1}$ which are both on the right (resp. left) we obtain terms $\beta_{e_i}^R$ (resp. $-\beta_{e_i}^L$) at $v_i$ and $-\beta_{e_i}^R$ (resp. $\beta_{e_i}^L$) at $v_{i+1}$ which cancel; when $v_i$ is on the right and $v_{i+1}$ on the left (resp. $v_i$ left and $v_{i+1}$ right), we obtain $\beta_{e_i}^R$ at $v_i$ and $\beta_{e_i}^L$ at $v_{i+1}$ (resp. $-\beta_{e_i}^L$ and $-\beta_{e_i}^R$) which sum to $\beta_{e_i}$ (resp. $-\beta_{e_i}$).

Thus we obtain a sum of closed $\alpha$ and $\beta$ curves, and $D_e$ is a periodic domain.
\end{proof}

The isomorphism $\P(S,\alpha,\beta) \To H_2(\Sigma \times S^1)$ is given by taking a periodic domain and capping off its boundary $\alpha$ and $\beta$ curves with discs. Now the domain $D_e$ has boundary, as computed above, consisting of $\alpha$ curves running along the edges of $l_e$, $\beta$ curves running along edges of $l_e$ where we change sides (left-to-right or right-to-left), and $\beta$ curves on edges lying between $e_{i-1}, e_i$ at each vertex $v_i$. When these curves are capped off with discs, we obtain a surface isotopic to $l_e \times S^1$. As the $l_e$ form a basis for $H_1(\Sigma)$, the $l_e \times S^1$ form a basis for $H_2(\Sigma \times S^1)$. We thus have the following.

\begin{prop}
As $e$ ranges over the edges of $G_Q^+ \backslash \Tree$, the periodic domains $D_e$ form a basis for $\P(S, \alpha, \beta)$.
\qed
\end{prop}

\subsection{Admissibility}
\label{sec:admissibility}

Given occupied $(\Sigma, V)$ with quadrangulation $Q$, we now have a Heegaard decomposition $(S, \alpha, \beta)$ and, after choosing a maximal tree $\Tree$ in $G_Q^+$, we have an explicit basis of $\P(S, \alpha, \beta)$.

Such a Heegaard decomposition is in general not admissible. In particular, if there exists an edge $e$ of $G_Q^+ \backslash \Tree$ such that the corresponding loop $l_e$ has all vertices on the same side, then the computations above show $D_e$ is a periodic domain with all coefficients of the same sign, violating admissibility. This problem cannot be remedied simply by choosing $\Tree$ more judiciously: for instance, the $(\Sigma,V)$ and $Q$ of figure \ref{fig:non_admissible}, with positive spine $G_Q^+$ shown, has only one maximal tree $\Tree$ (with no edges), and every loop $l_e$ has all vertices on the same side.

\begin{figure}
\begin{center}

\begin{tikzpicture}[
scale=1.7, 
boundary/.style={ultra thick}, 
vertex/.style={draw=green!50!black, fill=green!50!black},
suture/.style={thick, draw=red},
decomposition/.style={thick, draw=green!50!black}, 
vertex/.style={draw=green!50!black, fill=green!50!black},
>=triangle 90, 
decomposition glued1/.style={thick, draw=green!50!black, fill=green!50!black, postaction={nomorepostaction,decorate, decoration={markings,mark=at position 0.5 with {\arrow{>}}}}},
decomposition glued2/.style={thick, draw = green!50!black, fill=green!50!black, postaction={nomorepostaction, decorate, decoration={markings,mark=at position 0.5 with {\arrow{>>}}}}},
spine/.style={draw=blue}
]

\coordinate [label = above left: {$-$}] (tl) at (0,1);
\coordinate [label = below left: {$+$}] (bl) at (0,0);
\coordinate [label = above: {$+$}] (tm) at (1,1);
\coordinate [label = below: {$-$}] (bm) at (1,0);
\coordinate [label = above right: {$-$}] (tr) at (2,1);
\coordinate [label = below right: {$+$}] (br) at (2,0);

\fill [gray!10] (tl) -- (tr) -- (br) -- (bl) -- cycle;
\draw [boundary] (tl) -- (tm);
\draw [boundary] (bm) -- (br);
\draw [decomposition] (tm) -- (bm);
\draw [decomposition glued1] (bl) -- (tl);
\draw [decomposition glued1] (br) -- (tr);
\draw [decomposition glued2] (bl) -- (bm);
\draw [decomposition glued2] (tm) -- (tr);
\draw [spine] (bl) -- (tm) -- (br);
\draw (1,-0.5) node {$(\Sigma,V,Q)$};

\draw [spine] (4,0.7) arc (45:300:0.5);
\draw [spine] (4,0.7) arc (45:-30:0.5);
\draw [spine] (4,0.7) arc (135:495:0.5);
\draw (4,-0.5) node {$G_Q^+$};
\fill [black] (4,0.7) circle (2pt);

\draw [boundary] (5.5,0.75) -- (8,0.75);
\draw [spine] (6,0.75) arc (180:360:0.5);
\draw [spine] (6.5,0.75) arc (180:225:0.5);
\draw [spine] (7.5,0.75) arc (0:-115:0.5);
\draw (6.5,-0.5) node {$G_Q^+$ blown up};

\foreach \point in {tl, bl, tm, bm, tr, br}
\fill [green!50!black] (\point) circle (2pt);

\end{tikzpicture}

\caption{Quadrangulated $(\Sigma,V)$ with inadmissible Heegaard decomposition. The left diagram has edge markings indicating identifications; $(\Sigma,V)$ is a once-punctured torus with two vertices.}
\label{fig:non_admissible}
\end{center}
\end{figure}
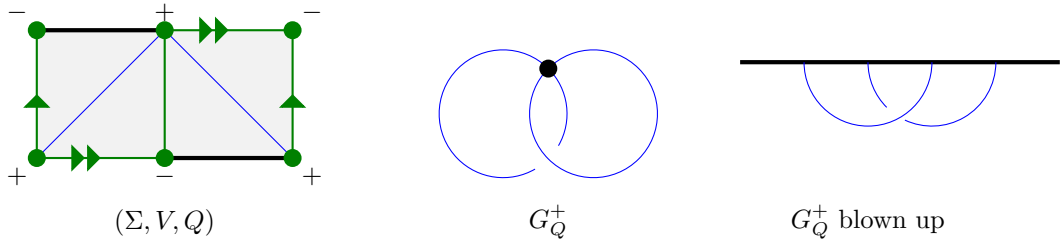

Juh\'{a}sz in \cite{Ju09} gives a general algorithm to render a balanced sutured Heegaard diagram admissible. We give an algorithm which is similar, but gentler on Heegaard curves and specific to our context.

The idea is to isotope $\beta$ curves so as to ensure that on each side of each $\beta$ curve there is a boundary-adjacent component of $S \backslash \left( \bigcup \alpha \cup \bigcup \beta \right)$. This ensures admissibility by the following argument. Any periodic domain $D$ is an integer combination of internal components of $S \backslash \left( \bigcup \alpha \cup \bigcup \beta \right)$, and has boundary $\partial D$ containing at least one $\beta$ curve $\beta_i$ with nonzero coefficient (otherwise the $\alpha$ curves would be linearly dependent in $H_1(S)$). As $D$ only consists of internal components, while $\beta_i$ has boundary-adjacent components on either side, $D$ must have positive and negative coefficients. Thus we have an admissible Heegaard diagram.

We isotope each $\beta_i$ as follows. Each $\beta_i$ in $(S, \alpha, \beta)$ is adjacent to $4$ components of $S \backslash \left( \bigcup \alpha \cup \bigcup \beta \right)$, two components on each side. These components may be boundary-adjacent or internal, corresponding to barrier or internal wedges of $G_Q^+$. For each side $\sigma$ of $\beta_i$ adjacent to only internal components, we isotope $\beta_i$ in the direction $\sigma$ by a finger move along an arc $\zeta_i^\sigma$ which we now define.

The arc $\zeta_i^\sigma$ starts on $\beta_i$ in the direction $\sigma$, and ends in a boundary-adjacent component of $S \backslash \left( \bigcup \alpha \cup \bigcup \beta \right)$. Recall that $S$ consists of sheets and tubes, and each tube has one $\alpha$ and one $\beta$ curve. Proceeding from $\beta_i$ in the direction $\sigma$, the tube is split by $\alpha_i$ into a left and right side. Take $\zeta_i^\sigma$ to run along the left side of the tube, turn left onto the sheet there, and proceed along that sheet, avoiding other tubes, to a boundary-adjacent component of $S \backslash \left( \bigcup \alpha \cup \bigcup \beta \right)$. (Projecting down to a thickened blown-up $G_Q^+$, $\zeta_i^\sigma$ proceeds from $\beta_i$ down the left side of a half-edge, turns left at the end vertex (blown up to an interval) and proceeds along the corresponding interval, almost to its end.) We may choose the arcs $\zeta_i^\sigma$, over all $i$ and $\sigma$, to be disjoint. Note $\zeta_i^\sigma \cap \beta_j = \emptyset$ for $i \neq j$; however $\zeta_i^\sigma$ may intersect several $\alpha_j$. For each $\zeta_i^\sigma$, we isotope $\beta_i$ by a finger move around $\zeta_i^\sigma$. See figure \ref{fig:make_admissible}.

\begin{figure}
\begin{center}
\def\svgwidth{400pt}
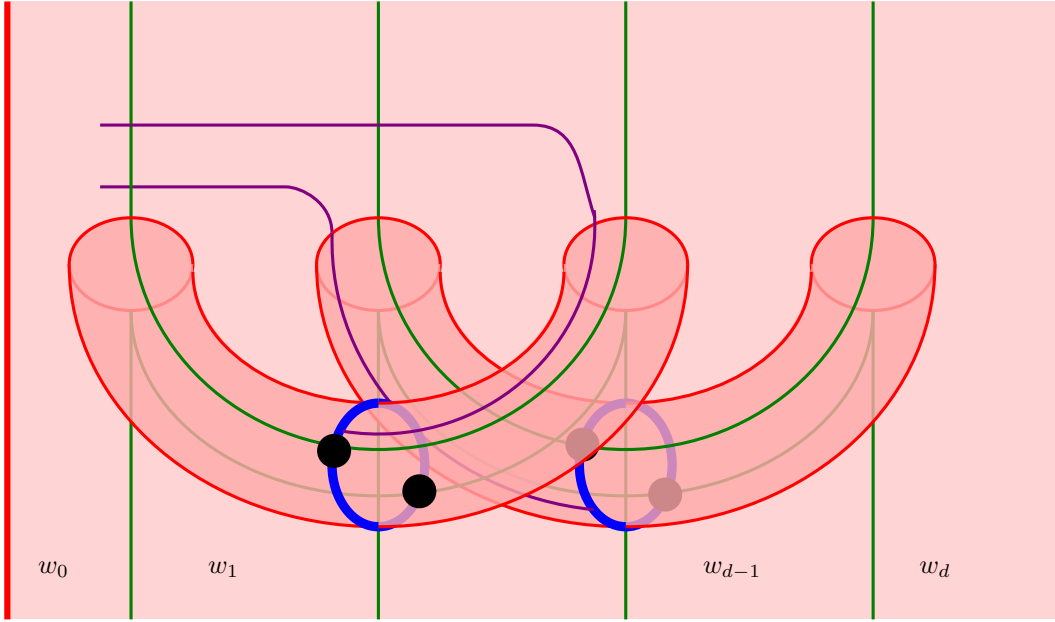
\caption{We make the Heegaard diagram of figure \ref{fig:non_admissible} admissible. The $\alpha$ curves are drawn in green, original $\beta$ curves in blue, and $\zeta$ curves in purple. We isotope the $\beta$ curves by finger moves along the $\zeta$ curves. The result is shown in figure \ref{fig:annulus_cut_decomposition}.}
\label{fig:make_admissible}
\end{center}
\end{figure}

Call the resulting curves of this isotopy $\beta^\Sigma = \{\beta^\Sigma_i\}_{i=1}^k$. As each $\beta^\Sigma_i$ has, on each side, boundary-adjacent components of $S \backslash \left( \bigcup \alpha \cup \bigcup \beta \right)$, we immediately have the following proposition.
\begin{prop}
\label{prop:Heegaard_block_decomposition_admissible}
$(S, \alpha, \beta^\Sigma)$ is an admissible balanced sutured Heegaard decomposition of $(\Sigma \times S^1, F \times S^1)$.
\qed
\end{prop}

Note we have performed more isotopies than is strictly necessary. For instance, if $G_Q^+$ is a tree, i.e. $\Sigma$ is a disc, then there are no loops $l_e$, no periodic domains and the Heegaard diagram is automatically admissible; yet our construction may isotope the $\beta$ curves even in this case.

Also note that there is nothing special about the left direction; we could have constructed the $\zeta^\sigma_i$ by ``always turning right" rather than left. However the choice of a definite direction assists us in decomposing surfaces, as we see next.

Propositions \ref{prop:Heegaard_block_decomposition} and \ref{prop:Heegaard_block_decomposition_admissible} together amount to proposition \ref{prop:Heegaard_block_decomposition_rough}

\subsection{Decomposing and computing SFH}
\label{sec:computing_SFH}

With the admissible balanced sutured Heegaard decomposition $(S, \alpha, \beta^\Sigma)$, we can form the chain complex $CF(S, \alpha, \beta^\Sigma)$ and the differential $\partial$. We now compute the homology by successively decomposing $S$, using the results of section \ref{sec:decompositions_twisted}, in particular corollary \ref{cor:Juhasz_decomp_adapted}.

First we need a lemma about the $\beta^\Sigma$ curves. Recall the projection $\pi: \Sigma \times S^1 \To \Sigma$ can be regarded as flattening the tubes and sheets of the Heegaard surface $S$ into the thickened edges and blown-up vertices of $G_Q^+$. Each $\zeta^\sigma_i$ projects to an embedded arc from the midpoint of an edge, to the blown-up vertex (interval) at the end of that edge, turning left and proceeding along that interval, across any intervening wedges, to a barrier wedge. The observation is that certain wedges are never crossed by any $\zeta^\sigma_i$, hence never by any $\beta^\Sigma_i$.

Let $v$ be a vertex of $G_Q^+$ of degree $d \geq 2$, with edges and wedges $w_0 < e_1 < w_1 < e_2 < \cdots < e_d < w_d$, where $w_0, w_d$ are barrier wedges and $w_1, \ldots, w_{d-1}$ are internal. Figure \ref{fig:make_admissible} provides an example.

\begin{lem}
\label{lem:disjoint_wedge}
For any vertex $v \in G_Q^+$ of degree $d \geq 2$, the internal wedge $w_{d-1}$ is disjoint from $\pi(\beta^\Sigma)$.
\end{lem}

\begin{proof}
Any $\beta^\Sigma_i$ with projection in $w_{d-1}$ must arise from a $\zeta^\sigma_i$ with projection passing through $w_{d-1}$. Because of the ``always turn left'' construction, such a $\zeta^\sigma_i$ must begin on $e_d$, proceed on the side $\sigma$ towards $v$ and turn left through $w_{d-1}$. But the $\beta$ curve on $e_d$ is adjacent to the boundary-adjacent component of $S \backslash \left( \bigcup \alpha \cup \bigcup \beta \right)$ corresponding to the barrier wedge $w_d$ on the same side $\sigma$, and so no such $\zeta^\sigma_i$ is ever drawn.
\end{proof}

If some vertex $v$ of $G_Q^+$ has degree $d \geq 2$, we now show how to decompose the occupied surface $(\Sigma, V, F)$, quadrangulation $Q$, sutured manifold $(M, \Gamma) = (\Sigma \times S^1, F \times S^1)$ and Heegaard decomposition $(S, \alpha, \beta^\Sigma)$ at the vertex $v$. We will decompose along the wedge $w_{d-1}$ at $v$. The internal wedge $w_{d-1}$ corresponds to an internal edge $a$ of the quadrangulation $Q$, from a positive to negative vertex, and cutting $(\Sigma,V,F)$ along $a$ produces an occupied surface $(\Sigma',V',F')$ with the same index but gluing number decreased by $1$, and a naturally associated quadrangulation $Q'$. We will decompose $(\Sigma,V,F)$ to $(\Sigma',V',F')$, $Q$ to $Q'$, and $(\Sigma \times S^1, F \times S^1)$ to $(\Sigma' \times S^1, F' \times S^1)$. However the requirements of section \ref{sec:decompositions_twisted} are a little subtle; the most obvious decomposing surface $a \times S^1$ is not good.

We thus isotope $a \times S^1$ to obtain a good decomposing surface $U$. Recall to be good, each component of $\partial U$ must intersect both $R_+$ and $R_-$. Let the boundary vertices of $a$ be $v_\pm \in V_\pm$, so $a \times S^1$ has boundary curves $v_\pm \times S^1 \subset R_\pm$. We isotope both curves $v_\pm \times S^1$ by finger moves, pushing a sub-arc of $v_\pm \times S^1$ in the direction of (oriented) $\partial \Sigma$, to intersect the next suture of $F \times S^1$ in two points. These two isotoped curves now intersect $\Gamma$ transversely in two points each, and bound an annulus $U$ isotopic to $a \times S^1$, which is a good decomposing surface. The sutured manifold decomposition of $(\Sigma \times S^1, F \times S^1)$ along $U$ is a balanced sutured 3-manifold naturally homeomorphic to $(\Sigma' \times S^1, F' \times S^1)$. See figure \ref{fig:decomposing_surface}.

\begin{figure}
\begin{center}
\def\svgwidth{400pt}
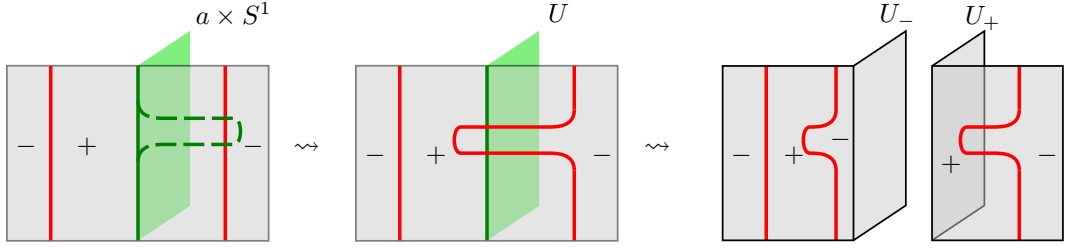
\caption{Isotopy of $a \times S^1$ to the good decomposing surface $U$. First isotope $v_\pm \times S^1$ (in green) to intersect sutures (left). This is equivalent to isotopy of sutures to intersect the surface (centre). Cutting along $U$ gives the result shown on right.}
\label{fig:decomposing_surface}
\end{center}
\end{figure}

We now find a Heegaard diagram $(S, \alpha, \beta^\Sigma, P)$ adapted to $U$. 

In $G_Q^+$, we have a blown-up vertex (interval) $I_v$; take a sub-interval $\I$ between the edges $e_{d-1}$ and $e_d$, i.e. in the wedge $w_{d-1}$. Then $\pi^{-1}(\I)$ is an embedded annulus $\mathcal{A}$ in $S$, which lies in the sheet of $S$ corresponding to the vertex $v$, between the tubes corresponding to $e_{d-1}$ and $e_d$; moreover, $\mathcal{A}$ is disjoint from $\alpha$ and (by lemma \ref{lem:disjoint_wedge}) from $\beta^\Sigma$. Let the two boundary curves of $\mathcal{A}$ be $c_-$, closer to $e_{d-1}$, and $c_+$, closer to $e_d$. 

We now isotope the $c_\pm$ by finger moves along arcs $\eta_\pm$, so that they still bound an annulus in $S$ isotopic to $\mathcal{A}$. After this isotopy, $c_\pm$ will become as a sequence of two arcs $A_\pm \cup B_\pm$, where $A_\pm$ avoids $\beta^\Sigma$ curves and $B_\pm$ avoids $\alpha$ curves.

Let $\eta_-$ be an arc which runs from $c_-$, along the sheet of $S$ at $v$, avoiding all tubes and $\beta^\Sigma_i$ curves, to the boundary component of $S$ at the initial wedge $w_0$. (This $\eta_-$ may intersect $\alpha$ curves.) We isotope $c_-$ along $\eta_-$ by a finger move to obtain a polygonal curve $A_- \cup B_-$ with vertices on $\partial S$, as shown in figure \ref{fig:annulus_cut_decomposition}. The arc $B_-$ is a short boundary-parallel arc; $A_-$ consists of most of $c_-$, and two arcs parallel to $\eta_-$. So $A_-$ avoids all $\beta^\Sigma$ curves, and $B_-$ avoids all $\alpha$ curves.

In the thickened $G_Q^+$, there is a boundary component which includes the wedge $w_{d-1}$, and the two adjacent sides of $e_{d-1}$ and $e_d$. As we trace out this boundary component we never change sides of an edge, and whenever we enter a vertex we exit at the next edge on the same side; eventually we arrive at a barrier wedge, by proposition \ref{prop:tape_graph_quadrangulation}. We take a corresponding arc $\eta_+$ in $S$ which starts on $c_+$, runs into the tube corresponding to $e_d$, and then runs around tubes and sheets of $S$, never crossing an $\alpha$ curve, and whenever entering a sheet exiting at the next tube on the same side; eventually arriving at $\partial S$. (Note $\eta_+$ will intersect $\beta^\Sigma$ curves.)  We now isotope $c_+$ along $\eta_+$ by a finger move to obtain a polygonal curve $A_+ \cup B_+$, as shown in figure \ref{fig:annulus_cut_decomposition}. The arc $A_+$ is a short boundary-parallel arc; $B_+$ consists of most of $c_+$, and two arcs parallel to $\eta_+$. So $A_+$ avoids all $\beta^\Sigma$ curves, and $B_+$ avoids all $\alpha$ curves.

\begin{figure}
\begin{center}
\def\svgwidth{450pt}
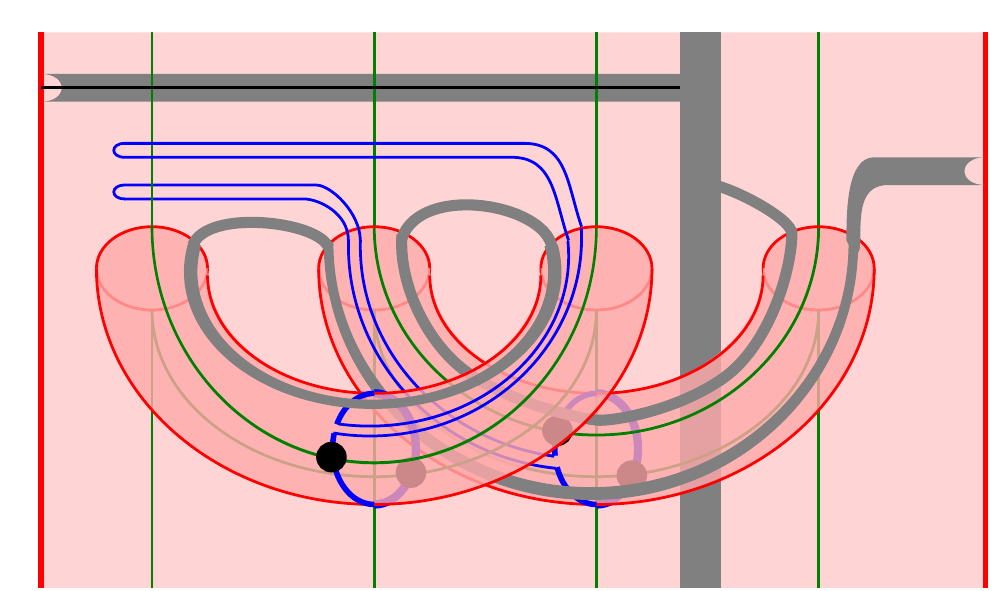
\caption{A Heegaard diagram $(S, \alpha, \beta^\Sigma, P)$ adapted to $U$; this is the same example as figure \ref{fig:make_admissible} above. The $\alpha$ curves are green, $\beta^\Sigma$ curves in blue, and $P$ in grey. The boundary curves $c_\pm$ of the annulus $\mathcal{A}$ are isotoped along curves $\eta_\pm$ to obtain polygonal boundaries $A_- \cup B_-$ and $A_+ \cup B_+$.}
\label{fig:annulus_cut_decomposition}
\end{center}
\end{figure}

Now $A_+ \cup B_+$ and $A_- \cup B_-$ form the polygonal boundary of an annulus $P \subset S$. We check that $(S, \alpha, \beta^\Sigma, P)$ is a good Heegaard diagram adapted to $U$, with $A = A_- \cup A_+$ and $B = B_- \cup B_+$. Thickening to $S \times [0,1]$ and surgering along $\alpha \times \{0\}$ and $\beta^\Sigma \times \{1\}$ curves, we see that $(\mathcal{A} \times \{1/2\}) \cup (c_+ \times [0, 1/2]) \cup (c_- \times [1/2, 1])$ is isotopic to the original annulus $a \times S^1$. Isotoping $c_+$ along $\eta_+$ corresponds to isotoping $v_- \times S^1$ around to intersect $R_+$, and isotoping $c_-$ along $\eta_-$ corresponds to isotoping $v_+ \times S^1$ around to intersect $R_-$. The resulting $(P \times \{1/2\}) \cup (A \times [1/2, 1]) \cup (B \times [0, 1/2])$ is isotopic to the decomposing surface $U$, so $(S, \alpha, \beta, P)$ is adapted to $U$.

We now decompose $(S, \alpha, \beta^\Sigma, P)$ along $P$, as per Juh\'{a}sz' construction. Recall we remove $P$ from $S$ and glue in two copies $P_A, P_B$ along $A$ and $B$ respectively, with $\alpha$ curves on $P_A$ and $\beta$ curves on $P_B$. But the original annulus $\mathcal{A}$ is disjoint from all $\alpha$ and $\beta^\Sigma$ curves; the ``finger" added along $\eta_-$ only intersects $\alpha$ curves; and the ``finger" added along $\eta_+$ only intersects $\beta^\Sigma$ curves. So the result is simply to cut out $\mathcal{A}$, and the decomposed Heegaard diagram is $(S \backslash \mathcal{A}, \alpha, \beta^\Sigma)$.

Indeed, this $(S \backslash \mathcal{A}, \alpha, \beta^\Sigma)$ is a sutured Heegaard diagram for $(\Sigma',V',F')$. The spine $G_{Q'}^+$ is obtained from $G_Q^+$ by cutting the blown-up vertex $v$ at the wedge $w_{d-1}$; and the Heegaard blocks for $(\Sigma', V', F')$ assemble to give a Heegaard surface $S' \cong S \backslash \mathcal{A}$ for $(\Sigma' \times S^1, F' \times S^1)$. The ``unperturbed" curves $\alpha_i$ and $\beta_i$ on $S$ in the Heegaard decomposition of $(\Sigma \times S^1, F \times S^1)$ correspond precisely to the ``unperturbed" curves $\alpha'_i$ and $\beta'_i$ on $S'$ for $(\Sigma' \times S^1, F' \times S^1)$; that is, $(S', \alpha', \beta') \cong (S \backslash \mathcal{A}, \alpha, \beta)$.

However, there is a difference. Since an annulus is removed from $S$ to obtain $S'$, there are more boundary-adjacent regions on $(S', \alpha', \beta')$ than on $(S, \alpha, \beta)$, so less need for isotoping $\beta'$-curves along $\zeta^\sigma_i$ arcs. Nonetheless for a curve $\beta'_i$ on $S'$ that has all adjacent components of $S' \backslash \left( \bigcup \alpha' \cup \bigcup \beta' \right)$ on a side $\sigma'$ internal, we can take the arc $\zeta^{\sigma'}_i$ connecting $\beta'_i$ to a boundary-adjacent component on side $\sigma'$ to be given by $\zeta^\sigma_i$ on $S' \cong S \backslash \mathcal{A}$: the same construction of ``always turn left" applies.

Thus, the admissible Heegaard decomposition $(S', \alpha', \beta^{\Sigma'})$ for $(\Sigma' \times S^1, F' \times S^1)$ is obtained from the admissible Heegaard decomposition $(S, \alpha, \beta^\Sigma)$ for $(\Sigma \times S^1, F \times S^1)$ by cutting the annulus $\mathcal{A}$ out of $S$, and then isotoping some of the curves $\beta^\Sigma_i$ curves back along $\zeta^\sigma_i$ arcs to their original position (``retracting fingers").

Moreover, the hypotheses of corollary \ref{cor:Juhasz_decomp_adapted} are satisfied. The map $\iota_* : H_2(\Sigma' \times S^1) \To H_2( \Sigma \times S^1)$ induced by inclusion is isomorphic to the map $H_1( G_{Q'}^+ ) \To H_1( G_Q^+ )$ induced by inclusion. As graphs are homotopy equivalent to wedges of circles, $\iota_*$ is injective with image a direct summand. The annular polygon $P$ contains no points of $\alpha \cap \beta^\Sigma$. As $H_1(\Sigma) \cong H_2(\Sigma \times S^1)$, we conclude the following.
\begin{prop}
\label{prop:decomposition_computation}
Tensoring with $\Z[H_1(\Sigma)]$ over $\iota_*: \Z[H_1(\Sigma')] \To \Z[H_1(\Sigma)]$ gives an isomorphism
\[
\Z[H_1(\Sigma)] \otimes_{\Z[H_1(\Sigma')]} SFH( \Sigma' \times S^1, F' \times S^1; \Z[H_1(\Sigma')]) \cong SFH(\Sigma \times S^1, F \times S^1; \Z[H_1(\Sigma)]).
\]
\qed
\end{prop}

Successively decomposing in this way, we finally arrive at a quadrangulation $Q$ of a $(\Sigma,V)$ where every vertex of $G_Q^+$ has degree $1$. In this case, $G_Q^+$ is a set of disconnected edges, $(\Sigma,V)$ is a set of disconnected squares, and $(S, \alpha, \beta)$ is a disconnected set of Heegaard blocks. Moreover, $S \backslash \left( \bigcup \alpha \cup \bigcup \beta \right)$ has no internal components so $(S, \alpha, \beta^\Sigma) = (S, \alpha, \beta)$ is admissible. There are no periodic domains and no continuous Whitney discs, let alone any holomorphic ones, so $\partial = 0$.

When $(\Sigma,V)$ is a single square we have a single $\alpha$ and $\beta$ curve intersecting at two points, which we call $|\0\rangle$ and $|\1\rangle$, with distinct spin-c structures. As $\partial = 0$ and  $\Z[H_1(\Sigma)] \cong \Z$ we obtain $SFH(\Sigma \times S^1, F \times S^1; \Z[H_1(\Sigma)]) \cong \Z|\0\rangle \oplus \Z|\1\rangle$. Similarly, when $(\Sigma, V, F)$ consists of $I$ disconnected squares we have $SFH(\Sigma \times S^1, F \times S^1) \cong \left( \Z|\0\rangle \oplus \Z|\1\rangle \right)^{\otimes I}$.

Now for a general occupied surface $(\Sigma, V, F)$, we have seen how to successively decompose $(\Sigma \times S^1, F \times S^1)$, and successively obtain $SFH(\Sigma \times S^1, F \times S^1)$ as isomorphic to the $SFH$ of a simpler $(\Sigma' \times S^1, F' \times S^1)$, after tensoring with the appropriate coefficient ring; moreover we simplify the Heegaard decomposition at each stage by removing an annulus from the Heegaard surface and retracting some fingers from $\beta$-curves. (In particular, retractions of $\beta$-curves do not affect homology; we have ``pair annihilation'' of intersection points \cite{OS04Closed}.) We thus have the following.

\begin{thm}
\label{thm:twisted_SFH_computation}
\[
SFH(\Sigma \times S^1, F \times S^1; \Z[H_1(\Sigma)]) \cong \Z[H_1(\Sigma)] \otimes_\Z \left( \Z\0 \oplus \Z\1 \right)^{\otimes I(\Sigma, V)} \cong \left( \Z[H_1(\Sigma)] \; \0 \oplus \Z[H_1(\Sigma)] \; \1 \right)^{\otimes I(\Sigma, V)}.
\]
Moreover, each tensor factor corresponds to a square of the quadrangulation, and the generators $|\0\rangle, |\1\rangle$ correspond to the two intersection points in a Heegaard block.
\qed
\end{thm}

\subsection{A twisted SQFT from SFH}
\label{sec:twisted_SQFT_from_SFH}

We can now prove a theorem for twisted coefficients directly analogous to theorem 1.2 of \cite{Me12_itsy_bitsy}.
\begin{thm}
The sutured Floer homology of 3-manifolds $(\Sigma \times S^1, V \times S^1)$, with twisted coefficients, contact elements, and maps on SFH induced by inclusions of surfaces, form a twisted sutured quadrangulated field theory.
\end{thm}

\begin{proof}
For occupied $(\Sigma,V,F)$, we let $\V(\Sigma,V)$ be $SFH(-\Sigma \times S^1, -F \times S^1; \Z[H_1(\Sigma)]) \cong SFH(-\Sigma \times S^1, -V \times S^1; \Z[H_1(\Sigma)])$ with twisted coefficients, which is a $\Z[H_1(\Sigma)]$-module. The twisted coefficient ring is $\Z[H_2(\Sigma \times S^1)] \cong \Z[H_1(\Sigma)]$, so $\V(\Sigma,V)$ is a $\Z[H_1(\Sigma)]$-module.

To sutures $\Gamma$ on $(\Sigma,V)$, we assign $c(\Gamma) = c(\xi_\Gamma)$, the contact invariant of the contact structure $\xi_\Gamma$ corresponding to $\Gamma$. This is a $\Z[H_1(\Sigma)]^\times$-orbit in $\V(\Sigma,V)$.

Given a decorated morphism $(\phi, \Gamma_c): (\Sigma,V) \To (\Sigma',V')$, as in part (v) of the construction in \cite[section 8.2, part (v)]{Me12_itsy_bitsy}, we can decorated-isotope the morphism to $(\phi^*, \Gamma_c^*)$ where $\phi^*$ is an embedding into $\Int \Sigma'$. This corresponds to an inclusion of $(\Sigma \times S^1, V \times S^1)$ into the interior of $(\Sigma' \times S^1, V \times S^1)$, with a contact structure $\xi^*$ on the complement corresponding to by $\Gamma_c^*$. We then use the TQFT property of twisted SFH (theorem \ref{thm:twisted_SFH_TQFT_map}) to define $\mathcal{D}_{\phi,\Gamma_c} = \Phi_{\xi^*} \; : \; \V(\Sigma,V) \To \V(\Sigma',V')$, which is a graded module homomorphism over $\phi_*$. Moreover, if $(\phi, \Gamma_c)$ and $(\phi', \Gamma'_c)$ are decorated-isotopic, then the embeddings-into-interiors obtained from them $(\phi^*, \Gamma_c^*)$, $(\phi'^*, \Gamma'^*_c)$ are isotopic (not just decorated-isotopic; they are isotopic without singularities at boundaries). Hence the complementary contact structures $\xi^*, \xi'^*$ are isotopic and $\Phi_{\xi^*} = \Phi_{\xi'^*}$, up to units. So $\D_{\phi,\Gamma_c}$ only depends on the decorated-isotopy class of $(\phi, \Gamma_c)$. 

In \cite[prop. 6.2]{HKM08} Honda--Kazez--Mati\'{c} show, for $\Z$ coefficients, the maps $\Phi_\xi$ on SFH respect compositions, up to units. Similarly, in \cite[thm. 6.1]{HKM08} it is shown for $\Z$ coefficients that, in our language, an identity morphism gives $\Phi_\xi$ the identity up to units. Since the $\Phi_\xi$ are defined, over $\Z$ or twisted coefficients, at the chain level simply by tensoring with a contact class (see also \cite[proof of thm. 12]{Ghiggini_Honda08}), these proofs extend without difficulty to twisted coefficients. Thus $\D$ is a functor up to units, satisfying condition (i).

It follows immediately from the computations above, that a quadrangulation $(\Sigma^\square_i, V^\square_i)$ gives a tensor decomposition of the desired type. For a square $(\Sigma^\square, V^\square)$ we have $\V(\Sigma^\square, V^\square) = SFH(-\Sigma^\square \times S^1, -V^\square \times S^1) \cong \Z|\0\rangle \oplus \Z|\1\rangle$, and theorem \ref{thm:twisted_SFH_computation} immediately gives
\[
\V(\Sigma,V) \cong \Z[H_1(\Sigma)] \otimes \bigotimes \V(\Sigma_i^\square, V_i^\square).
\]
Note that an occupied vacuum $(\Sigma^\emptyset, V^\emptyset)$ corresponds to $(\Sigma^\emptyset \times S^1, V^\emptyset \times S^1)$, which has a Heegaard decomposition without $\alpha$ or $\beta$ curves, corresponding to its null quadrangulation, and $\V(\Sigma^\square,V^\square) \cong \Z$.

Suppose $\Gamma$ is a basic set of sutures, restricting to $\Gamma_i$ on $(\Sigma^\square_i, V^\square_i)$. Start with the squares disconnected and the sutures $\Gamma_i$ on them; we clearly have $\V(\sqcup(\Sigma^\square_i, V^\square_i)) \cong \bigotimes_i \V(\Sigma^\square_i, V^\square_i)$ and $c(\sqcup \Gamma_i) = \bigotimes_i c(\Gamma_i)$. Now we consider successively gluing these squares together into $(\Sigma,V)$. Proposition \ref{prop:decomposition_computation} then says that at each stage, the effect on SFH is simply to tensor with the appropriate coefficient ring; so after gluing all the squares together, we have $c(\Gamma) = \Z[H_1(\Sigma)]^\times \otimes \bigotimes_i c(\Gamma_i)$. So condition (ii) is satisfied.

As noted in theorem \ref{thm:twisted_SFH_TQFT_map}, the maps $\Phi_\xi$ respect suture elements; so $\mathcal{D}_{\phi,\Gamma_c} c(\Gamma) \subseteq c(\Gamma \cup \Gamma_c)$, satisfying condition (iii).

It is explained in \cite{Me09Paper, HKM08} how the spin-c grading of SFH of manifolds $(\Sigma \times S^1, V \times S^1)$ corresponds to Euler class of contact structures; all elements in $c(\Gamma)$ have grading given by $e(\Gamma)$.

Finally, consider a square $(\Sigma^\square, V^\square)$. We have $\V(\Sigma^\square,V^\square) \cong \Z \oplus \Z$, with generators of the two summands corresponding to the two intersection points $\alpha \cap \beta$ in a Heegaard block, which have distinct spin-c structures. On the other hand, $c(\Gamma_+)$ has grading $1$ and $c(\Gamma_-)$ has grading $-1$. In fact \cite[sec. 7.2]{HKM08} $c(\Gamma_+)$ generates the $1$-graded summand and $c(\Gamma_-)$ generates the $0$-graded summand. Letting $c(\Gamma_+) = \{ \pm \1 \}$ and $c(\Gamma_-) = \{ \pm \0 \}$ then, $\0, \1$ form a free basis for $\V(\Sigma^\square, V^\square)$ of the desired gradings. Thus conditions (iv) and (v) hold.
\end{proof}

\begin{cor}
A twisted SQFT exists.
\qed
\end{cor}

\addcontentsline{toc}{section}{References}

\small

\bibliography{danbib}
\bibliographystyle{amsplain}

\end{document}